\documentclass[a4paper,10pt]{book}

\usepackage{amsmath}
\usepackage{amssymb}

\newcount\quantno
\everydisplay{\quantno=0}\everycr{\quantno=0}
\newcommand\quant{\advance\quantno by1
                      \ifnum\quantno=1\qquad\else\quad\fi\forall }

\newtheorem{defi}{Definition}[chapter]
\newtheorem{teo}[defi]{Theorem}
\newtheorem{prop}[defi]{Proposition}
\newcommand{\quadro}{\hfill\vrule height .9ex width .8ex depth -.1ex}

\newenvironment{proof}{\bf{Proof.}\mdseries}{\quadro \bigskip \\}
\newtheorem{coro}[defi]{Corollary}
\newtheorem{lem}[defi]{Lemma}
\newtheorem{ese}[defi]{Example}
\newtheorem{ossn}[defi]{Remark}
\newenvironment{teo*}{\itshape}{\\}
\newenvironment{intro*}{\footnotesize}{\\}

\newcommand{\NN}{\mathbb{N}}
\newcommand{\ZZ}{\mathbb{Z}}
\newcommand{\RR}{\mathbb{R}}
\newcommand{\CC}{\mathbb{C}}
\newcommand{\FF}{\mathbb{F}}
\newcommand{\HH}{\mathbb{H}}
\newcommand{\al}{\alpha}     
\newcommand{\be}{\beta}

\newcommand{\Ht}{H-type }
\newcommand{\CZ}{Calder\'on--Zygmund }
\newcommand{\DR}{Damek--Ricci }
\newcommand{\MH}{Mihlin--H\"ormander }
\newcommand{\Qka}{Q_{\alpha}^k}

\newcommand{\nka}{n_{\alpha}^k}

\newcommand{\n}{\mathfrak{n}}
\newcommand{\ag}{\mathfrak{a}}
\newcommand{\kg}{\mathfrak{k}}
\newcommand{\p}{\mathfrak{p}}
\newcommand{\g}{\mathfrak{g}}
\newcommand{\vg}{\mathfrak{v}}
\newcommand{\zg}{\mathfrak{z}}
\newcommand{\s}{\mathfrak{s}}
\newcommand{\LB}{\mathcal{L}} %operatore di Laplace-Beltrami
\newcommand{\D}{\Delta} %Laplaciano invariante sinistro
\newcommand{\LQ}{\mathcal{L}_Q} %operatore di Laplace Beltrami traslato
\newcommand{\kD}{k_{M(\Delta)}}
\newcommand{\kQ}{k_{M(\mathcal{L}_Q)}}
\newcommand{\du}{\delta^{1/2}} %funzione modulare alla 1/2
\newcommand{\dum}{\delta^{-1/2}} %funzione modulare alla -1/2
\newcommand{\MD}{M(\Delta)}
\newcommand{\MQ}{M(\mathcal{L} _Q)}
\newcommand{\arch}{\mathrm{arc\,cosh}}  %arcocoseno iperbolico

\newcommand{\Bp}{BMO_p}

\newcommand{\di}{\,{\rm{d}}}
\newcommand{\dir}{\,{\rm{d}}\rho}  %misura destra dirho
\newcommand{\dil}{\,{\rm{d}}\lambda}  %misura sinistra dila

\newcommand\lu[1]{L^1(#1)}
\newcommand\lorentz[3]{L^{#1,#2}(#3)}
\newcommand{\nondivisible}{nondivisible } %insiemi grandi al limite
\newcommand{\rd}{\eta}  %raggio delle palle diadiche in N
\newcommand{\Rnd}{\sigma_{N,\,\beta}}  %raggio degli insiemi amm grandi nondivisibili
\newcommand{\Rndp}{\sigma_{N',N'',\,\beta}}  %raggio degli insiemi amm grandi nondivisibili nel prodotto
\newcommand{\tF}{\tilde {F}}
\newcommand{\tB}{\tilde {B}}
\newcommand{\tR}{\tilde {R}}
\newcommand{\tE}{\tilde {E}}
\newcommand{\tC}{\tilde {C}}
\newcommand{\nep}{{\rm{e}}} %numero di Nepero
\newcommand{\htR}{h_t^{\RR}} %nucleo del calore su R
\newcommand{\chQ}{{\rm{Ch_{Q}}}}
\newcommand{\jl}{j_{\ell}}
\newcommand\lp[1]{L^p(#1)}
\newcommand\ld[1]{L^2(#1)}

\newcommand\opnorm{\vert\!\vert\!\vert}

\begin{document}

%\maketitle

\titlepage
{\center{

\vspace{16cm}
\Huge{Analysis on harmonic extensions\\
of \Ht groups}\\
\vspace{4cm}
\Large{Maria Vallarino}\\
\vspace{0,3cm}
\Large{Advisor:  Prof. Stefano Meda}\\
\vspace{6cm}
\Large{Universit\'a degli Studi di Milano Bicocca}\\
\vspace{0,3cm}
\Large{Scuola di Dottorato di Scienze}\\
\vspace{0,3cm}
\Large{Dottorato in Matematica Pura e Applicata}

}}

\chapter*{Acknowledgements}
I would like to thank my advisor Stefano Meda for the precious help and support which he gave me during these years.

I thank also Giancarlo Mauceri who introduced me to the fascinating world of Harmonic Analysis.

I thank Detlef M\"uller for inviting me to the Mathematisches Seminar of Christian-Albrechts-Universit\"at in Kiel were we had interesting discussions about my thesis.

\tableofcontents
\addcontentsline{toc}{chapter}{Introduction}

\chapter*{Introduction}

It became clear through the years that the fundamental 
real variable ideas behind the theory of maximal operators
of Hardy--Littlewood type and of certain singular integral operators,
which we shall refer to as \emph{classical Calder\'on--Zygmund}
theory, depend only on a few properties of the Euclidean space and the Lebesgue measure.
After a lot of research activity, 
spaces of homogeneous type (in the sense of Coifman and Weiss) emerged
as a reasonably general setting where the main ideas and some of the 
results of the classical theory could be, and indeed were, carried over
(see \cite{CW,S} and the references therein).

Disregarding technicalities, a space of homogeneous type 
is a measured metric space $(M,\mu,\rho)$ in which balls satisfy the so-called 
doubling property, i.e., there exists a constant $C$ such that 
$$
\mu\bigl(B(x,2r)\bigr)
\leq C \, \mu\bigl(B(x,r)\bigr)
\qquad \forall x \in M \quad \forall r \in \RR^+.
$$
This property is key in establishing Vitali type covering lemmata and 
the Calder\'on--Zygmund decomposition of integrable functions,
which are fundamental steps to proving weak type $1$ estimates
for some important operators, including
maximal operators of Hardy--Littlewood type and singular
integral operators.

A natural, but challenging problem, is whether a reasonable Calder\'on--Zygmund
theory, with emphasis on weak type $1$ estimates 
for maximal and singular integral operators, may be developed on 
measured metric spaces without the doubling property.
Recall that in the classical setting, singular integral operators
arise, for instance, as inverse Fourier transforms of Fourier
multipliers satisfying Mihlin--H\"ormander condition.  
Since no Fourier transform is available on a generic measured
metric space, it is part of the problem to envisage an interesting
class of singular integral operators to investigate.  Often, albeit not
always, such operators arise as kernels of spectral multipliers of some Laplacian.

Despite the efforts of
many mathematicians, only a few examples have been thoroughly investigated
and understood so far.  
These include rank one symmetric spaces, where singular integral
operators are related to spectral multipliers of
the Laplace--Beltrami operator \cite{An1, An2, A1, A2, CS, CGHM, CGM1, CGM2, CGM3, CGM4, GM, I1, I2, ST}, and the space  $(\RR^d,\mu,\rho)$,
where $\rho$ denotes the Euclidean distance and $\mu$ a possibly nondoubling
measure of polynomial growth, in which case singular integral operators
are associated to the Cauchy integral \cite{NTV, T1, T2, T3}. 
It is worth noticing that the new ideas and techniques employed in these
situations are quite different from each other.

Recently, Hebisch and Steger \cite{HS} realised that a \CZ theory
of singular integrals may be developed also on some groups of exponential 
growth, i.e., on solvable groups coming from the Iwasawa decomposition
of the groups $SO(d+1,1)$.  As an application, they proved
weak type $1$ estimates for spectral operators of a distinguished
Laplacian $\D$ on $S$ associated to spectral multipliers
satisfying Mihlin--H\"ormander type conditions (see the discussion below).

This thesis is concerned with real analysis 
on a class of noncompact solvable Lie groups known as
\emph{harmonic extensions of groups of Heisenberg type}, and their direct products,
which is much bigger than the class of groups studied by Hebisch and Steger.  

Harmonic extensions of Heisenberg type groups,
which have been introduced by E. Damek and 
F.~Ricci \cite{D1, D2, D3, DR1, DR2} and studied by 
M.~Cowling, A.~H.~Dooley, A.~Kor\'anyi and Ricci \cite{CDKR1,CDKR2},
are semidirect products of an Heisenberg type group, 
briefly an $H$-type group, $N$, 
and the multiplicative group $\RR^+$, which acts on $N$ by dilations with
eigenvalues $1/2$ and (possibly) $1$. 
Specifically, given an $H$-type algebra $\n=\vg\oplus \zg$,
let $N$ denote the connected and simply connected Lie group associated to $\n$. 
Let $S$ denote the one-dimensional extension of $N$ obtained by making 
$A=\mathbb{R}^+$ act on $N$ by homogeneous dilations.  
Let $H$ denote a vector in $\ag$ acting on
$\n$ with eigenvalues $1/2$ and (possibly) $1$; we extend the inner
product on $\n$ to the algebra $\s=\n\oplus\ag$, by requiring $\n$
and $\ag$ to be orthogonal and $H$ to be unitary. 

We denote by $d$ the left invariant distance on $S$ associated to the
Riemannian metric on $S$ which agrees with the inner product on $\s$ at the identity. 
The Riemannian manifold $(S,d)$ is usually referred to as \emph{Damek--Ricci space}.
Damek--Ricci spaces are harmonic manifolds and
include the class of symmetric spaces of noncompact type and real rank one properly. 
Those which are not symmetric provide counterexamples to the Lichnerowicz conjecture \cite{DR1}.
Harmonic analysis on these spaces has been the object of many investigations \cite{ADY, A1, DR2, D, R}.

Note that $S$ is nonunimodular. 
Let  $\lambda$ and $\rho$ denote the left and right Haar measures on $S$ respectively.
Then
$$
\lambda(B_r)= \rho(B_r)\asymp 
\left\{ 
\begin{array}{ll}
r^{n}   & \qquad \forall r\in (0,1)  \\
\nep^{Q\,r} & \qquad \forall r\in [1,\infty), \\
\end{array}
\right.
$$
where $B_r$ denotes the ball with centre the identity and radius $r$, $n$ is the dimension of $S$
and $Q$ is a constant determined by the structure of the algebra $\s$.
Thus $S$ is a group of \emph{exponential growth}:
hence both $\lambda$ and $\rho$ are nondoubling measures, which, in view of
the discussion above, makes $(S,\lambda,d)$ and $(S,\rho,d)$ interesting settings 
where to investigate maximal and singular integral operators.  
We  shall now describe some important operator on $S$ and explain why it is
natural to study some of them with respect
to $\lambda$ and some other with respect to $\rho$. 

Given a vector $E$ in the Lie algebra $\s$ we denote by
$\tilde{E}$ and $\dot{E} $ the right invariant and the left invariant
vector fields which agree with $E$ at the identity.
%defined by
%$$
%\tilde{E}f(x)
%= \frac{d}{d\tau }\Big\lvert_{\tau =0}f\big({\rm{exp}}(\tau\,E)\,x\big)
%\qquad\hbox{and}\qquad
%\dot{E} f(x)
%= \frac{d}{d\tau }
%    \Big\lvert_{\tau =0}f\big(x\,{\rm{exp}}(\tau\,E)\big)
%$$
%on smooth functions of compact support, respectively.
Now, let $E_0,...,E_{n-1}$ denote an orthonormal basis of the algebra $\s$ 
adapted to the (orthogonal) decomposition $\s=\vg\oplus \zg\oplus \RR$. 
The following operators on $S$ have been the object of investigation:
\begin{itemize}
\item[(i)] the Laplace--Beltrami operator $\LB$ associated to the Riemannian metric $d$:
$\LB$ is selfadjoint on $L^2(\lambda)$ and its spectrum is the half line $[Q^2/4,\infty)$;
\item[(ii)] the right invariant operator $\Delta_r=-\sum_{i=0}^{n-1}\tilde{E}_i^{2}$. 
A.~Hulanicki \cite{Hu} proved that the operator $\Delta_r$ defined on 
$C_c^{\infty}(S)$ is essentially selfadjoint on $L^2(\lambda)$ and its spectrum is $[0,\infty)$; 
\item[(iii)] the left invariant operator $\Delta_{\ell}=-\sum_{i=0}^{n-1}\dot{E}_i^{2}$. 
The operator $\Delta_{\ell}$ defined on $C_c^{\infty}(S)$ is essentially selfadjoint on $L^2(\rho)$ and its spectrum is $[0,\infty)$. 
\end{itemize}
It may be worth observing that all these operators are, so to speak, relatives.  Indeed,
let $V$ denote the element 
$$
- \sum_{j=0}^{n-1} E_j^2
$$
of the enveloping algebra of $S$.  We denote by $\pi_{\ell}$ and $\pi_r$
the left and the right regular representations of $S$ on $L^2(\lambda)$,
and consider the operators $\pi_{\ell}(V)$ and $\pi_r(V)$. 
Let $\LQ$ denote the shifted operator $\LB-{Q^2}/{4}$.
It is straightforward to check that $\pi_{\ell}(V) = \D_r$,
and that $\pi_{r}(V) = \LQ$.

Considerable effort has been produced to study the so-called 
$L^p$ functional calculus for either operators $\LB$, $\D_\ell$ and $\D_r$.  
To illustrate the results in the literature concerning these operators, 
we need some more notation and terminology.
Let $(X,\mu)$ be a measure space and $A$ a linear, nonnegative, (possibly
unbounded) selfadjoint operator on $L^2(\mu)$. 
Let $\{E(\lambda)\}$ denote the spectral resolution of the identity 
for which $A = \int_0^{\infty}\lambda\di E(\lambda)$. 
By the spectral theorem, if $M$ is a bounded Borel measurable function on 
$[0,\infty)$, then the operator $M(A)$ defined by
$$
M(A)=\int_0^{\infty}M(\lambda)\di E(\lambda)   
$$
is bounded on $L^2(\mu)$.
An interesting and intensely investigated
problem is to find sufficient conditions on $M$ such 
that the operator $M(A)$ extends
either to bounded operator on $L^p(\mu)$, for some $p$ different from $2$, 
%i.e.
%$$\|M(A)f\|_{L^p(\mu)}\leq C\,\|f\|_{L^p(\mu)}\qquad \forall f\in L^p(\mu)\,,$$
or satisfies a weak type $1$ estimate, i.e.,
$$
\mu\big(\{ x\in X:~|M(A)f(x)|>t\}\big)
\leq C\,\frac{\|f\|_{L^1(\mu)}}{t}\qquad\forall t>0\quad \forall f\in L^1(\mu)\,.
$$
If $M(A)$ extends to a bounded operator on $L^p(\mu)$,
we say that $M$ is an $L^p(\mu)$ \emph{spectral multiplier} of $A$.

We say that a complex valued function $M$ on $[0,\infty)$ satisfies a 
\emph{Mihlin--H\"ormander condition} of order $\alpha$ if
$$
\sup_{\lambda >0} |\lambda^j\,D^j M(\lambda)| \leq C
\qquad \mathrm{for~} j=0,\ldots,\alpha\,.
$$
An operator $A$ is said to admit an $L^p$ \emph{Mihlin--H\"ormander type functional calculus} 
if every function which satisfies a Mihlin--H\"ormander condition of suitable order is 
a $L^p$-multiplier for $A$.

By contrast, we say that an operator $A$ admits an $L^p$ \emph{holomorphic functional calculus}
if every $L^p$-multiplier of $A$ extends to an holomorphic function
on some neighbourhood of its $L^2$ spectrum.

Going back to operators on $S$, we observe preliminarly that
$(\dot{E}_i f^{\lor})^{\lor}=\tilde{E_i}f$
for every smooth function $f$, where $f^{\lor}(x)=f(x^{-1})$. Thus 
$$
(\D_{\ell} f^{\lor})^{\lor}=\D_r f.
$$  
Since $\lor$ is an isometry between $L^p(\rho)$ and $L^p(\lambda)$, 
a function $M$ is a $L^p(\rho)$ multiplier for $\D_{\ell}$ 
if and only if it is a $L^p(\lambda)$ multiplier for $\D_{r}$. 
Thus the multiplier problems for $\D_\ell$ on $L^p(\rho)$  
and of $\D_r$ on $L^p(\lambda)$ are equivalent.  It is a matter of taste
which Laplacian to choose:  Hebisch and Steger worked with $\D_r$,
whereas we shall work with $\D_\ell$, which we shall denote simply by $\D$ in Chapter 1-4.

%Many authors studied the case where the group $S$ is defined as follows. 
%Let $G$ be a connected noncompact semisimple Lie group of finite centre and $G=NAK$ be 
%its Iwasawa decomposition. 
%Let $S$ denote the group $NA$ which can be identified with the symmetric space $G/K$. 
%The group $S$ is endowed with a left invariant Riemannian metric. 
%This is an {\it{exponential growth}} group, i.e. the measure of a ball 
%$B_r$ centred at the identity of radius $r$ grows as
%$$
%\lambda(B_r)\asymp r^{n}\qquad r\to 0\qquad{\rm{and}}\qquad 
%\lambda(B_r)\asymp \nep^{Q\,r}\qquad r\to \infty\,.$$
%In this context J.~L.~Clerc and E.~M.~Stein \cite{CS} proved that 
%the operator $\mathcal L_b$ has a holomorphic functional calculus.

An interesting and perhaps surprising fact is that if $p\neq 2$, then $\LB$ possesses
a $L^p$ holomorphic functional calculus, whereas $\D_\ell$ admits a
$L^p$ functional calculus of \MH type. 
The result for $\LB$ was obtained by F. Astengo \cite{A2}, who extended
to Damek--Ricci spaces classical results of
various authors on symmetric spaces of the noncompact type
\cite{CS, CGHM}. 
Astengo's result has recently been improved by A. Ionescu \cite{I1, I2} in
the case where $S$ is symmetric.

Results concerning $\D_r$, equivalently $\D_\ell$, are much more recent.
On $NA$ groups coming from the Iwasawa decomposition of
a noncompact semisimple Lie group with finite centre 
and arbitrary rank 
Cowling, S.~Giulini, Hulanicki and G.~Mauceri \cite{CGHM} proved  
that the Laplacian $\D_r$ admits a nonholomorphic $L^p(\lambda)$ functional calculus. 
Specifically, they showed that if a function $M$ satisfies a Mihlin--H\"ormander 
condition of suitable order at infinity and 
belongs locally to a suitable Sobolev space, then $M(\D_r)$ is of weak type $1$ and it is 
bounded on $L^p(\lambda)$ for $1<p<\infty$. 
Note that, because of the assumption on $M$ near $0$, 
this is not a true Mihlin--H\"ormander type condition.   
An important consequence of this fact is that the kernel of $M(\D_r)$ is
integrable at infinity, while locally it behaves like a \CZ singular kernel.

F.~Astengo \cite{A2} extended the result of \cite{CGHM} to all Damek--Ricci spaces. 

Subsequently,  Hebisch and Steger \cite{HS} sharpened the results in 
\cite{A2} in the case of solvable groups coming from the Iwasawa decomposition of the groups $SO(d+1,1)$.
Since their work will be key for us, we describe some of their results in some detail.
Their main contribution is that they realised 
that a \CZ theory of singular integrals may be developed also on certain
measured metric spaces of exponential growth.  
A measured metric space $(X,\mu,d)$ is said to be 
a \emph{Calder\'on--Zygmund space}, with \CZ constant $\kappa_0,$ 
provided that there exists a family $\mathcal{R}$ of open subsets of $X$, which we call \CZ sets, satisfying the following property: 
for each $R$ in $\mathcal{R}$
there exists a positive number $r$ such that 
$R$ is contained in a ball $B$ of radius at most $\kappa_0 \, r$ and 
$\mu(R^*)\leq  \kappa_0\,\mu(R)~~$, where$~~R^*$ is a dilated of $R$ defined by $R^*=\{x\in X~:~d(x,R)< r\}$.  Furthermore, 
each integrable function $f$ may be decomposed as 
$g + \sum_{i=1}^\infty b_i,$
where the ``good'' function $g$ is bounded, and each ``bad'' function 
$b_i$ has vanishing integral and
is supported in a set $R_i$ in $\mathcal{R}$.

Suppose now that $\mathcal{T}$ is an operator on a \CZ space $(X,\mu,d)$, 
which is bounded on 
$\ld\mu$ and admits a locally integrable kernel $K$ off the diagonal that satisfies
the following H\"ormander type condition (see \cite{S} for the
Euclidean case)
\begin{equation} \label{f: integral condition}
\sup_R \sup_{y,\,z \in R} \int_{(R^*)^c} |K(x,y) - K(x,z)| \di\mu(x) < \infty,
\end{equation}
where the supremum is taken over all \CZ sets $R$ in $\mathcal R$.
Then $\mathcal{T}$ extends to a bounded operator
on $\lp\mu$, for $1<p<2$, and is of weak type $1$.

Hebisch and Steger \cite{HS} proved that solvable groups coming from the Iwasawa decomposition
of the groups $SO(d+1,1)$ are \CZ spaces.
As an application, they showed that if $M$ satisfies a Mihlin--H\"ormander type condition 
of suitable order, then $M(\D_r)$ satisfies a weak type $1$ estimate and 
is bounded on $L^p(\lambda)$ for all $p$ in $(1,\infty)$.
The kernel of $M(\D_r)$ is, unlike in \cite{CGHM} and \cite{A2},
no longer integrable at infinity, and 
the full strength of the aforementioned result is required to control this singularity.

Recently, the multiplier result of \cite{HS} 
has been re-obtained by D. M\"uller and C. Thiele
\cite{MT} via a different method which hinges on estimates of the wave propagator.
The author, in collaboration with M\"uller, has obtained
similar estimates for all Damek--Ricci spaces, thereby extending the 
results of \cite{MT}.  This result is not included in this thesis, and will
appear elsewhere. 

The main results of this thesis are the following:
\begin{itemize}
\item[(i)]
all Damek--Ricci spaces are \CZ spaces (see Thm \ref{CZd});   
\item[(ii)]
maximal operators of 
Hardy--Littlewood type associated to various families of sets on 
Damek--Ricci spaces satisfy weak type $1$ estimates (see Thm \ref{maxopweak});
\item[(iii)]
an analogue of the multiplier result of Hebisch and Steger for all
Damek--Ricci spaces and to their direct products holds (see Thm \ref{moltiplicatori} and Thm \ref{moltiplicatorip});
\item[(iv)]
we introduce an atomic Hardy space on \CZ spaces, identify
its topological dual with a suitable space of bounded mean oscillation, 
and complement Hebisch and Steger's result concerning
singular integral operators by showing that under condition (\ref{f: integral condition})
above the operator $\mathcal{T}$ extends to a bounded operator from $H^1$ to $L^1$ (see Thm \ref{TeolimH1}).   
\end{itemize}

We shall now briefly discuss the results (i)-(iv).

To prove Theorem \ref{CZd} we need to envisage on each Damek--Ricci
space $S$ a family of sets $\mathcal{R}$, referred to as \emph{family of
\CZ sets}, that satisfy the requirements of Definition \ref{CZdef}.  
Hebisch and Steger \cite{HS} defined an appropriate family of sets, referred to as \emph{family of admissible sets}, on solvable groups coming from the Iwasawa decomposition
of the groups $SO(d+1,1)$: it consists of ``rectangles'' of the form $Q\times I$,
where $Q$ is a dyadic cube of $\RR^d$ and $I$ is an interval in $\RR^+$,
satisfying a certain admissibility condition, which relates their sizes.  

It is tempting to extend this definition of admissible sets
to all Damek--Ricci spaces by strict analogy.  
Unfortunately, this does not work: indeed, roughly speaking, the right measure of a ``small rectangle'' $R$ and its dilated set $R^*$ are not comparable, while they are comparable for a ``big rectangle'' (see Section \ref{CZdecextH} for details). Thus we consider the family $\mathcal{R}$ of
admissible sets consisting of ``rectangles'' of the form $Q\times I$ much as before, but of big size and geodesic balls of small radius.
The proof that this family makes $(S,\rho,d)$ a Calder\'on--Zygmund space 
is quite involved and technically difficult and occupies Section \ref{CZdecextH}.

It is interesting to observe that the family of admissible sets of 
Hebisch and Steger made its appearance in the literature in a paper of S. Giulini
and P. Sj\"ogren \cite{GS}, where they proved a weak type $1$ estimate
for the associated maximal operator of Hardy--Littlewood type on the affine 
group of the real line.
It is therefore natural to investigate whether a similar estimate
holds for the maximal operator
of Hardy--Littlewood type associated to the family $\mathcal{R}$ of admissible
sets we have defined on a generic Damek--Ricci space.

To be specific, for each $x$ in $S$ let $\mathcal{R}(x)$ denote
the subcollection of $\mathcal{R}$ of 
all sets in $\mathcal{R}$ which contain $x$.  
We consider the left invariant Hardy--Littlewood type maximal operator 
$M^{\mathcal R}$, whose action on $f$ in $L^1_{\rm{loc}}(\rho)$ is
$$
M^{\mathcal R}f(x)
= \sup_{R\in\mathcal{R}(x)}\frac{1}{\rho(R)}\int_{R} |f|
\dir\qquad \forall x \in S.
$$ 
It may we worth remarking that the maximal operator $M^{\mathcal R}$  
appears in the \CZ decomposition of integrable functions on $S$. 
We shall prove (see Theorem \ref{maxopweak} below) that $M^{\mathcal R}$  
is of weak type $1$, and that a similar estimate holds for a 
maximal operator associated to a family related to $\mathcal{R}$. 
It is remarkable that the proof of the weak type $1$ inequality is standard, 
but it is based on a nonstandard and perhaps surprising
covering lemma of Vitali type (Lemma \ref{famnice} below).

%Let $\Gamma$ be the affine group of the real line and 
%$\mathcal F$ be the family of sets 
%$F_{r,\,\beta}=(-r^{\beta},r^{\beta})\times (1/r,r)$, 
%for $r>1$ and $\beta>1/2$. 
%We consider the left invariant maximal operator defined by
%$$
%Nf(x)
%= \sup_{r>1}\frac{1}{\rho(x\,F_{r,\,\beta})}
%\int_{x\,F_{r,\,\beta}}|f|\dir\qquad\forall f\in L^1_{\rm{loc}}(\Gamma)\,.
%$$
%Giulini and Sj\"ogren proved that it is of weak type $(1,1)$. 
%Note that this is a centred maximal operator, while $M^{\mathcal F}$ is noncentred.
%
%The natural generalization of these sets to harmonic extensions of $H$-type groups is given by sets $F_{r,\,\beta}=B_N(n,r^{\beta})\times (a_0/r,a_0\,r)$, where $n\in N$, $a_0\in\RR^+$ and  $r\geq \nep$ (we are looking for ``big'' sets). Let $\mathcal R^{\infty}$ be the family of such sets. In this thesis we prove that the associated maximal operator $M^{\mathcal R^{\infty}}$ is of weak type $(1,1)$. 
%

%To do it they defined a suitable family of {\it{admissible sets}} which are involved in the \CZ decomposition: they are products of dyadic sets in $\RR^d$ and intervals in $\RR^+$.

Next we discuss briefly the results in (iii).  
Let $\psi$ be a function in $C^{\infty}_c(\RR^+)$, 
supported in $[1/4,4]$, such that $\sum_{j\in\ZZ}\psi(2^{-j}\lambda)=1$ 
for all $\lambda\in\RR^+$. 
We define $\|M\|_{0,s_0}$ and $\|M\|_{\infty,s_{\infty}}$ thus:
\begin{align*}
\|M\|_{0,s}&=\sup_{t<1}\|M(t\cdot)\,\psi(\cdot)\|_{H^s(\RR)}\,,\\
\|M\|_{\infty,s}&=\sup_{t\geq 1}\|M(t\cdot)\,\psi(\cdot)\|_{H^s(\RR)}\,,
\end{align*}
where $H^s(\RR)$ denotes the $L^2$-Sobolev space of order $s$ on $\RR$. 
We say that a bounded Borel measurable function $M$ defined on $\RR^+$ satisfies a 
{\it{mixed \MH condition of order $(s_0,s_{\infty})$}} if 
$\|M\|_{0,s_0}<\infty$ and $\|M\|_{\infty,s_{\infty}}<\infty$.

Our multiplier theorem for the Laplacian $\D_{\ell}$ is the following.

\vspace{0,3cm}
{\bf{Theorem 1. }}{\it{Suppose that $s_0>3/2$ and $s_{\infty}>\max\left\{{3/2},{n}/{2}\right\}$, where $n$ denotes the 
dimension of $S$.  If $M$ is a bounded Borel measurable function on $\RR^+$ 
that satisfies a mixed \MH condition of order $(s_0,s_{\infty})$, then $M(\D_{\ell})$
is of weak type $1$ and bounded on $L^p(\rho)$, for all $p$ in $ (1,\infty)$.}}
\vspace{0,3cm}

The proof of Theorem 1 follows the same line as that of \cite[Thm 2.4]{HS}.
Since $(S,\rho,d)$ is a Calder\'on--Zygmund space, as we have shown 
in Theorem \ref{CZd}, we may apply the basic result of 
Hebisch and Steger concerning singular integral operators
on Calder\'on--Zygmund spaces.   Thus, we are led
to prove that (\ref{f: integral condition}) holds with the kernel 
of $M(\D_\ell)$ in place of $K$.  
This may be further reduced to showing that 
$$
\|{\nabla h_t}\|_{L^1(\rho)}
\leq C \, t^{-1/2} \qquad\forall  t \in \RR^+,
$$
where $h_t$ denotes the heat kernel associated to $\D_\ell$.
This estimate was obtained by Hebisch and Steger by a descent method,
which reduces the estimate to the case where $n$ is odd, and the
formulae for the inverse spherical transform easier. 
The descent method seems to be unapplicable in our generality.  
Therefore, we have to prove the gradient estimate above
by finding poitwise estimates for the inverse spherical Fourier transform 
of the heat kernel $h_t$, and then estimating its gradient.   
This is hard, at least in the case where the dimension of the centre $\zg$ of the algebra $\n$ 
is odd, and requires fine estimates of some integrals. 

An interesting and possibly very hard question which arises
naturally is the following:  to what extent the theory developed 
on Damek--Ricci spaces may be generalised to $AN$ groups coming
from the Iwasawa decomposition of a noncompact semisimple Lie 
group of arbitrary rank.  As a first attempt to grasp
the problem, we consider the simplest possible
model case, i.e., the product of two Damek--Ricci spaces,
and show that an analogue of Theorem~1 holds.
Computations on the group of upper triangular matrices in 
$SL(3,\RR)$ are at a preliminary stage, and we cannot say 
whetherthey will give the desired extension.

Finally we comment briefly on (iv).  A classical method
to obtain sharp boundedness results for multipliers is
to prove endpoint results which involve the Hardy space $H^1$ or
the space $BMO$ of bounded mean oscillation \cite{CW2,FeS,S}. 
For a spectacular application of this method, see, for instance, \cite{FS},
where sharp estimates for the regularity of solutions to 
the wave equation is obtained. 
It is natural to ask whether it is possible to develop an 
$H^1$--$BMO$ theory on Damek--Ricci spaces.  

Our results in this direction are of preliminary nature. 
For each $q$ in $(1,\infty)$ we define, on a 
Calder\'on--Zygmund space satisfying an additional condition of technical nature,
 an atomic Hardy space
$H^{1,\,q}$.  Atoms are functions supported on
admissible sets, with vanishing integral and satisfying
a certain size condition.
We also define a space of functions of bounded mean oscillation
$BMO_{p}$, $1<p<\infty$, and prove that the topological dual of
$H^{1,\,q}$ is isomorphic to $BMO_{q'}$,
where $q'$ denotes the index conjugate to $q$ (see Theorem \ref{duality} below).
Further, we show that a singular integral operator, whose kernel
$K$ satisfies (\ref{f: integral condition}), extends to a bounded
operator from $H^{1,\,q}$ to $L^1(\mu)$ for all $q$ in $(1,2)$
(see Theorem \ref{TeolimH1} below).
 
An important feature of the classical theory is that
all the spaces $H^{1,q}$, $q$ in $(1,\infty)$, are equivalent. 
Furthermore, they are equivalent to the space $H^{1,\infty}$,
which is defined in terms of $(1,\infty)$-atoms. 
A similar comment applies to $BMO_p$ spaces.
We extend this result to the case where $S$ is associated to solvable groups coming from the Iwasawa decomposition
of the groups $SO(d+1,1)$.  The proof is quite involved
and there seems no easy way to extend it to all Damek--Ricci spaces
(see Lemma \ref{coincidono} below). 
As a consequence of this result, we show that spectral multipliers
for $\D_\ell$ extend to bounded operators from $H^1$ to
$L^1(\rho)$ and from $L^\infty(\rho)$ to $BMO$.

The major drawback of the aforementioned results is that 
it is not clear how to prove that the intermediate space
between $H^1$ and $L^2$ is $L^p$, with $p$ in $(1,2)$,
and thus obtain $L^p$ results by interpolation between
$H^1$--$L^1$ and $L^2$ results.
This seems to be a difficult problem, which will be
the object of further research.

\vspace{1cm}

The thesis consists of four chapters. 

In Chapter 1 we introduce Heisenberg type groups and their harmonic extensions, and we recall the basic spherical analysis on these spaces. Then we introduce products of \DR spaces and recall spherical analysis results on these product spaces.
 
In Chapter 2 we study some left invariant Hardy--Littlewood 
type maximal operators on \DR spaces and on products of \DR spaces, focusing our attention on their weak type $1$ boundedness. 

In Chapter 3 we recall the general \CZ theory of Hebisch and Steger and develop a $H^1$-$BMO$ theory on \CZ spaces, proving a result which concerns the boundedness of sigular integral operators from $H^1$ to $L^1$. Then we show that \DR spaces and products of \DR spaces are \CZ spaces.
 
In Chapter 4 we study spectral multipliers for the left invariant 
Laplacian $\D_{\ell}$ and prove Theorem 1. Then we generalize the multiplier theorem to products of \DR spaces.

\chapter{Notation and preliminary results}
\begin{intro*}
In this chapter we recall the definition of $H$-type algebras: they are two-step nilpotent Lie algebras endowed with a suitable inner product. Then we introduce $H$-type groups, which are Lie groups whose Lie algebra is of $H$-type, and study their properties. In Section \ref{Hext} we define the harmonic extension of an $H$-type group and discuss some geometrical properties of this space. Then in Section \ref{sphericalanalysis} we summarize the basic spherical harmonic analysis on harmonic extensions of $H$-type groups: we recall the notions of spherical Fourier transform, spherical functions and Abel transform.

Finally in Section \ref{products} we introduce direct products of harmonic extensions of $H$-type groups and recall the basic spherical analysis on these product spaces.

For the details see \cite{ADY, A1, AD, CDKR1, CDKR2, D1, D2, D3, DR1, DR2, D, K}.
\end{intro*}
\section{H-type algebras}
Let $\n$ be a Lie algebra equipped  with an inner product $\langle\cdot,\cdot\rangle$ and denote by $|\cdot|$ the associated norm. Let $\vg$ and $\zg$ be complementary orthogonal subspaces of $\n$ such that $[\n,\zg ]=\{0\}$ and $[\n,\n]\subseteq \zg$. In particular $\n$ is two-step nilpotent, unless $\vg$ or $\zg$ is trivial, when $\n$ is abelian. Unless explicitly stated, we assume that $\n$ is nonabelian. 

We define the map $J:\zg\to End(\vg)$ by 
$$\langle J_ZX,Y\rangle\,=\,\langle Z,[X,Y]\rangle\qquad\forall X, Y\in \vg \quad\forall Z\in\zg\,.$$
\begin{defi}
The algebra $\n$ is said to be of Heisenberg type, briefly, an {\rm{$H$-type algebra}}, if the map $J$ satisfies one of the following equivalent conditions:
\begin{itemize}
\item[(i)]$|J_ZX|\,=\,|Z||X|$ for all $X$ in $\vg$ and $Z$ in $\zg\,;$
\item[(ii)]$J_Z^2=-|Z|^2\,Id_{\vg}$ for all $Z$ in $\zg$;
\item[(iii)]$\langle J_ZX,J_{Z'}X'\rangle+\langle J_ZX',J_{Z'}X\rangle=2\,\langle X,X'\rangle\,\langle Z,Z'\rangle $ for all $X,X'$ in $\vg$ and $Z,Z'$ in $\zg$;
\item[(iv)]$J_ZJ_{Z'}+J_{Z'}J_Z=-2\langle Z,Z'\rangle\,Id_{\vg}$ for all $Z,Z'$ in $\zg$.
\end{itemize} 
The connected and simply connected Lie group $N$ associated to $\n$ is called an {\rm{$H$-type group}}.
\end{defi}
We now give examples of $H$-type algebras. 
\begin{ese}
{\rm{Let $\FF$ denote one of the fields $\RR$, $\CC$ or $\HH$ and consider on $\FF$ the inner product $\langle z,z'\rangle={\rm{Re}}(z\,\bar{z'})$ for all $z,z'\in\FF$. We define $\n=\FF^k\times \FF^k\times \FF$ endowed with the inner product
$$\langle(x,y,z),(x',y',z') \rangle=\sum_{i=1}^k\langle x_i,x_i'\rangle+\sum_{i=1}^k\langle y_i,y_i'\rangle+\langle z,z'\rangle\,,$$
and the bracket
$$[(x,y,z),(x',y',z')]=\sum_{i=1}^k(x_i\,y_i'-x_i'\,y_i)\,,$$
for all $(x,y,z),(x',y',z')\in \n\,.$ We have that $\vg=\FF^k\times\FF^k$ and $\zg=\FF$ are two orthogonal subspaces such that $[\n,\zg ]=\{0\}$ and $[\n,\n]\subseteq \zg$. The map $J$ is defined by
$$J_z\big((x,y)\big)=(-z\,\bar{y},z\,\bar{x})\qquad \forall (x,y)\in\FF^k\times\FF^k\quad z\in\FF\,.$$
Thus $|J_z(x,y)|=|z|\,|(x,y)|$ forall $(x,y)$ in $\FF^k\times\FF^k$ and $z$ in $\FF$. Then $\n$ is an $H$-type algebra.}}
\end{ese}
\begin{ese}{\rm{ (see \cite{CDKR1, CDKR2}) Let $\g$ be a noncompact, semisimple Lie algebra of real rank one. Let $B$ be the Killing form defined by
$$B(X,Y)={\rm tr}({\rm ad}X\circ {\rm ad}Y)\qquad\forall X, Y\in\g\,,$$
and let $\theta$ be a Cartan involution of $\g$. We denote by $\kg$ and $\p$ the eigenspaces corresponding to the eigenvalues $+1$ and $-1$, i.e.
$$\kg=\{ X\in\g:~\theta X=X\}\quad \p= \{ X\in\g:~\theta X=-X\}\,.$$
Let $\ag$ be a maximal abelian subalgebra of $\p$. Since $\g$ has rank one, $\ag$ is one-dimensional; let $H$ be a vector which spans $\ag$. For each linear form $\alpha\in \ag ^*$, we define
$$\g_{\alpha}=\{X\in\g:~{\rm ad}H(X)=[H,X]=\alpha(H)X\}\,.$$
If $\g_{\alpha}\neq 0$, $\alpha$ is called a root of $\g$; $\alpha$ is called positive if $\alpha(H)>0$. There are at most two positive roots $\alpha$ and $2\alpha$ and the Lie algebra $\g$ has the following decomposition:
$$\g=\g_{-2\alpha}\oplus \g_{-\alpha}\oplus\g_0\oplus\g_{\alpha}\oplus\g_{2\alpha}\,.$$
Let $\n$ denote the subalgebra $\g_{\alpha}\oplus\g_{2\alpha}$ of $\g$. Since $[\g_{\alpha},\g_{\beta}]\subseteq\g_{\alpha+\beta}$ for any roots $\alpha$ and $\beta$, $\n$ is a nilpotent Lie algebra. The following decomposition of $\g$, which is called Iwasawa decomposition, holds:
$$\g=\n\oplus\ag\oplus\kg\,.$$ 
We define an inner product in $\n$ as
$$\langle X,Y\rangle\,=\,-\frac{1}{m_{\alpha}+4m_{2\alpha}}\,B(X,\theta Y)\qquad\forall X, Y\in \n\,,$$
where $m_{\alpha}$ and $m_{2\alpha}$ denote the dimension of $\g_{\alpha}$ and $\g_{2\alpha}$ respectively. The map $J$ is given by $J_Z X=[Z,\theta X]$ and 
$$|J_ZX|=|Z|\,|X| \qquad\forall Z\in\g _{2\alpha}\quad \forall X\in\g _{\alpha}\,.$$
Then $\n$ is an algebra of Heisenberg type. Moreover one can show that $\n$ satisfies an additional condition, known as ``$J^2$ condition'', which is defined below: for all $X$ in $\vg$ and all orthogonal $Z, Z'$  in $\zg$ there exists $Z''$ in $\zg$ such that $J_ZJ_{Z'}X=J_{Z''}X$.

The classification of $H$-type algebras satisfying the $J^2$ condition is described in \cite{CDKR1}. Cowling, Dooley, Kor\'anyi and Ricci showed that an $H$-type algebra satisfies the $J^2$ condition if and only if it appears in the Iwasawa decomposition of a noncompact semisimple Lie algebra of real rank one.}}
\end{ese}
\bigskip
Let $N$ be an $H$-type group. We identify $N$ with its Lie algebra $\n$ via the exponential map
\begin{align*}
\vg\times\zg &\to N\\
(X,Z)&\mapsto \exp(X+Z)\,.
\end{align*}
The product law in $N$ is
$$(X,Z)(X',Z')=\big(X+X',Z+Z'+({1}/{2})\,[X,X']\big)\qquad\forall X,\,X'\in \vg\quad\forall Z,\,Z'\in\zg\,.$$
The group $N$ is a two-step nilpotent, hence unimodular, group with Haar measure $\di X \di Z$.
We define the following dilations on $N$:
\begin{align*}
\delta_a(X,Z)&=(a^{1/2}\,X,a\,Z)\qquad\forall(X,Z)\in N \quad\forall a\in\RR^+\,.
\end{align*}
Set $Q={(m_{\vg}+2m_{\zg})}/{2}\,$, where $m_{\vg}$ and $m_{\zg}$ denote the dimensions of $\vg$ and $\zg$ respectively. For each measurable subset $E$ of $N$ we have that
$|\delta_a E|=a^Q\,|E|$, for all $a$ in $\RR^+$. The group $N$ is an homogeneous group with homogeneous norm
$$\mathcal N (X,Z)=\left(\frac{|X|^4}{16}+|Z|^2\right)^{1/4}\qquad\forall(X,Z)\in N\,.$$
Note that $\mathcal N\big(\delta_a(X,Z)\big)=a^{1/2}\mathcal N(X,Z)$. We denote by $d_N$ the homogeneous distance on $N$ which is defined by
$$d_N\big((X_0,Z_0),(X,Z) \big)=\mathcal N\big((X_0,Z_0)^{-1}(X,Z)\big)\qquad\forall (X_0,Z_0),\,(X,Z)\in N\,.$$
Given $(X_0,Z_0)$ in $N$ and $r>0$, the homogeneous ball centred at $(X_0,Z_0)$ of radius $r$ is 
$$B_N\big((X_0,Z_0),r\big)=\big\{(X,Z)\in N:~d_N\big((X_0,Z_0),(X,Z)\big)<r\big\}\,.$$
Obviously, $B_N\big((X_0,Z_0),r\big)=\delta_{r^2}B_N\big((X_0,Z_0),1\big)$ and its measure is $r^{2Q}|B_N(0_N,1)|$.

It is useful to emphasize  the relation between the homogeneous norm $\mathcal N$ and the norm $|\cdot|$ induced by the inner product on $N$.  
\begin{prop}\label{normaN}
There exist constants $c_1$ and $c_2$ such that if $(X,Z)\in N$ and $\mathcal N(X,Z)=r$, then
\begin{itemize}
\item [(i)]$|X|\leq c_2\,r~$ and $~|Z|\leq c_2\,r^2$;
\item [(ii)] either $~|X|\geq c_1\,r~$ or $~|Z|\geq c_1\,r^2$.
\end{itemize}
\end{prop}
\begin{proof}
Let $(X,Z)$ be in $N$ such that ${\mathcal{N}}(X,Z)=r$. Obviously $(X,Z)=\delta_{r^2}(X_0,Z_0)=(r\,X_0,r^2\,Z_0)$, for some $(X_0,Z_0)$ such that ${\mathcal N}(X_0,Z_0)=1$. Since in homogeneous groups the closed ball $\overline{B_N\big((0,0),1\big)}$ is compact \cite[Lemma 1.4]{FS}, there exist constants $c_1$ and $c_2$ such that
\begin{align*}
|X_0|&\leq c_2   & &{\rm and} & |Z_0|&\leq c_2;\\
{\rm either}~|X_0|&\geq c_1  & &{\rm or} & |Z_0|&\geq c_1\,.
\end{align*}
%\begin{eqnarray*}
%|X_0|\leq c_2~~&&{\rm and}~~~~~~~~~|Z_0|\leq c_2;\\
%{\rm either}~~ |X_0|\geq c_1~~&&{\rm or}~~~~~~~~~|Z_0|\geq c_1\,.
%\end{eqnarray*}
It follows that the norms of $X$ and $Z$ satisfy properties (i) and (ii):  
\begin{align*}
|X|=r|X_0|&\leq c_2\,r & &{\rm and}  & |Z|&=r^2|Z_0|\leq c_2\,r^2;\\
{\rm either}~ |X|=r|X_0|&\geq c_1\,r & &{\rm or}  & |Z|&=r^2|Z_0|\geq c_1\,r^2\,,
\end{align*}
as required.
%\begin{eqnarray*}
%|X|=r|X_0|\leq c_2r~~&&{\rm and}~~~~~~|Z|=r^2|Z_0|\leq c_2r^2;\\
%{\rm either}~~ |X|=r|X_0|\geq c_1r~~&&{\rm or}~~~~~~|Z|=r^2|Z_0|\geq c_1r^2\,.
%\end{eqnarray*}
\end{proof}
Next we recall an integration formula in polar coordinates on $N$ \cite[Proposition 1.15]{FS}.
\begin{prop}\label{integrationN}
There exists a Radon measure $\sigma$ on $\Sigma =\{(X',Z')\in N:~\mathcal N(X',Z')=1\}$ such that for all $f$ in $ L^1(N)$ 
\begin{align*}
\int_Nf(X,Z)\di X\di Z&=\int_0^{\infty}\int_{\Sigma}f\big(\delta_r(X',Z')\big)\,r^{Q-1}\di\sigma (X',Z')\di r\,.\\
%&=&\int_{\RR++}\int_{\Sigma}f(\delta_a(X',Z'))\,a^{2Q-1}\di\sigma (X',Z')\di a\,.
\end{align*}
%In particular if $f$ is such that $f(X,Z)=f_0(\mathcal N((X,Z)))$ for all $(X,Z%)\in N$ then
%$$\int_Nf(X,Z)dXdZ=C\int_0^{\infty}f_0(r^{1/2})dr\,.CONTROLLA!!!$$
\end{prop}

\section{Harmonic extensions of $H$-type groups}\label{harmonicextension}\label{Hext}
Let $\n$ be an $H$-type algebra and let $\ag$ be a $1$-dimensional Lie algebra with an inner product, spanned by the unit vector $H$. We denote by $\s$ the Lie algebra $\n\oplus\ag$ where the Lie bracket is determined by linearity and the requirement that
\begin{align*}
[H,X] &=\frac{1}{2}X\qquad \forall X\in\vg \\
[H,Z] &= Z\qquad\forall Z\in\zg\,.
\end{align*}
We extend the inner products on $\n$ and $\ag$ to $\s$ by requiring that $\n$ and $\ag$ be orthogonal. The algebra $\s$ is a solvable Lie algebra. Let $N$, $A$ and $S$ be the connected and simply connected groups which correspond to $\n$, $\ag$ and $\s$. We say that $S$ is the \emph{harmonic extension of the $H$-type group $N$}. It is also called a \emph{\DR space}. The map
\begin{align*}
\vg\times\zg\times\RR^+ &\to S\\
(X,Z,a)&\mapsto \exp(X+Z)\exp(\log a \,H)
\end{align*}
gives global coordinates on $S$. The product in $S$ is given by the rule
$$(X,Z,a)(X',Z',a')=\big(X+a^{1/2}X',Z+a\,Z'+1/2\,a^{1/2}\,[X,X'],a\,a'\big )$$
for all $(X,Z,a),\,(X',Z',a')\in S$. Let $e=(0,0,1)$ be the identity of the group $S$. We shall denote by $n=m_{\vg}+m_{\zg}+1$ the dimension of $S$. The
group $S$ is nonunimodular: the right and left Haar measures on
$S$ are  given by $\dir(X,Z,a)=a^{-1}\di X\di Z\di a\,$ and
$\dil(X,Z,a)=a^{-(Q+1)}\di X\di Z\di a\,$ respectively. Then the modular function is $\delta(X,Z,a)=a^{-Q}$. We denote by $L^p(\rho)$, $1\leq p<\infty$,  the space of all measurable functions $f$ such that $\int_S|f|^p \dir<\infty$ and by $\lorentz{1}{\infty}{\rho}$
the Lorentz space of all measurable functions $f$ such that
$$\sup_{t>0}\,t\,\rho\big(\{x\in S:~|f(x)|>t\}\big)<\infty\,.$$
We equip $S$ with the left invariant Riemannian metric which agrees with the inner product on $\s$ at the identity $e$. Let $d$ denote the distance induced by this Riemannian structure. It is well known \cite[formula (2.18)]{ADY} that 
\begin{align}\label{distanza}
\cosh ^2\left(\frac{d\big((X,Z,a),e\big)}{2}\right)=\left(\frac{a^{1/2}+a^{-1/2}}{2}+\frac{1}{8}\,a^{-1/2}|X|^2\right)^2+\frac{1}{4}\,a^{-1}|Z|^2\,,
\end{align}
for all $(X,Z,a)\in S$. We denote by $B\big((X_0,Z_0,a_0),r\big)$ the ball in $S$ centred at $(X_0,Z_0,a_0)$ of radius $r$. In particular let $B_r$ denote the ball of centre $e$ and radius $r$. Note that \cite[formula (1.18)]{ADY} there exist positive constants $\gamma_1$, $\gamma_2$ such that for all $r$ in $(0,1)$
\begin{align}\label{misurapalle1}
\gamma_1\,r^n&\leq\rho\big(B_r\big)\leq \gamma_2\,r^n\,,
\end{align}
and for all $r$ in $[1,\infty)$
\begin{align}\label{misurapalle2}
\gamma_1\,\nep^{Qr}&\leq\rho\big(B_r\big)\leq \gamma_2\,\nep^{Qr}\,.
\end{align}
This shows that $S$, equipped with the right Haar measure $\rho$, is a group of exponential growth.\\
F{}rom (\ref{distanza}) we can easily deduce various properties
of balls in $S$.
\begin{prop}\label{palleS}
The following hold:
\begin{itemize}
\item[(i)] there exists a constant $c_3$ such that 
$$B(e,\log r)\subseteq  B_N\big(0_N, c_3\,r\big)\times (1/r,r) \qquad\forall r\in (1,\infty)\,;$$
\item[(ii)] for all constants $b>0, B>1/2$ there exists a constant $c_{b,B}$ such that
$$ B_N\big(0_N, b\,r^{B}\big)\times (1/r,r)\subseteq  B\big(e,c_{b,B}\,\log r\big) \qquad \forall r\in [\nep,+\infty) \,.$$
\end{itemize}
\end{prop}
\begin{proof}
To prove (i) let $(X,Z,a)$ be in $B(e,\log r)$. By using (\ref{distanza}), we see that
\begin{align*}
\frac{a^{1/2}+a^{-1/2}}{2}&\leq\cosh \left(\frac{d\big((X,Z,a),e\big)}{2}\right)\\
&<\cosh \left(\frac{\log r}{2}\right)\\
&=\frac{r^{1/2}+r^{-1/2}}{2}\,,
\end{align*}
so that $1/r<a<r$, as required. 

Now put $c_3=\max \{{1}/{c_1}, {1}/{\sqrt c_1}\}$, where $c_1$ is the constant which appears in Proposition \ref{normaN}. If $\mathcal N(X,Z)\geq c_3\,r$, then by Proposition \ref{normaN} either $|X|\geq r$ or $|Z|\geq r^2$. In both cases $\cosh ^2\left(\frac{d\big((X,Z,a),e\big)}{2}\right)\geq \cosh ^2\left(\frac{\log r}{2}\right)$, which is a contradiction. This concludes the proof of (i).

To prove (ii), let $(X,Z,a)$ be in $B_N\big(0_N, b\,r^{B}\big)\times (1/r,{r})$. We have that $|X|\leq c_2\,b\,r^{B}$, $|Z|\leq c_2\,b^2\,r^{2B}$ and $a\in(1/r,{r})$, where $c_2$ is the constant which appears in Proposition \ref{normaN}. Applying (\ref{distanza}) we obtain that
\begin{align*}
\cosh ^2\left(\frac{d\big((X,Z,a),e\big)}{2}\right)&< \Big(\frac{r^{1/2}+r^{-1/2}}{2}+\frac{c_2\,b}{8}\,r^{(B-1/2)}  \Big)^2+\frac{c_2\,b^2}{4}\,r^{(2B-1)}\\
&\leq C\,r^{(2B-1)}\,,
\end{align*}
where $C$ depends only on $b$ and $B$. Thus, there exists a constant $c_{b,B}$ such that 
$$d\big((X,Z,a),e\big)<c_{b,B}\,\log r\,,$$
as required.
\end{proof}
We now give the simplest example of \DR spaces, the $ax+b\,$-groups.
\begin{ese}\label{ax+b}
{\bf{The $ax+b\,$-groups}}

{\rm{The group $N=\RR^d$ is an $H$-type group where $\vg=\RR^d$ and $\zg=(0)$. The dilations on $\RR^d$ are given by $\delta_ax=a^{1/2}\,x$, for all $x$ in $\RR^d$, $a$ in $\RR^+$. Note that the homogeneous norm agrees with the usual Euclidean norm up to a multiplicative constant and the homogeneous dimension is $Q=d/2$. The harmonic extension of $\RR^d$ is given by $S=\RR^d\times \RR^+$, where the product rule is
$$(x,a)(x',a')=(x+a^{1/2}x',a\,a')\qquad \forall (x,a),\,(x',a')\in S\,.$$
When $d=1$ the group $S$ is usually called the affine group of the real line or $ax+b\,$-group. We abuse the terminology and call {\it{$ax+b\,$-groups}} all the groups $S$ just described. 

The dimension of $S$ is $n=d+1$, the right and left Haar measures are given by $\dir(x,a)=a^{-1}\di x\di a$ and $\dil(x,a)=a^{-d/2-1}\di x\di a$.

For more details on these groups see \cite{F}.
}}
\end{ese}
\section{Spherical analysis}\label{sphericalanalysis}
In this section we summarize some facts concerning the spherical analysis on \DR spaces. For details we refer the reader to \cite{DR1, DR2, R}.

We may identify $S$ with the open unit ball $\mathcal B$ in $\mathfrak s$
$$\mathcal B=\{(X,Z,t)\in\vg \times \zg \times \RR: ~|X|^2+|Z|^2+t^2<1\}\,,$$
via the bijection (see \cite{CDKR2}) $F:S\to \mathcal B$ defined by 
\begin{align*}
F(X,Z,a)=\frac{1}{\Big(1+a+\frac{1}{4}|X|^2\Big)^2+|Z|^2}\Big(&\Big(1+a+\frac{1}{4}|X|^2-J_Z\Big)X,2Z, \\
&-1+\Big(a+\frac{1}{4}|X|^2\Big)^2+|Z|^2\Big)\,.
\end{align*}
The following integration formula holds.
\begin{teo}\label{intS}
The left Haar measure on $S$ may  be normalized in such a way that for all function $f$ in $C_c^{\infty}(S)$ 
$$\int _S f\dil=\int_0^{\infty}\int_{\partial \mathcal B}f\big(F^{-1}(r\omega)\big)A(r)\di r\di \sigma(\omega)\,,$$
where $\di \sigma$ is the surface measure on $\partial \mathcal B$ and
\begin{equation}\label{misura}
A(r)=2^{m_{\vg}+2m_{\zg}}\sinh ^{m_{\vg}+m_{\zg}}\left(\frac{r}{2}\right)\cosh ^{m_{\zg}}\left(\frac{r}{2}\right)\,.
\end{equation}
\end{teo}
It is easy to check that
\begin{equation}\label{pesoA}
A(r)\leq C \left(\frac{r}{1+r}\right)^{n-1}\nep^{Qr}\qquad\forall r\in\RR^+\,.
\end{equation}
We say that a function $f$ on the group $S$ is radial if it depends only on the distance from the identity, i.e., if there exists a function $f_0$ defined on $[0,+\infty)$ such that $f(X,Z,a)=f_0(r)$, where $r=d\big((X,Z,a),e\big)$. We abuse the notation and write $f(r)$ instead of $f_0(r)$. For every function $f$ on $S$ let $\check{f}$ be defined by
$$\check{f}(x)=f(x^{-1})\qquad\forall x\in S\,.$$
Note that if $f$ is radial, then $\check{f}=f$.

Damek and Ricci \cite{DR1} defined the radialisation operator
$$\mathcal R: ~C^{\infty}_c(S)\to C^{\infty}_c(S)$$
in the following way:
$$\mathcal R f(x)=\left({\mathcal{\widetilde{R}}}(f\circ F^{-1})\right)(Fx)\qquad\forall x\in S,$$
where ${\mathcal{\widetilde{R}}}$ is the radialisation operator on the ball $\mathcal B\,$ defined by
$$({\mathcal{\widetilde{R}}}\phi)(\omega)=\frac{1}{|\partial \mathcal B|}\int_{\partial \mathcal B}\phi(\|\omega\|\omega)\di \sigma(\omega)\,.$$
A function $f$ is radial if and only if $\mathcal R(f)=f$. Damek and Ricci proved the following result.
\begin{prop}\label{radialisation}
The operator $\mathcal R$ extends to a bounded operator on $L^p(\lambda)$, for all $1\leq p\leq \infty$,  and is an orthogonal projector of $L^2(\lambda)$. Moreover for all $f,g\in C^{\infty}_c(S)$ 
$$\int_S\mathcal R f\dil=\int_S f\dil\,,$$
and
$$\int_S (\mathcal R f)\, g\dil=\int_S f\,(\mathcal R g)\dil\,.$$
\end{prop} 

For future developments it is useful to recall the definition of convolution on the space $S$. For all functions $f$, $g$ in $C_c(S)$ their convolution is defined by
\begin{align*}
f\ast g(x)&=\int_Sf(xy)\,g(y^{-1})\dil (y)\\
&=\int_Sf(xy^{-1})\,g(y)\dir (y)\qquad\forall x\in S\,.
\end{align*}

On the Riemannian manifold $S$ we may consider the (positive definite) Laplace-Beltrami operator $~\mathcal L$, defined by
$$\mathcal L=- {\rm div}\circ {\rm grad}\,.$$
\begin{defi}
A radial function $\phi$ on the group $S$ is called {\rm{spherical}} if $\phi$ is an eigenfunction of the Laplace-Beltrami operator $\mathcal L$ and $\phi(e)=1$.
\end{defi}
Let $E_0,...,E_{n-1}$ be an orthonormal basis of the algebra $\s$ such that 
$E_0=H$, $E_1,...,E_{m_{\vg}}$ is an orthonormal basis of $\vg$ 
and $E_{m_{\vg}+1},...,E_{n-1}$ is an orthonormal basis of $\zg$. 
Let $X_0,X_1,...,X_{n-1}$ be the left invariant vector fields 
on $S$ which agree with $E_0, E_1,...,E_{n-1}$ at the identity. For every $f$ in $ C^{\infty}_c(S)$
$$X_if(x)=\frac{d}{d\tau }\Big\lvert_{\tau =0}f\big(x\exp(\tau E_i)\big)\qquad\forall x\in S\,,i=0,...,n-1\,.$$
In particular, if $i=0$, then
\begin{align*}
X_0f(X,Z,a)&=\frac{d}{d\tau }\Big\lvert_{\tau =0}f\big((X,Z,a)\exp(\tau E_0)\big)\\
&=a\,\partial _a f(X,Z,a)\,,
\end{align*}
while the vector fields $X_i$ do not involve derivative in the variable $a$, for all $i\neq 0$.
Damek \cite {D} proved that
\begin{equation}\label{LaplaceBeltrami}
\LB f=-\sum_{i=0}^{n-1} X_i^2f+QX_0f \qquad\forall f\in C^{\infty}_c(S)\,.
\end{equation}
In particular, by applying (\ref{LaplaceBeltrami}) to the functions $\delta^{is/Q-1/2}$, $s\in\CC$, we have that
\begin{align*}
\mathcal L \big(\delta^{is /Q-1/2}\big)(X,Z,a)&=(-X_0 ^2+Q\,X_0)(a^{-is+Q/2})\\
&=\big(-(-is+Q/2)^2+Q(is+Q/2)\big)a^{-is+Q/2}\\
&=\left(s^2-{Q^2}/{4}\right)\delta^{is/Q-1/2}(X,Z,a)\qquad \forall (X,Z,a)\in S\,.
\end{align*}
Thus the functions $\delta^{is /Q-1/2}$ are eigenvalues of the Laplace-Beltrami operator and take value $1$ in $e$, but they are not radial. Let $\phi _{s}$ be defined by $\mathcal R (\delta^{is /Q-1/2})$, for $s\in\CC$. We have the following result.
\begin{prop}
The function $\phi _{s}$ is a spherical function with eigenvalue $s ^2-Q^2/4$. All spherical functions are of this type and $\phi _{s}=\phi _{-s}$.
\end{prop}
In particular, the spherical function $\phi_0$ satisfies the following estimate \cite[Lemma 1]{A2}:
\begin{equation}\label{stimafi0}
\phi_0(r)\leq C\,(1+r)\,\nep^{-Qr/2}\qquad\forall~ r\in\RR^+\,.
\end{equation}
We shall use the following integration formula on $S$, whose proof is reminiscent of \cite[Lemma 1.3]{CGHM} and \cite[Lemma 3]{A1}:
\begin{lem}\label{intduf}
For every radial function $f$ in $C_c^{\infty}(S)$
\begin{align*}
\int_S\delta^{1/2}f \,{\rm{d}} \rho &=\int_0^{\infty}\phi_0(r)\,f(r)\,A(r)\,{\rm{d}} r \,.
\end{align*}
\end{lem}
\begin{proof}
By using the fact that $\mathcal R f=f$, the integration formula (\ref{intS}) and Proposition \ref{radialisation}, we obtain that
\begin{align*}
\int_S\delta^{1/2}\,f\dir &=\int_S\delta^{-1/2}\,f\dil\\
&=\int_S\delta^{-1/2}\,(\mathcal R f)\dil\\
&=\int_S(\mathcal R \delta^{-1/2})\,f\dil\\
&=\int_S\phi_0\,f\dil\\
&=\int_0^{\infty}\phi_0(r)\,f(r)\,A(r)\di r\,,
\end{align*}
as required.
\end{proof}
We now define the spherical transform on $S$ and recall its properties.
\begin{defi}
The spherical Fourier transform of an integrable radial function $f$ on $S$ is defined by 
$$\mathcal H f(s)=\int_S \phi _{s}\,f\dil \,.$$
\end{defi}
For ``nice'' radial functions the spherical Fourier transform satisfies the following inversion and Plancherel formulas.
\begin{teo}
For every radial function $f$ in $C_c(S)$ the following inversion formula holds:
$$f(x)=c_S\int_{0}^{\infty}\mathcal H f (s)\phi _{s}(x)| {\bf{c}} (s) |^{-2} \di s\,,$$
where the constant $c_S$ depends only on $m_{\vg}$ and $m_{\zg}$ and ${\bf{c}}$ denotes the Harish-Chandra function. The Plancherel measure satisfies the following estimate:
\begin{equation}\label{HC}
|{\bf{c}}(s)|^{-2}\leq \begin{cases}
|s|^2 & \text{if $| s |\leq 1$}\\
|s|^{n-1} & \text{if $| s |>1\,.$}
\end{cases}
\end{equation}
Moreover the Plancherel formula holds:
$$\int_S|f|^2\dil=c_S\int_0^{\infty}|\mathcal H f (s)|^2|{\bf{c}}(s)|^{-2}\di s\,.$$
The spherical Fourier transform extends to an isometry between the space of radial functions in $L^2(\lambda)$ and $L^2(\RR^+,c_S\,|{\bf{c}}(s)|^{-2}\di s)$. 
\end{teo}
In this context an analogue of the Paley--Wiener Theorem holds. 
\begin{teo}
The spherical Fourier transform is an isomorphism between the space of all radial functions in $C^{\infty}_c(S)$ and the space of even entire fun\-ctions of exponential type on $\CC$. Moreover the function $f$ has support in the ball $B_r$ if and only if its spherical transform $\mathcal H f$ satisfies
$$|\mathcal H f (s)|\leq C_N(1+s)^{-N}\,\nep^{r|\rm{Im}(s)|}\qquad\forall s\in\CC\quad\forall N\in\NN\,.$$
\end{teo}
{F}rom the definition of the spherical transform it follows that for every radial function $f$
\begin{align*}
\mathcal H f(s)&=\int_{\RR^+}\int_N f(X,Z,a)\,\phi _{s}(X,Z,a)\,a^{-Q-1}\di X\di Z\di a\\
&=\int_{\RR^+}\int_N \mathcal R f(X,Z,a)\,\delta^{is/Q-1/2}(X,Z,a)\,a^{-Q-1}\di X\di Z\di a\\
&=\int_{\RR^+}\Big(\int_N f(X,Z,a)\,a^{-Q/2}\di X\di Z\Big)a^{-is-1}\di a\\
&=\int_{\RR}\Big(\int_N f(X,Z,\nep^t)\,\nep^{-Qt/2}\di X\di Z\Big)\nep^{-ist}\di t\\
&=\int_{\RR}\mathcal Af(t)\,e^{-ist}\di t\\
&=\mathcal F \circ \mathcal A f (s)\,,
\end{align*}
where $\mathcal F$ denotes the Fourier transform on the real line and $\mathcal A$ denotes the Abel transform defined by
$$\mathcal Af(t)= \int_N f(X,Z,\nep^t)\,\nep^{-Qt/2}\di X\di Z\,.  $$
Hence, at least formally
$$\mathcal H=\mathcal F\circ \mathcal A~~~\text{and}~~~\mathcal H ^{-1}=\mathcal A ^{-1}\circ \mathcal F ^{-1}\,.$$
We shall use the inversion formula for the Abel transform \cite[formula (2.24)]{ADY}, which we now recall. Let $\mathcal D_1$ and $\mathcal D_2$ be the differential operators on the real line defined by
\begin{align}\label{D1D2}
\mathcal D_1=\,-\frac{1}{\sinh r}\,\frac{\partial}{\partial r}\,,\qquad \mathcal D_2=\,-\frac{1}{\sinh(r/2)}\,\frac{\partial}{\partial r} \,.
\end{align} 
If $m_{\zg}$ is even, then
\begin{equation}\label{inv1}
\mathcal A^{-1}f(r)=a_S^e\,\mathcal D_1^{m_{\zg}/2}\mathcal D_2^{m_{\vg}/2}f (r)\,,
\end{equation}
where $a_S^e=2^{-(2m_{\vg}+m_{\zg})/2}\pi^{-(m_{\vg}+m_{\zg})/2}$, while if $m_{\zg}$ is odd, then
\begin{align}\label{inv2}
\mathcal A^{-1}f(r)=a_S^o\int_r^{\infty}\mathcal D_1^{(m_{\zg}+1)/2}\mathcal D_2^{m_{\vg}/2}f (s) \di\nu(s)\,,
\end{align}
where $a_S^o= 2^{-(2m_{\vg}+m_{\zg})/2}\pi^{-n/2}$ and $\di\nu(s)=(\cosh s-\cosh r)^{-1/2}\sinh s \di s$.

\section{Direct products of \DR spaces}\label{products}
Let $S'$ and $S''$ be the harmonic extensions of two $H$-type groups $N'$ and $N''$. We consider the direct product $S=S'\times S''$. 

The dimension of $S$ is obviously $n=n'+n''$. The group $S$ is nonunimodular: the right and left Haar measures on $S$ are given by $\dir\big((x',x'')\big)=\dir '(x')\dir ''(x'')$ and $\dil\big((x',x'')\big)=\dil '(x')\dil ''(x'')$, respectively. Then the modular function is $$\delta\big((X',Z',a'),(X'',Z'',a'')\big)=a'^{-Q'}\,a''^{-Q''}\,.$$
We equip $S$ with the Riemannian metric which is the direct sum of the Riemannian metrics of $S'$ and $S''$ and denote by $d$ the distance induced by this Remannian structure. We denote by $d_{max}$ the following ``product distance'' on $S'\times S''$:
$$d_{max}\big((x',x''),(y',y'')\big)=\max \big\{d_{S'}(x',y'),d_{S''}(x'',y'')\big\}\,.$$
It is easy to check that the distances $d$ and $d_{max}$ satisfy
$$d_{max}\big((x',x''),(y',y'')\big)\leq d\big((x',x''),(y',y'')\big)\leq 2\,d_{max}\big((x',x''),(y',y'')\big)\,,$$
for all $(x',x''),(y',y'')$ in $S$. 

In the sequel we denote by $B(x,r)$ the ball of centre $x$ and radius $r$ with respect to metric $d_{max}$ which is equal to $B_{S'}(x',r)\times B_{S''}(x'',r)$. Note that \cite[formula (1.18)]{ADY} there exist positive constants $\gamma_1$, $\gamma_2$ such that for all $r$ in $(0,1)$
\begin{align}\label{misurapalle1p}
\gamma_1\,r^n&\leq\rho\big(B(e,r)\big)\leq \gamma_2\,r^n\,,
\end{align}
and for all $r$ in $[1,\infty)$
\begin{align}\label{misurapalle2p}
\gamma_1\,\nep^{(Q'+Q'')r}&\leq\rho\big(B(e,r)\big)\leq \gamma_2\,\nep^{(Q'+Q'')r}\,.
\end{align}
Hence $S$, equipped with the right Haar measure $\rho$, is a group of exponential growth.

Given two functions $h'$ and $h''$ on $S'$ and $S''$ respectively, we denote by $h'\otimes h''$ the function on $S$ defined by $(h'\otimes h'')(x',x'')=h'(x')\,h''(x'')$.

On the Riemannian manifold $S$ endowed with the direct sum of the Riemannian metric of $S'$ and $S''$ we may consider the Laplace-Beltrami operator $\,\mathcal L$. It is easy to check that for all $f$ in $C^{\infty}_c(S)$ 
$$\mathcal L f(x',x'')=\big(\mathcal L'f(\cdot,x'')\big)(x')+\big(\mathcal L''f(x',\cdot)\big)(x'')\,,$$
where $\LB '$ and $\LB ''$ denote the Laplace-Beltrami operators on $S'$ and $S''$, respectively.

Let $\phi '_{s'}$ and $\phi ''_{s''}$ denote the spherical functions on $S'$ and $S''$ respectively. Then
\begin{align*}
\mathcal L( \phi'_{s'}\otimes \phi''_{s''})&=(\mathcal L'\phi'_{s'})\otimes \phi''_{s''}+\phi '_{s'}\otimes (\mathcal L''\phi ''_{s''})\\
&=\big(s'^2+Q'^2/4+s''^2+Q''^2/4\big)\phi'_{s'}\otimes \phi''_{s''}\\
&=\big(s'^2+s''^2+{(Q'^2+Q''^2)}/4\big)\phi'_{s'}\otimes \phi''_{s''}\,.
\end{align*}
Thus $\phi'_{s'}\otimes\phi''_{s''}$ are spherical functions on $S$. The spherical Fourier transform of an integrable radial function $f$ on $S$ is defined by 
$$\mathcal H f(s',s'')=\int_S (\phi'_{s'}\otimes\phi''_{s''})\,f\di s'\di s'' \,.$$
For ``nice'' radial functions the spherical Fourier transform satisfies the following inversion and Plancherel formulas:
$$f(x)=c_S\int_{0}^{\infty}\int_{0}^{\infty}\mathcal H f (s',s'')\,(\phi'_{s'}\otimes \phi''_{s''})(x)| {\bf{c}} (s') |^{-2} \di s'| {\bf{c}} (s'') |^{-2}\di s''\,,$$and 
$$\int_S|f|^2\dil=c_S\int_0^{\infty}\int_0^{\infty}|\mathcal H f (s',s'')|^2\,|{\bf{c}}(s')|^{-2}\di s'\,|{\bf{c}}(s'')|^{-2}\di s''\,,$$
where ${\bf{c}}$ denotes the Harish-Chandra fun\-ction. 

We say that a function $f$ on $S'\times S''$ is {\it{biradial}} if there exists a function $f_0$ on $\RR^+\times \RR^+$ such that $f(x',x'')=f_0\big(d(x',e'),d(x'',e'')  \big)$, for all $(x',x'')\in S'\times S''$. 
We shall use the following integration formula on $S$.
\begin{lem}\label{intdufp}
Let $f$ be a biradial function in $C_c^{\infty}(S)$. Then
\begin{align*}
\int_S\delta^{1/2}f \,{\rm{d}} \rho &=\int_0^{\infty}\int_0^{\infty}\phi '_0(r')\,\phi ''_0(r'')\,f_0(r',r'')\,A'(r')\,A''(r''){\rm{d}} r'\di r'' \,,
\end{align*}
where $A'$ and $A''$ are defined as in Theorem \ref{intS}.
\end{lem}

\chapter{Maximal operators}
\begin{intro*}
Let $S$ be a \DR space. In this chapter we introduce some Hardy--Littlewood type maximal operators on $S$ and we study their weak type $(1,1)$ boundedness with respect to the right Haar measure. More precisely we study the left invariant Hardy--Littlewood type maximal operator $M^{\mathcal F}$ defined by
$$
M^{\mathcal F}f(x)
= \sup_{F\in\mathcal F,\,x\in F}\frac{1}{\rho(F)}\int_{F} |f|
\dir\qquad \forall f\in L^1_{\rm{loc}}(\rho)\,,
$$ 
where $\mathcal F$ is a family of open subsets of $S$. We look for families $\mathcal F$ such that $M^{\mathcal F}$ is of weak type $(1,1)$ with respect to the right Haar measure $\rho$. 

In Section \ref{nice} we give a general result: if a family $\mathcal F$ satisfies a ``good property'', i.e. it is \emph{nicely ordered}, then the associated maximal operator $M^{\mathcal F}$ is of weak type $(1,1)$.

In the following section we introduce some families of subsets of $S$: a family $\mathcal R^0$ which consists of small balls, a family $\mathcal R^{\infty}$ of ``big rectangles'' and a family $\mathcal D^{\infty}$ of ``big dyadic sets''. We prove that they are all nicely ordered and that the associated maximal operators are of weak type $(1,1)$. Then in Section \ref{equivalence} we show that the maximal operator $M^{\mathcal D^{\infty}}$ and $M^{\mathcal D^{\infty}}$ are equivalent.

In Section \ref{maxoperatorsp} we introduce some maximal operators on products of \DR spaces and study their weak type $(1,1)$ boundedness. 
\end{intro*}
\section{The maximal operator $M^{\mathcal F}$}\label{nice}
Let $\mathcal F$ be a family of open subsets of $S$. We study the maximal operator $M^{\mathcal F}$ associated to $\mathcal F$ with respect to the right Haar measure defined by
$$
M^{\mathcal F}f(x)
= \sup_{F\in\mathcal F,\,x\in F}\frac{1}{\rho(F)}\int_{F} |f|
\dir\qquad \forall f\in L^1_{\rm{loc}}(\rho)\,.
$$ 
We are interested in finding families $\mathcal F$ for which $M^{\mathcal F}$ is bounded from $L^1(\rho)$ to the Lorentz space $\lorentz{1}{\infty}{\rho}.$ 
\begin{defi}\label{wellorder}
We say that $(\mathcal F,\leq )$, where $\mathcal F$ is a family of open subsets of $S$ and $\leq$ is a preorder relation on $\mathcal F$, is {\rm{nicely ordered}} if for each set $F$ in $\mathcal F$ there exists a set $\tF$ which satisfies the following properties:
\begin{itemize}
\item [(i)] $\rho(\tF)\leq \tC\,\rho(F)$, where the constant $\tC$ does not depend on the set $F$;
\item[(ii)] if $F_1,\,F_2$ are sets in $\mathcal F$ such that $F_2\leq F_1$ and $F_2\cap F_1\neq \emptyset$, then $F_2\subseteq \tF_1$.
\end{itemize} 
\end{defi}
If $\mathcal F$ is a nicely ordered family, then it possesses a nice covering property.
\begin{lem}\label{covering}
Let $\mathcal F$ be a nicely ordered family of open subsets of $S$. For every finite collection of sets $\{F_i\}_{i\in I}$ in $\mathcal F$, there exists a subcollection of mutually disjoint sets $F_1,...,F_k$ such that $$\bigcup_{i\in I} F_i\subseteq
\bigcup_{j=1}^k\tF_j.$$
\end{lem}
\begin{proof}
Choose a set $F_1$ such that $F_i\leq F_1$, for all $i\in I$. This is possible because $I$ is finite.

Suppose that $F_1,..,F_n$ have been chosen. Then either there are no sets of $\{F_i\}_{i\in I}$ disjoint from $F_1,..,F_n$ and the process stops, or we choose a set $F_{n+1}$ disjoint from the sets already selected and such that $F_i\leq F_{n+1}$, for all $i\in I$ such that $F_i\cap \big(F_1\cup \ldots\cup F_n\big)=\emptyset$.

This process stops after a finite number of steps (because $I$ is finite). Let $F_1,\ldots,F_k$ be the selected sets. Now, given a set $F_i$, $i\in I$, either $F_i$ is one of the sets which we have selected or there exists an index $F_j$, $1\leq j\leq k$ such that $F_i\leq F_j$ and $F_i\cap F_j\neq \emptyset$; in this case, by property (iii) of Definition \ref{wellorder}, $F_i\subseteq \tF_j$, and the lemma is proved.
\end{proof}
A noteworthy consequence of Lemma \ref{covering} is that maximal operators associated to nicely ordered families are of weak type $(1,1)$.
\begin{teo}\label{well11}
If $\mathcal F$ is a nicely ordered family, then the maximal operator $M^{\mathcal F}$ is bounded from $L^1(\rho)$ to $\lorentz{1}{\infty}{\rho}.$ 
\end{teo}
\begin{proof}
Let $f$ be in $L^1(\rho)$ and $t>0$. Set $\Omega _{t}=\{x\in S:~M^{\mathcal F}f(x)>t\}$ and let $K$ be any compact subset of $\,\Omega_{t}\,$. By the compactness of $K$, we can select a finite collection of sets $\{F_i\}$ in $\mathcal F$ which cover $K$ and such that
\begin{align*}
\frac{1}{\rho(F_i)}\int_{F_i}|f|\dir &>t\,.
\end{align*}
By Lemma \ref{covering} we may select a disjoint subcollection $F_1,...,F_k$ of $\{F_i\}$ such that $K\subseteq \bigcup_{i=1}^k \tF_i$. Thus,
$$\rho(K)\leq \sum_{i=1}^k\rho(\tF_i)\leq \tC\,\sum_{i=1}^k\rho(F_i)\leq \frac{\tC}{t}\,\sum_{i=1}^k\int_{F_i}|f|\dir\leq \frac{\tC}{t}\,\|f\|_{L^1(\rho)}\,.$$
By taking the supremum of both sides over all $K\subseteq \Omega_{t}$, we obtain that $\rho(\Omega_t)\leq \frac{\tC}{t}\,\|f\|_{L^1(\rho)}$, as required.
\end{proof}

\section{The maximal operators $M^{\mathcal R^0}$, $M^{\mathcal R^{\infty}}$ and $M^{\mathcal D^{\infty}}$ }\label{small}
In this section we introduce three families of subsets of $S$: $\mathcal R^0$, $\mathcal R^{\infty}$ and $\mathcal D^{\infty}$. 

The family $\mathcal R^0$ is the family of balls of small radius, i.e.
$$\mathcal R^{0}=\{B(x_0,\, r):~x_0\in S\,,r<1/2\}\,.$$
%For each ball $B(x_0,\, r)$ in $\mathcal R^{0}$ let $\tilde{B}$ denote the ball centred at $x_0$ of radius $3\,r$.

\bigskip
The family $\mathcal R^{\infty}$ consists of ``big rectangles''. Suppose that $b',\,b''$ are constants such that $1/2< b'<b''$. We define $\mathcal R^{\infty}$ as the family of all left-translates of the sets $E_{r,\,b}$, i.e.
\begin{equation}\label{Rinftydef}
\mathcal R^{\infty} =\{x_0 E_{r,\,b}:~x_0\in S,\,r \geq  \nep,\,b'<b<b''\}\,,
\end{equation}
where
\begin{equation}\label{Erb}
E_{r,\,b}=B_N\big(0_N,r^{b}\big)\times \big(1/r,r\big) \qquad \forall r\geq \nep\,.
\end{equation}
Let $\gamma$ be a constant greater than $\frac{4b'' +2b'+1}{2b'-1}$. For each set $E_{r,\,b}$, we define its dilated set as 
$$\tE_{r,\,b}=B_N\big(0_N,\gamma\, r^{b}\big)\times \big({1}/{r^{\gamma}},r^{\gamma}\big)  \,.$$
Note that
\begin{align}\label{misuraE}
\rho(E_{r,\,b})&=2\,r^{2b Q}\,|B_N(0,1)|\,\log r\nonumber\\
&\asymp r^{2b Q}\log r \,.
\end{align}
Furthermore the measures of $E_{r,\,b}$ and $\tE_{r,\,b}$ are comparable.

\bigskip
The family $\mathcal D^{\infty}$ is similar to $\mathcal R^{\infty}$, the main difference being that we replace $B_N(0_N,r^b)$ in (\ref{Erb}) with a dyadic set in $N$. On spaces of homogeneous type, in particular on nilpotent groups, dyadic sets have been introduced by M.~Christ \cite[Theorem 11]{C}. For the reader's convenience we recall their properties in the following theorem.
\begin{teo}\label{dyadic}
Let $N$ be an $H$-type group endowed with the homogeneous distance $d_N$ and the Haar measure. There exist a collection of sets $\{\Qka\subset N:~k\in\ZZ, \al\in I_k\}$, where $I_k$ is a countable index set, constants $\rd>1$, $c_N>0$, $C_N>0$ and an integer $M$ such that: 
\begin{itemize}
\item[(i)] $|N-\bigcup_{\al\in I_k}\Qka|=0\qquad\forall k\in\ZZ$;
\item[(ii)]there are points $n_{\al}^k$ in $N$ such that 
$B_N(n_{\al}^k,c_N\,\rd^k)\subseteq\Qka\subseteq B_N(n_{\al}^k,C_N\,\rd^k)$;
\item[(iii)] $\Qka \cap Q_{\beta}^k =\emptyset$ if $\al \neq \be$;
\item[(iv)] each set $\Qka$ has at most $M$ subsets of type $Q_{\be}^{k-1}$;
\item[(v)]$\forall \ell\leq k$ and $\beta$ in $I_{\ell}$ there is a unique $\alpha$ in $I_k$ such that $Q_{\be}^{\ell}\subseteq\Qka $;
\item[(vi)]if $\ell\leq k$, then either $\Qka\cap Q_{\be}^{\ell}=\emptyset$ or $Q_{\be}^{\ell}\subseteq \Qka $.
\end{itemize} 
\end{teo}
\begin{ossn} {\rm{Note that if $N=\RR^d$ the dyadic sets agree with standard dyadic cubes in $\RR^d$. For all $k$ in $\ZZ$ and $m\in \ZZ^d$ the dyadic cube $Q^k_m$ is defined by
$$Q^k_m=\prod_{i=1}^d \big(m_i\,2^k,(m_i+1)\,2^k\big)\,,  $$
which is centred at the point $x^k_m=\big((m_1+1/2)\,2^k,\ldots,(m_d+1/2)\,2^k\big)$. The constants which appear in Theorem \ref{dyadic} in this case are $\rd=2$, $c_N=1/2$, $C_N=\sqrt{d}/2$, $M=2^d$. Indeed, 
$$B_{\RR^d}(x^k_m,2^{k-1})\subseteq Q^k_m\subseteq B_{\RR^d}(x^k_m,\sqrt{d}\,2^{k-1})\,.$$ }}
\end{ossn}

\bigskip
We define ``big admissible sets'' as products of dyadic sets in $N$ and intervals in $A$. Roughly speaking, we may think of these sets as left translates of a family of sets containing the identity. We cannot exactly do that, because left translates and dilated of dyadic sets in $N$ are not dyadic sets. 
\begin{defi}\label{admissibility}
A \emph{big admissible set} is a set of the form $\Qka\times (a_0/r,a_0\,r)\,,$ where $\Qka$ is a dyadic set in $N$, $a_0\in A$, $r\geq {\rm{e}}$,
\begin{align}\label{adm}
a_0^{1/2}r^{\beta}&\leq \rd^k < a_0^{1/2}r^{4\beta}\,,\end{align}
and $\beta$ is a constant $>{\rm  max}\left\{ 3/2,1/4+\log \rd , 1+\log\big(c_3/c_N\big)\right\}\,$, where $c_3$, $\rd$, $c_N$ are the constants which appear in Proposition \ref{palleS} and Theorem \ref{dyadic}. 
\end{defi}
The family $\mathcal D^{\infty}$ is defined by
$$\mathcal D^{\infty}=\{R:\,R ~{\rm{is~ a~ big~admissible~set}} \}\,.$$
Given a set $R=\Qka\times (a_0/r,a_0\,r)$ in $\mathcal D^{\infty}$ we define 
\begin{equation}\label{dil}
\tR=\Qka\times (a_0/{r^{\gamma}},a_0\,r^{\gamma})\,,
\end{equation}
where $\gamma=\frac{9\,\beta+1}{\beta\,-1}$. 

\begin{lem}\label{famnice}
The following families are nicely ordered:
\begin{itemize}
\item[(i)] the family $(\mathcal R^0,\,\leq )$, where $B(x_2,\,r_2)\leq B(x_1,\,r_1)$ if $r_2\leq r_1$ and the dilated of $B(x_0,r)$ is the ball $\tilde{B}$ centred at $x_0$ of radius $3\,r$;
\item[(ii)] the family $(\mathcal R^0,\,\leq )$, where $B_2\leq B_1$ if $\rho(B_2)\leq 2\,\rho(B_1)$ and the dilated of $B(x_0,r)$ is the ball $\tilde{B}$ centred at $x_0$ of radius $\gamma \,r$, where $\gamma=1+2\Big(2\,\nep^Q\,\gamma_2/{\gamma_1}\Big)^{1/n}$ and $\gamma_i$, $i=1,2$, are as in (\ref{misurapalle1});
\item[(iii)] the family $(\mathcal R^{\infty},\,\leq )$, where $R_2\leq R_1$ if $\rho(R_2)\leq \rho(R_1)$ and the dilated of a set $x_0E_{r,\,b}$ is the set $x_0\tE_{r,\,b}$;
\item[(iv)] the family $(\mathcal D^{\infty},\,\leq )$, where $Q^{k_2}_{\alpha_2}\times (a_2/r_2,a_2\,r_2)\leq Q^{k_1}_{\alpha_1}\times (a_1/r_1,a_1\,r_1)$ if $k_2\leq k_1$ and the dilated of a set is defined by (\ref{dil}).
\end{itemize}
\end{lem}
\begin{proof}
We prove (i). First note that by (\ref{misurapalle1}) there exists a constant $\tC$ such that $\rho(\tB)\leq \tC\,\rho(B)$, for each ball $B$ in $\mathcal R^0$.

Let $B_i$ be in $\mathcal R^{0}$ such that $B_2\leq B_1$ and $B_1\cap B_2\neq \emptyset$. Then for each
point $x$ in $ B_2$, we have that
$$d(x,x_1)\leq d(x,x_2)+d(x_2,x_1)< 2\,r_2+ r_1\leq 3\, r_1\,.$$
Thus the point $x$ is in $\tB_1$. This proves that $\mathcal R^{0}$ is nicely ordered.

We now prove (ii). First note that by (\ref{misurapalle1}) there exists a constant $\tC$ such that $\rho(\tB)\leq \tC\,\rho(B)$, for each ball $B$ in $\mathcal R^0$.

Let $B_i=B(x_i,r_i)$ be in $\mathcal R^{0}$ such that $B_2\leq B_1$ and $B_1\cap B_2\neq \emptyset$. The condition $\rho(B_2)\leq 2\,\rho(B_1)$ implies that $\delta(x_2)^{-1}\,\gamma_1\,r_2^n\leq 2\,\delta(x_1)^{-1}\,\gamma_2\,r_1^n$. Thus
\begin{align*}
r_2&\leq \big(2\,\delta(x_2x_1^{-1})\,\gamma_2/{\gamma_1}\big)^{1/n}\,r_1\,.
\end{align*}
Since $x_2x_1^{-1}$ is in $B(e,1)$ we have that $\delta(x_2x_1^{-1})\leq \nep^Q$ and then 
\begin{align*}
r_2&\leq \big(2\,\nep^Q\,\gamma_2/{\gamma_1}\big)^{1/n}\,r_1\,.
\end{align*}
It follows that
\begin{align*}
B(x_2,r_2)&\subseteq B(x_2,2r_2+r_1)\\
&\subseteq B\Big(x_1,\left(1+2\big(2\,\nep^Q\,\gamma_2/{\gamma_1}\big)^{1/n}\Big)r_1\right)\\
&=B(x_1,\gamma \,r_1)\\
&=\tilde B_1\,,
\end{align*}
as required. This proves that $\mathcal R^{0}$ is nicely ordered.

We now prove (iii). Suppose that $R_i=x_iE_{r_i,\,b_i}$, for $i=1,2$, are sets in $\mathcal R^{\infty}$ such that $R_2\leq R_1$ and $R_1\cap R_2\neq \emptyset$. Without loss of generality we may suppose that $R_2$ is centred at the identity. Indeed, if not, then $x_2^{-1}R_2\leq x_2^{-1}R_1$, they intersect and $x_2^{-1}R_2$ is centred at the identity. If we prove that $x_2^{-1}R_2\subseteq x_2^{-1}\tR_1$, then also $R_2\subseteq \tR_1$.

Then we suppose that $x_2=e$ and $x_1=(n_1,a_1)$. It is straightforward to check that the condition $R_2\leq R_1$ implies that
\begin{align}\label{misure}
\frac{\big(a_1^{1/2}r_1^{b_1}\big)^{2Q}}{r_2^{2b_2Q}}\,\frac{\log r_1}{\log r_2}&\geq 1\,.
\end{align}
The fact that $R_1$ and $R_2$ intersect implies that
\begin{align}\label{condint}
&\frac{1}{r_1r_2}<{a_1}<r_1r_2 \hspace{1cm} \text{and} \hspace{1cm}  d_N(n_1,0_N)<a_1^{1/2}r_1^{b_1}+r_2^{b_2}\,.
\end{align}
Let $(n,a)$ be a point of $R_2$; we shall prove that it belongs to $\tR_1$. {F}rom (\ref{condint}) we deduce that
\begin{align}\label{As}
\frac{1}{r_2^2r_1}&<\frac{a}{a_1}<r_2^2r_1
\end{align}
and
\begin{align}\label{N}
d_N(n,n_1)&\leq d_N(n,0_N)+d_N(0_N,n_1)\nonumber\\
&< 2\,r_2^{b_2}+a_1^{1/2}\,r_1^{b_1}\,.
\end{align}
Now we examine two cases separately.

{\it{Case} $r_2\geq r_1$.} In this case, from (\ref{misure}) we deduce that
\begin{align*}
r_2^{b_2}&\leq a_1^{1/2}r_1^{b_1}\,\Big(\frac{\log r_1}{\log r_2}\Big)^{1/{2Q}}\\
&\leq a_1^{1/2}r_1^{b_1}\,,
\end{align*}
and then from (\ref{N}) we obtain that
\begin{align*}
d_N(n,n_1)&< 2\,r_2^{b_2}+a_1^{1/2}r_1^{b_1}\\
&\leq 3\,a_1^{1/2}r_1^{b_1}\,.
\end{align*}
Again from (\ref{misure}) and (\ref{condint}) it follows that
\begin{align*}
r_2^{b_2}&\leq {a_1}^{1/2}r_1^{b_1}\,\Big(\frac{\log r_1}{\log r_2}\Big)^{1/{2Q}}\\
&\leq r_2^{1/2}r_1^{b_1+1/2}\,,
\end{align*}
and then $r_2\leq r_1^{\frac{2b_1 +1}{2b_2-1}}\leq r_1^{\frac{2b'' +1}{2b'-1}}$. Thus, from (\ref{As})
\begin{align*}
\Big(\frac{1}{r_1}\Big)^{\frac{4b''+2b'+1}{2b' -1}}\leq\frac{1}{r_2^2r_1}< \frac{a}{a_1}< r_2^2r_1\leq r_1^{\frac{4b''+2b'+1}{2b' -1}}\,.
\end{align*}
Since $\gamma>\frac{4b''+2b'+1}{2b' -1}>3$ by assumption, we have that
\begin{align*}
&\frac{1}{r_1^{\gamma}}<\frac{a}{a_1}<r_1^{\gamma} \hspace{1cm} \text{and} \hspace{1cm}  d_N(n_1,0_N)<\gamma\,a_1^{1/2}r_1^{b_1}\,.
\end{align*}
Thus the point $(n,a)$ is in $\tR_1$ as required.

{\it{Case} $r_2<r_1$. } In this case, by using (\ref{As}) we have that
$$\frac{1}{r_1^3}< \frac{1}{r_2^2r_1}< \frac{a}{a_1}< r_2^2r_1< r_1^3\,.$$
It remains to verify that $d_N(n,n_1)< \gamma\,a_1^{1/2}r_1^{b_1}$. We examine two situations separately.
\begin{itemize}
\item[(i)] If $r_2<r_1^{\frac{2b_1-1}{2b_2 +1}}$, then from (\ref{N}) we obtain that
\begin{align*}
d_N(n,n_1)&< a_1^{1/2}r_1^{b_1}\,\Big(2\,\frac{r_2^{b_2}}{a_1^{1/2}r_1^{b_1}} +1\Big)\\
&<  a_1^{1/2}r_1^{b_1}\,\Big(2\,\frac{r_2^{b_2}}{r_1^{b_1}}\, r_1^{1/2}r_2^{1/2}+1\Big)\\&< 3 \,a_1^{1/2}r_1^{b_1}\,.
\end{align*}
\item[(ii)] If $r_1^{\frac{2b_1-1}{2b_2 +1}}\leq r_2<r_1$, then from (\ref{misure}) we deduce that
\begin{align*}
r_2^{b_2}&\leq a_1^{1/2}r_1^{b_1}\,\Big(\frac{\log r_1}{\log r_2}\Big)^{1/{2Q}}\\
&\leq a_1^{1/2}r_1^{b_1}\,\Big(\frac{2b_2 +1}{2b_1 -1}\Big)^{1/{2Q}}\\
&\leq a_1^{1/2}r_1^{b_1}\,\Big(\frac{2b'' +1}{2b' -1}\Big)^{1/{2Q}}\,.
\end{align*}
This implies that
\begin{align*}
d_N(n,n_1)&< 2\,r_2^{b_2}+a_1^{1/2}r_1^{b_1}\\
&\leq \left(2\,\Big(\frac{2b'' +1}{2b' -1}\Big)^{1/{2Q}}+1\right)a_1^{1/2}r_1^{b_1}\,.
\end{align*}
\end{itemize}
Since $\gamma>\frac{4b''+2b'+1}{2b'-1}>2\,\Big(\frac{2b'' +1}{2b' -1}\Big)^{1/{2Q}}+1>3$ by assumption, we have proved that
\begin{align*}
&\frac{1}{r_1^{\gamma}}<\frac{a}{a_1}<r_1^{\gamma} \hspace{1cm} \text{and} \hspace{1cm}  d_N(n_1,0_N)<\gamma\,a_1^{1/2}r_1^{b_1}\,.
\end{align*}
Thus the point $(n,a)$ is in $\tR_1$, as required. This concludes the proof of the fact that $\mathcal R^{\infty}$ is nicely ordered.

Finally we prove (iv). Suppose that $R_i=Q^{k_i}_{\alpha_i}\times (a_i/r_i,a_i\,r_i)$, $i=1,2$, are sets in $\mathcal D^{\infty}$ such that $R_2\leq R_1$ and $R_2\cap R_1\neq\emptyset$. In this case, since $k_2\leq k_1$ and $Q^{k_2}_{\alpha_2}\cap Q^{k_1}_{\alpha_1}\neq\emptyset$, by property (vi) of Theorem \ref{dyadic} we have that $Q^{k_2}_{\alpha_2}\subseteq Q^{k_1}_{\alpha_1}$. Moreover, since $R_2$ and $R_1$ intersect, we have that 
\begin{align}\label{coint}
\frac{1}{r_1\,r_2}&\leq\frac{a_1}{a_2}\leq r_1\,r_2\,.
\end{align}
By the admissibility condition (\ref{adm}) we obtain that
$$a_2^{1/2}\,r_2^{\beta}\leq \rd^{k_2}\leq \rd^{k_1}\leq a_1^{1/2}\,r_1^{4\,\beta}\,,$$
and then, by (\ref{coint}), 
$$r_2\leq r_1^{\frac{4\,\beta+1}{\beta\,-1}}\,.$$
It follows that 
\begin{align*}
(a_2/r_2,\,a_2\,r_2)&\subseteq (a_1/{r_1\,r_2^2},\,a_1\,r_1\,r_2^2)\\
&\subseteq (a_1/{r_1^{\gamma}},\,a_1\,r_1^{\gamma})\,.
\end{align*}
This proves that $R_2\subseteq \tR_1$. Thus, the family $\mathcal D^{\infty}$ is nicely ordered.
\end{proof}
A straighforward consequence of Lemma \ref{famnice} is that the maximal operators associated to the families defined above are of weak type $(1,1)$.
\begin{teo}\label{maxopweak}
The maximal operators $M^{\mathcal R^0}$, $M^{\mathcal R^{\infty}}$ and $M^{\mathcal D^{\infty}}$  are bounded from $L^1(\rho)$ to $\lorentz{1}{\infty}{\rho}.$ 
\end{teo}
\begin{proof}
It follows from Theorem \ref{well11} and Lemma \ref{famnice}. 
\end{proof}

\begin{ossn} {\rm{In this section we have proved that the maximal operator associated to the family of balls of small radius is of weak type $(1,1)$. Observe that the maximal operator associated to the family of balls of big radius is not of weak type $(1,1)$. More precisely, let $\mathcal R^{1}=\{B(x_0,\, r):~x_0\in S\,,r\geq 1\}$. If the weak type $(1,1)$ inequality for the maximal operator $M^{\mathcal R^1}$ holds, then it can be extended from $L^1(\rho)$ functions to finite measures by a standard limiting procedure. Let $\delta_e$ be the unit point mass at the identity $e$. At a point $x=(n,a)$, such that $d(x,e)>1$, we have that
$$M^{\mathcal R^1}\delta_e(x)\geq \sup_{r\geq 1}\frac{1}{\rho\big(B(x,r)\big)}\,\delta_e\big(B(x,r)\big)\,.$$
Notice that $\delta_e\big(B(x,r)\big)\neq 0$ if and only if $d(x,e)<r$. Since $\rho\big(B(x,r)  \big)\leq \gamma_2\, a^Q\,\nep^{Qr}$ (where $\gamma_2$ is as in (\ref{misurapalle2})),
\begin{align*}
M^{\mathcal R^1}\delta_e(x)&\geq \sup \Big\{ \frac{1}{\gamma_2\,a^Q\,\nep^{Qr}}:~\,r>d(x,e)\Big\}\\
&=\frac{1}{\gamma_2\,(a\,\nep^{d(x,e)})^Q}\,.
\end{align*}
By estimating the level sets of the function $~\frac{1}{(a\,\nep^{d(x,e)})^Q}$ in the region $\{x=(n,a)\in S:~d(x,e)> 1\}$, we disprove the weak type inequality.

Indeed, suppose that $0<t<\nep^{-2Q}$ and consider the set
\begin{align*}
\Omega_{t}&=\left\{x=(n,a)\in S:~d(x,e)> 1,\,\frac{1}{(a\,\nep^{d(x,e)})^Q}>t\right\}\\
&= \Big\{x=(n,a)\in S:~d(x,e)> 1,\,a\,\nep^{d(x,e)}<\big(1/t)^{1/Q}\Big\}\,.
\end{align*}
We prove that $\rho(\Omega_{t})=\infty$.

First we observe that the set $O_t$ defined by
$$O_t=\Big\{ x=(n,a)\in S:~0<a<1,\,4\,\cosh ^{2}(1/2)<\mathcal N(n)^4<1/{c_2^2+1}\big[\big(1/t)^{1/Q}-4\big]\Big\}\,, $$
where $c_2$ is as in Proposition \ref{normaN}, is contained in $\Omega_t$. 

Indeed, if $(n,a)=(X,Z,a)$ is in $O_t$, then 
\begin{align*}
\cosh^2\Big(d(x,e)/2 \Big)&=a^{-1}\,\Big[\Big((a+1)/2+1/8\,|X|^2\Big)^2+1/4\,|Z|^2 \Big] \\
&>(1/4)\,\Big(|X|^4/16+|Z|^2\Big)\\
&=\mathcal N(n)^4/4\\
&>\cosh^2(1/2)\,.
\end{align*}
Thus $d(x,e)> 1$ and
\begin{align*}
a\,\nep^{d(x,e)}&<4\,a\,\cosh ^2\big( d(x,e)/2\big)\\
&=4\,\Big[\Big((a+1)/2+1/8\,|X|^2\Big)^2+1/4\,|Z|^2 \Big]\\
&<4\,\Big[\Big(1+|X|^2/8\Big)^2+|Z|^2/4 \Big]\\
&<4+|X|^2+\mathcal N(X,Z)^4\\
&<4+(c_2^2+1)\,\mathcal N(X,Z)^4\\
&<\big(1/t)^{1/Q}\,.
\end{align*}
This proves that $O_t\subseteq \Omega_t$. Since $\rho(O_t)=+\infty$, there does not exist a constant $C$ such that $\rho(\Omega_t)\leq C/t\,.$ 

This shows that the maximal operator $M^{\mathcal R^1}$ is not of weak type $(1,1)$.}}
\end{ossn}

\begin{ossn} {\rm{Note that the parameter $b$ which appear in the Definition \ref{Rinftydef} of sets of the family $\mathcal R^{\infty}$ is bounded away from both $1/2$ and $\infty$. If this does not hold, then the corresponding maximal operator is not of weak type $(1,1)$. More precisely let $\tilde{\mathcal R}$ be the family
$$\tilde{\mathcal R}=\{x_0E_{r,\,b}:~x_0\in S,~r\geq \nep,\,b>1/2\}\,.$$
The maximal operator $M^{\tilde{\mathcal R}}$ is not of weak type $(1,1)$. 

Indeed, if the weak type $(1,1)$ inequality holds, then it can automatically be extended from $L^1(\rho)$ functions to finite measures. Let $\delta_e$ be the unit point mass at the identity $e$. At a point $x=(n,a)$ we have that
$$M^{\tilde{\mathcal R}}\delta_e(x)\geq \sup_{r\geq \,{\textrm e},\,b>1/2}\frac{1}{\rho\big(xE_{r,\,b}\big)}\,\delta_e(xE_{r,\,b})\,.$$
Note that $\delta_e(xE_{r,\,b})\neq 0$ if and only if $\frac{a}{r}<1<a\,r$ and $d_N(n,0_N)<a^{1/2}r^{b}$.

Since $\rho\big(xE_{r,\,b}  \big)=a^Qr^{2b Q}\log r$,
\begin{align*}
M^{\tilde{\mathcal R}}\delta_e(x)\geq \sup &\Big\{ \frac{1}{a^Qr^{2b Q}\log r}:~b>1/2,\,r\geq {\rm{e}},\,\\
&\frac{a}{r}<1<a\,r,\,a^{1/2}r^{b}>d_N(n,0_N  ) \Big\} \,.
\end{align*}
Now suppose that $a>\textrm e$ and $d_N(n,0_N)>a$. We may choose $b=\log_ad-1/2>1/2$ and $r=a\Big(1+\frac{1}{b}\Big)$. Obviously $r>a$ and $a^{1/2}r^{b}>a^{b+1/2}=d_N(n,0_N)$.

Moreover,
\begin{align*}
a^Qr^{2b Q}\log r&=a^{2Q(b+1/2)}\Big(1+\frac{1}{b}\Big)^{2Q}\log \Big(a+\frac{a}{b}\Big)\\
&\leq C\big[d_N(n,0_N)\big]^{2Q}\log a\,.
\end{align*}
It follows that
$$M^{\tilde{\mathcal R}}\delta_e(x)\geq  \frac{1}{C\big[d_N(n,0_N)\big]^{2Q}\log a}\,.$$
We shall disprove the weak type inequality by estimating the level sets of the function $~~\frac{1}{\big[d_N(n,0_N)\big]^{2Q}~\log a}~~$ in the region $\{x=(n,a)\in S:~a>\textrm e,\,d_N(n,0_N)>a\}$.

Indeed suppose that $0<t<\textrm e^{-2Q}$ and consider the set
\begin{align*}
\Omega_{t}&=\left\{(n,a)\in S:~a>\textrm e,\,d_N(n,0_N)>a,\,\frac{1}{\big[d_N(n,0_N)\big]^{2Q}~\log a}>t\right\}\\
&= \Big\{(n,a)\in S:~\textrm e<a<\alpha,\,a<d_N(n,0_N)<\frac{1}{(t\log a)^{1/{2Q}}}\Big\}\,,
\end{align*}
where $\alpha^{2Q}\log \alpha=\frac{1}{t}$. Then 
\begin{align*}
\rho(\Omega_{t})&=\int_{\textrm e}^{\alpha}\frac{\di a}{a}\int_{a^2}^{\frac{1}{(t\log a)^{1/Q}}}\sigma^{Q-1}\di\sigma\\
&=\frac{1}{Q}\int_{\nep}^{\alpha}\Big(\frac{1}{t\log a}-a^{2Q}\Big)\frac{\di a}{a}\\
&=\frac{1}{Qt}\log \log\alpha-\frac{1}{2Q^2}\alpha^{2Q}+\frac{1}{2Q^2}\textrm e^{2Q}\\
&\geq \frac{1}{Qt}\log \log\alpha-\frac{1}{2Q^2}\frac{1}{t\log\alpha}\\
&\geq\frac{1}{Qt}\Big(\log \log\alpha-\frac{1}{2Q}\Big)\,,
\end{align*}
where we have used the integration formula for radial functions on $N$ in Proposition \ref{integrationN}.

It is easy to check that $\alpha>t^{-1/(4Q)}$, and then
\begin{align*}
\rho(\Omega_t)&\geq\frac{1}{Qt} \Big(\log \log\Big(\frac{1}{t^{1/{(4Q)}}}\Big)-\frac{1}{2Q}\Big)\\
& \geq\frac{1}{Qt} \Big(\log \log\Big(\frac{1}{t}\Big)-\frac{1}{2Q}\Big)\,,
\end{align*}
which is not bounded above by $\frac{C}{t}$.  Thus, the weak type inequality for the maximal operator $M^{\tilde{\mathcal R}}$ does not hold.}}
\end{ossn}

\section{Equivalence between $M^{\mathcal R^{\infty}}$ and $M^{\mathcal D^{\infty}}$}\label{equivalence}
Given two families $\mathcal F_1$ and $\mathcal F_2$ of open subsets of $S$ we say that the associated maximal operators are {\it{equivalent}} if $M^{\mathcal F_1}$ is of weak type $(1,1)$ if and only if $M^{\mathcal F_2}$ is of weak type $(1,1)$.

In this section we prove that the maximal operators $M^{\mathcal R^{\infty}}$ and $M^{\mathcal D^{\infty}}$ are equivalent. First we prove that for every locally integrable function $f$ the maximal function $M^{\mathcal D^{\infty} }f$ is pointwise bounded above by $C\, M^{\mathcal R^{\infty} }f$, for an appropriate constant $C$. 
\begin{prop}\label{unsenso}
Let $\beta$ be as in Definition \ref{admissibility}, $b'$ and $b''$ as in Definition \ref{Rinftydef}. If $b'=\beta$ and $b''=
4\,\beta+\log C_N$, then 
$$ M^{\mathcal D^{\infty} }f(x)\leq \frac{C_N}{c_N}\,M^{\mathcal R^{\infty}}f(x)\qquad \forall f\in L^1_{\rm{loc}}(\rho)\quad\forall x\in S\,.$$
\end{prop}   
\begin{proof}
Let $f$ be in $L^1_{\rm{loc}}(\rho)$ and $x$ in $S$. Let $E=\Qka\times (a_0/r,a_0\,r)$ be a set in $\mathcal D^{\infty}$ which contains $x$. We define
$$R=B_N(\nka,\,C_N\,\rd^k)\times (a_0/r,a_0\,r)\,,$$
where $C_N,\,\rd$ are as in Theorem \ref{dyadic}. It is easy to check that $E\subset R$ and $\rho(R)\leq \big(C_N/c_N\big)\,\rho(E)\,.$ We choose $b$ such that $a_0^{1/2}\,r^{b}=C_N\,\rd^k\,.$

By the admissibility of $E$ we deduce that 
$$C_N\,a_0^{1/2}\,r^{\beta}\leq a_0^{1/2}\,r^b<C_N\,a_0^{1/2}\,r^{4\beta}\,,$$
so that
$$b'=\beta\leq b< 4\,\beta+\log C_N=b''\,.$$
Thus $R$ is in $\mathcal R^{\infty}$ and
$$\frac{1}{\rho(E)}\int_E|f|\dir \leq \frac{C_N}{c_N}\,\frac{1}{\rho(R)}\int_{R}|f|\dir\leq\frac{C_N}{c_N}\, M^{\mathcal R^{\infty}}f(x)\,.$$
By taking the supremum over all set $E$ in $\mathcal D^{\infty}$ which contain $x$, we obtain that
$$ M^{\mathcal D^{\infty} }f(x)\leq \frac{C_N}{c_N}\,M^{\mathcal R^{\infty}}f(x)\,,$$
as required.
\end{proof}
Clearly, by Proposition \ref{unsenso}, if the maximal operator $M^{\mathcal R^{\infty} }$ is of weak type $(1,1)$, then $M^{\mathcal D^{\infty}}$ is of weak type $(1,1)$.

Next we show that a pointwise inequality of the form $M^{\mathcal R^{\infty}}\leq C\,M^{\mathcal D^{\infty}}$, with $C>0$, fails.

\bigskip
{\bf{A counterexample.}} Let $S=\RR\times \RR^+$ be the affine group of the real line. Set $f=\chi_{(0,1)\times (\nep^{-1/2},\,\nep^{3/2})}$. We compute $M^{\mathcal D^{\infty}}f$ and $M^{\mathcal R^{\infty}}f$ on the set $(-1,0)\times(\nep^{-1/2},\nep^{3/2}) $.

Let $(x,a)$ be a point in $(-1,0)\times(\nep^{-1/2},\nep^{3/2}) $ and $E=Q_m^k \times (a_0/r,a_0\,r)$ be a set in $\mathcal D^{\infty}$ which contains $(x,a)$. Since $x\in Q_m^k$, $Q_m^k\cap (0,1)=\emptyset$. Thus the average of $f$ on $E$ is equal to zero. This proves that $M^{\mathcal D^{\infty}}f(x,a)=0$.

Let us consider a constant $b>1/2$ and the set $R=(-\nep^{b},\nep^{b})\times (\nep^{-1/2},\nep^{3/2})$. This is a set in ${\mathcal R^{\infty}}$ centred at $(0,\nep^{1/2})$ which contains $(x,a)$. The average of $f$ on $R$ is 
$$\frac{1}{\rho(R)}\int_R|f|\dir=\frac{1}{4\,\nep^{b}}\,\rho\big( (0,1)\times(\nep^{-1/2},\nep^{3/2}  \big)= \frac{1}{2\,\nep^{b}}\,.$$
Thus, there does not exist a constant $C$ such that
$$\frac{1}{2\,\nep^{b}}\leq M^{\mathcal R^{\infty}}f(x,a)\leq C\,M^{\mathcal D^{\infty}}f(x,a)=0\qquad\forall (x,a)\in(-1,0)\times(\nep^{-1/2},\nep^{3/2}) \,.$$

\bigskip
Though the pontwise inequality between $M^{\mathcal R^{\infty}}$ and $M^{\mathcal D^{\infty}}$ does not hold, an inequality between the measures of the level sets of $M^{\mathcal R^{\infty}}$ and $M^{\mathcal D^{\infty}}$ holds. This is proved in the following proposition.
\begin{prop}\label{altrosenso}
Let $\beta$ be as in Definition \ref{admissibility}, $b'$ and $b''$ as in Definition \ref{Rinftydef}. If $b'=\beta$ and $b''=4\beta-\log \rd$, then there exists a constant $A$ such that for every locally integrable function $f$ and $t>0$
\begin{align}\label{converse}
\rho\big(\{x:~M^{\mathcal R^{\infty} }f(x)>A\,t\}\big)&\leq \Big(\frac{4C_N+2}{c_N}  \Big)^{2Q}(\gamma+1)\,\rho\big(\{x:~M^{\mathcal D^{\infty} }f(x)>t\}\big)\,.
\end{align}
\end{prop}
\begin{proof}
Let $f$ be in $L^1_{\rm{loc}}(\rho)$. Set $O^{\mathcal D^{\infty}}_{t}=\{x:~M^{{\mathcal D}^{\infty}}f(x)>t  \}$. For each point $x$ in $ O^{\mathcal D^{\infty}}_{t}$ we choose a set $E_x$ in $\mathcal D^{\infty}$ such that $\frac{1}{\rho(E_x)}\int_{E_x}|f|\dir >t$ and $E_x\geq E$, for every set $E$ which contains $x$ such that $\frac{1}{\rho(E)}\int_{E}|f|\dir >t$. 

Now we select a disjoint subfamily of $\{E_x\}_x$.

We choose $E_{x_1}$ such that $E_{x_1}\geq E_x$, for all $x\in O^{\mathcal D^{\infty}}_{t} \,.$ Next, suppose that $E_{x_1},...,E_{x_n}$ have been chosen. Then we choose a set $E_{x_{n+1}}$ disjoint from $E_{x_1},...,E_{x_n}$ and such that $E_{x_{n+1}}\geq E_x$, for all $x\in O^{\mathcal D^{\infty}}_{t} $ such that $E_x\cap E_{x_i}=\emptyset,\,i=1,...,n\,.$

We put $E_j=E_{x_j}=Q_{\alpha_j}^{k_j}\times (a_j/r_j,a_j\,r_j)$, $\tE_j= Q_{\alpha_j}^{k_j}\times (a_j/{r_j^{\gamma}},a_j\,r_j^{\gamma})$ and $\overline{E}_j=B_N\big(n_{\alpha_j}^{k_j},(4C_N+2)\rd^{k_j}\big)\times (a_j/{r_j^{\gamma+1}},a_j\,r_j^{\gamma+1})$. 

Now set $O^{\mathcal R^{\infty}}_{A\,t}=\{x:~M^{{\mathcal R}^{\infty}}f(x)>A\,t \}$. We prove that $O^{\mathcal R^{\infty}}_{A\,t}\subseteq \bigcup_j\overline{E}_j$. 

Let $x$ be in $O^{\mathcal R^{\infty}}_{A\,t}$. There exists a set $R=B_N(n_0,a_0^{1/2}\,r^b)\times (a_0/r,a_0r)$ in $\mathcal R^{\infty}$ which contains $x$ such that $\frac{1}{\rho(R)}\int_R|f|\dir>A\,t\,.$ We choose the integer $k$ such that $\rd^{k-1}< a_0^{1/2}\,r^b\leq\rd^k$ and denote by $I$ the set of indeces $I=\{\alpha\in I_k:~\Qka\cap B_N(n_0,a_0^{1/2}\,r^b)\neq \emptyset  \}$. We claim that the cardinality of $I$ is at most $L$, where $L$ is a fixed number which depends only on the group $N$. 

For all $\alpha$ in $I$ we define $R_{\alpha}=\Qka\times (a_0/r_0,a_0\,r_0)$. Note that $R\subseteq \bigcup_{\alpha\in I}R_{\alpha}$ and $\rho(R_{\alpha})\leq (C_N\,\rd)^{2Q}\,\rho(R)$. Each set $R_{\alpha}$ is in $\mathcal D^{\infty}$. Indeed,
$$a_0^{1/2}\,r^{\beta}\leq a_0^{1/2}\,r^{b'}\leq a_0^{1/2}\,r^{b}\leq\rd^k\,,$$
and
$$\rd^k< \rd\, a_0^{1/2}\,r^{b}<\rd\, a_0^{1/2}\,r^{b''}<a_0^{1/2}\,r^{4\beta}\,.$$
Now observe that
\begin{align*}
A\,t&< \frac{1}{\rho(R)}\int_R|f|\dir\\
&\leq \sum_{\alpha\in I}\frac{\rho(R_{\alpha})}{\rho(R)}\,\frac{1}{\rho(R_{\alpha})}\int_{R_{\alpha}}|f|\dir\\
&\leq (C_N\,\rd)^{2Q}\sum_{\alpha\in I}\frac{1}{\rho(R_{\alpha})}\int_{R_{\alpha}}|f|\dir\,.
\end{align*}
Thus there exists at least an index $\alpha_0$ in $I$ such that $\frac{1}{\rho(R_{\alpha_0})}\int_{R_{\alpha_0}}|f|\dir>\frac{A\,t}{L\,(C_N\,\rd)^{2Q}}$. By choosing $A=L\,(C_N\,\rd)^{2Q}$, we have that  $\frac{1}{\rho(R_{\alpha_0})}\int_{R_{\alpha_0}}|f|\dir>t$. It follows easily that there exists an index $j_0$ such that $R_{\alpha_0}\leq E_{j_0}$ and $R_{\alpha_0}\cap E_{j_0}\neq \emptyset$. By Lemma \ref{famnice} (iv) we have that $k\leq k_{j_0}$, $r\leq r_{j_0}^{\gamma}$ and $R_{\alpha_0}\subseteq \tE_{j_0}$. 

Now we prove that $x=(n,a)$ is in $\overline{E}_j$.

Let $\alpha_x$ be the index in $I$ such that $x\in R_{\alpha_x}$. It follows that
$$\frac{a}{a_{j_0}}= \frac{a}{a_0}\frac{a_0}{a_{j_0}}<r\,r_{j_0}\leq r_{j_0}^{\gamma +1}\,,$$
and
$$\frac{a}{a_{j_0}}>\frac{1}{r_{j_0}^{\gamma +1}}\,.  $$
Moreover
\begin{align*}
d_N(n,n^{k_{j_0}}_{\alpha_{j_0}})&<d_N(n,n^{k}_{\alpha_x})+d_N(n^{k}_{\alpha_x},n^k_{\alpha_0})+d_N(n^k_{\alpha_0},n^{k_{j_0}}_{\alpha_{j_0}})\\
&\leq C_N\,\rd^k+d_N(n^{k}_{\alpha_x},n_{0})+d_N(n_0,n^{k}_{\alpha_0})+C_N\,\rd^{k_{j_0}}\\
&\leq C_N\,\rd^{k_{j_0}}+2(C_N\rd^k+a_0^{1/2}r^b)+C_N\,\rd^{k_{j_0}}\\
&\leq C_N\,\rd^{k_{j_0}}+2(C_N+1)\rd^k+C_N\,\rd^{k_{j_0}}\\
&\leq (4C_N+2)\rd^{k_{j_0}}\,.
\end{align*} 
This proves that $x\in \overline{E}_{j_0}$. %, where $$\tilde{E}_{j_0}=B_N\big(n_{\alpha_{j_0}}^{k_{j_0}},(4C_N+2)\rd^{k_{j_0}}  \big)\times (a_{j_0}/{r_{j_0}^{\gamma +1}},a_{j_0}\,{r_{j_0}^{\gamma +1}})\,.$$

Thus $O^{\mathcal R^{\infty}}_{A\,t}\subseteq \bigcup_j \overline{E}_j$ and 
\begin{align*}
\rho(O^{\mathcal R^{\infty}}_{A\,t})&\leq\sum_j\rho(\overline{E}_j)\\
&\leq\Big(\frac{4C_N+2}{c_N}  \Big)^{2Q}(\gamma+1)\,\sum_j\rho(E_j)\\
&\leq \Big(\frac{4C_N+2}{c_N}  \Big)^{2Q}(\gamma+1)\,\rho(O^{\mathcal D^{\infty}}_{t})\,,
\end{align*}
as required.

It remains to prove that the cardinality of the index set $I$ is at most $L$. 

Indeed, for all $\alpha\in I$ we have that
$$d_N(\nka,n_0)<C_N\,\rd^k+ a_0^{1/2}\,r^b<(C_N+1)\rd^k\,.$$
Given two indeces $\alpha,\,\beta\in I$ we have that $d_N(\nka,n^k_{\beta})>\rd^k\,.$ Thus points $\nka$, $\alpha\in I$, are in the ball $B_N(n_0,(C_N+1)\rd^k)$ and their mutually distances are greater than $\rd^k$. Since $N$ is a space of homogeneous type, there exists a constant $L$ such that $\sharp I\leq  L$ \cite{CW}.
\end{proof}
By Proposition \ref{unsenso} and \ref{altrosenso} we deduce that given $\beta$ as in Definition \ref{admissibility}, $b',\,b''$ as in Definition \ref{Rinftydef}, if $b'=\beta$ and $b''=4\beta+\log C_N$, then the maximal operators $M^{\mathcal R^{\infty}}$ and $M^{\mathcal D^{\infty}}$ are equivalent.

\section{Maximal operators on products of \DR spaces}\label{maxoperatorsp}
Let $S=S'\times S''$ be the product of two \DR spaces. It is natural to consider maximal operators of Hardy--Littlewood type on $S$ which are the analogue of the maximal operators which we studied in the nonproduct case. 

Given a family $\mathcal F$ of open subsets of $S$ the associated maximal operator $M^{\mathcal F}$ is defined by
$$M^{\mathcal F}f(x)=\sup _{x\in F}\int_F|f|\dir\,,$$
where the supremum is taken over all the sets $F$ in the family $\mathcal F$. 

We now introduce two families of sets and study the weak type $(1,1)$ boundedness of the associated maximal operators. 

The family ${\mathcal R}^0$ is the family of balls of small radius, i.e. 
$$\mathcal R^0=\{ B_{d_{max}}\big((x',x''),r\big)=B_{S'}(x',r)\times B_{S''}(x'',r):~ r<1/2\}\,. $$
By (\ref{misurapalle1p}) there exist positive constants $\gamma_1$, $\gamma_2$ such that
$$\gamma_1 \,(a_0')^{Q'}(a_0'')^{Q''}\,r^n\leq \rho\big(B\big((n_0',a_0')(n_0'',a_0''),r\big)\big)\leq \gamma_2 \,(a_0')^{Q'}(a_0'')^{Q''}\,r^n\qquad \forall r\in (0,1/2)\,.$$
Set $\gamma=1+2\Big(2\,\nep^{Q'+Q''}\,\gamma_2/{\gamma_1}\Big)^{1/n}$ and $\tB(x_0,r)=B(x_0,\gamma \,r)$.

\bigskip
The family $\mathcal R^{\infty}$ consists of products of ``big sets'' in $S'$ and $S''$, i.e.
\begin{equation}\label{Rinfp}
\mathcal R^{\infty}=\{ (x',x'')\cdot (E'_{r,\,b}\times E''_{r,\,b}):~(x',x'')\in S,\,r\geq\nep\} \,,
\end{equation}
where sets $E'_{r,\,b}$, $E''_{r,\,b}$ are defined as in (\ref{Erb}) and $b$ is a constant $>1/2$.

The family $\mathcal R^0$ si nicely ordered as in the nonproduct case.

\begin{lem}\label{coverp}
The family $(\mathcal R^0, \leq )$, where $B_2\leq B_1$ if $\rho(B_2)\leq 2\,\rho(B_1)$ and the dilated of a ball $B$ is the ball $\tB$ defined above, is nicely ordered. 
\end{lem}
\begin{proof}
We may suppose without loss of generality that $B_1$ is centred at the identity. Let $r_i$ denote the radius of the ball $B_i$, $i=1,2$ and $(x_i',x_i'')$ denote the centre of $B_i$. Since $B_1\cap B_2\neq \emptyset$, we have that $d_{S'}(e,x_2')<r_1+r_2<1$ and $d_{S''}(e,x_2'')<r_1+r_2<1$. The condition $\rho(B_2)\leq 2\,\rho(B_1)$ implies that $\gamma_1\,\delta(x_2)^{-1}\,r_2^n\leq 2\,\gamma_2\,r_1^n$. Thus
\begin{align*}
r_2&\leq \big(2\,\delta(x_2)\,\gamma_2/{\gamma_1}\big)^{1/n}\,r_1\,.
\end{align*}
Since $x_2$ is in $B(e,1)$ we have that $\delta(x_2)\leq \nep^{Q'+Q''}$ and then 
\begin{align*}
r_2&\leq \big(2\,\nep^{Q'+Q''}\,\gamma_2/{\gamma_1}\big)^{1/n}\,r_1\,.
\end{align*}
Hence
\begin{align*}
B(x_2,r_2)&\subseteq B(e,2r_2+r_1)\\
&\subseteq B\Big(e,\left(1+2\big(2\,\nep^{Q'+Q''}\,\gamma_2/{\gamma_1}\big)^{1/n}\Big)r_1\right)\\
&=B(e,\gamma \,r_1)\\
&=\tB_1\,,
\end{align*}
as required.
\end{proof}
{F}rom Lemma \ref{coverp} it follows that $M^{{\mathcal R}^0}$ is bounded from $\lu{\rho}$ to $\lorentz{1}{\infty}{\rho}$.

While in the nonproduct case $M^{\mathcal R^{\infty}}$ is of weak type $(1,1)$, this does not hold in the product case. 

\bigskip 
{\bf{A counterexample.}} We consider the space $S=\RR^2\times \RR^2_+=(\RR\times \RR^+)\times (\RR\times \RR^+)$ endowed with the semidirect product
$$(x',x'',a',a'')\cdot(y',y'',\alpha',\alpha'')=(x'+a^{'\,1/2}\,y',\,x''+a^{''\,1/2}\,y'',\,a'\,\alpha',\,a''\,\alpha'')\,.$$
The right Haar measure on $S$ is given by $\dir(x',x'',a',a'')=(a'\,a'')^{-1}\di x'\di x''\di a'\di a''$. 

Suppose that $b>1/2$ and let $Q$ be the set $(-\nep^{b},\nep^{b})\times (-\nep^{b},\nep^{b})\times (1/{\nep},\nep)\times (1/{\nep},\nep)$. We consider the family of sets
$$\mathcal F=\{(x',x'',a',a'')\cdot Q:~(x',x'')\in\RR^2, (a',a'')=(2^{2i},2^{-4i}), i\geq 0\}\,.$$
We want to prove that the maximal operator $M^{\mathcal F}$ is not of weak type $(1,1)$. To see this we define a sequence of functions $\{f_k\}_k$ such that $\|f_k\|_{L^1(\rho)}=16\,\nep^{2\beta}$ and $\rho(\{M^{\mathcal F}f_k>1/2\})\geq \|f_k\|_{L^1(\rho)}\, (k+1)$.

First we define a sequence $\{E_k\}_k$ of sets in $\RR$ such that $E_k$ is the union of $2^k$ intervals whose length is $2\cdot2^{-2k}\,\nep^{\beta}$. Set $E_0=(-\nep^{b},\,\nep^{b})$. Given a set $E_{k-1}$, for each component $I=(a-4\cdot 2^{-2k}\,\nep^{b},a+4\cdot 2^{-2k}\,\nep^{b})$ of $E_{k-1}$, we take two subintervals of $I$, $I_1=( a-4\cdot 2^{-2k}\,\nep^{b},a-2\cdot 2^{-2k}\,\nep^{b})$, $I_2=( a+2\cdot 2^{-2k}\,\nep^{b},a+4\cdot 2^{-2k}\,\nep^{b})$ and let $E_k$ be the union of these intervals. Then we have $E_k\subset E_{k-1}\subset...\subset E_1\subset E_0$ and $|E_k|=2\cdot 2^{-k}\,\nep^{b}$. 

We now define a sequence of functions $f_k$ as 
$$f_k=2^{k}\,\chi_{(-\nep^{b},\,\nep^{b})\times E_k\times (1/{\nep},\nep)\times (2^{-4k}/{\nep},2^{-4k}\,\nep)}\,.$$
Then $\|f_k\|_{L^1(\rho)}=16\,\nep^{2b}$. We will show that $\{M^{\mathcal F}f_k> 1/2 \}\supseteq \bigcup_{i=0}^{k} (-2^i\,\nep^{b},2^i\,\nep^{b})\times E_i\times (2^{2i}/{\nep}, 2^{2i}\,\nep)\times (2^{-4i}/{\nep},2^{-4i}\,\nep)$.

Fix $k\geq 0$ and let $I$ be one of the components of $E_i$, $i\leq k$. Let $R=(-2^{i}\,\nep^{b},2^{i}\,\nep^{b})\times I\times (2^{2i}/{\nep},2^{2i}\,\nep)\times (2^{-4i}/{\nep},2^{-4i}\,\nep)$. Obviously $R$ belongs to the family $\mathcal F$ and $\rho(R)=16\,2^{-i}\,\nep^{2b}$. Moreover
\begin{align*}
\frac{1}{\rho(R)}\int_R|f_k|\dir&=\frac{2^i}{16\,\nep^{2b}}\,2^{k}\,|(-\nep^{b},\nep^{b})|\,|I\cap E_k|\,4\\
&=\frac{1}{4}\,2^{k+i}\,2\,2^{k-i}\,2\,2^{-2k}\\
&=1\,.
\end{align*}
Then $R\subseteq \{M^{\mathcal F}f_k>1/2 \}$. It follows that $\{M^{\mathcal F}f_k>1/2 \}\supseteq \bigcup_{i=0}^{k} (-2^i\,\nep^{b},2^i\,\nep^{b})\times E_i\times (2^i/{\nep},2^i\,\nep)\times (2^{-2i}/{\nep},2^{-2i}\,\nep)$. Thus
\begin{align*}
\rho(\{M^{\mathcal F}f_k>1/2 \})&\geq\rho\Big( \bigcup_{i=0}^{k}(-2^i\,\nep^{b},2^i\,\nep^{b})\times E_i\times (2^i/{\nep},2^i\,\nep)\times (2^{-2i}/{\nep},2^{-2i}\,\nep)   \Big)\\
&\geq\rho\Big( \bigcup_{i=0}^{k-1}   (-2^i\,\nep^{b},2^i\,\nep^{b})\times E_i\times (2^i/{\nep},2^i\,\nep)\times (2^{-2i}/{\nep},2^{-2i}\,\nep)    \Big)+\\
&+2\,\rho\big((2^{k-1}\,\nep^{b},2^k\,\nep^{b})\times E_k\times (2^k/{\nep},2^k\,\nep)\times (2^{-2k}/{\nep},2^{-2k}\,\nep)\big)\\
&\geq \sum_{i=0}^{k-1}2\,(2^i-2^{i-1})\,\nep^{b}\,|E_i|\,4+2\,(2^k-2^{k-1})\,\nep^{b}\,|E_k|\,4=\\
&=16\,\nep^{2b}\,(k+1)\\
&=\|f_k\|_{L^1(\rho)}\,(k+1)\,.
\end{align*}
Hence $M^{\mathcal F}$ is not of weak type $(1,1)$. 
 
Note that $\mathcal F$ is contained in the family $\mathcal R^{\infty}$ defined by (\ref{Rinfp}). Then also the maximal operator $M^{\mathcal R^{\infty}}$ is not of weak type $(1,1)$.

\chapter{\CZ theory}\label{CZdec}
\begin{intro*}
Recently Hebisch and Steger gave an assiomatic definition of \CZ spaces and proved a boundedness theorem for singular integral operators on these spaces: we recall their results in Section \ref{CZdecgen}. In a \CZ space, satisfying an additional hypothesis of technical nature, we may define Hardy spaces $H^{1,q}$, $1<q<\infty$, and $BMO_p$ spaces, $1<p<\infty$: we prove that $BMO_p$ may be identified with the topological dualof $H^{1,p'}$. Then we prove a $H^{1,q}-L^1$ boundedness theorem for integral operators on \CZ spaces.

In Section \ref{CZdecextH} we generalize the result in \cite{HS} proving that $(S,\rho,d)$ is a \CZ space for every \DR space. To do it we shall use a family of {\it{admissible sets}}: it may be worth observing that ``small sets'' are balls of small radius, 
while ``big sets'' are rectangles, i.e. products of dyadic sets in $N$ and intervals in $A$.

In Section \ref{Hardyax+b} we study the $H^1-BMO$ theory on $ax+b\,$-groups: in this case we introduce a $H^{1,\infty}$ space and prove that all spaces $H^{1,q}$, $1<q\leq \infty $, are equivalent.

Finally in Section \ref{CZdecompositionp} we show that products of \DR spaces are \CZ spaces.
\end{intro*}
\section{General \CZ theory}\label{CZdecgen}
Recently Hebisch and Steger \cite{HS} gave the following assiomatic definition of Cal\-der\'on--Zygmund space.
\begin{defi}\label{CZdef}
Let $(X,\mu,d)$ be a metric measured space. Let $\mathcal R$ be a family of sets in $X$ and $\kappa_0$ be a positive constant. We say that $(X,\mu,d)$ is a {\emph{Calder\'on--Zygmund space}} with \CZ constant $\kappa_0$ if the following hold:
\begin{itemize}
\item [(i)] for every set $R$ in $\mathcal R$ there exists a positive number $r$ such that $R$ is contained in a ball of radius at most $\kappa_0 \,r$ and $\mu(R^*)\leq \kappa_0\,\mu(R)$, where $R^*=\{x\in X~:~d(x,R)< r\}$;
\item [(ii)] for every $f$ in $L^1(\mu)$ and $\alpha>\kappa_0\,{\|f\|_{L^1(\mu)}}/{\mu(X)}$ ($\alpha>0$ if $\mu(X)=\infty$) there exists a decomposition $f=g+\sum_{i\in \NN} b_i$ such that
\begin{itemize}
\item [(ii1)]$|g|\leq \kappa_0\,\alpha\qquad \mu$-almost everywhere;
\item [(ii2)] $b_i$ is supported in a set $R_i$ of $\mathcal R$ and $\int b_i \di\mu=0\qquad \forall i\in \NN$;
\item[(ii3)] $\|b_i\|_{L^1(\mu)}\leq \kappa_0\,\alpha\,\mu(R_i)\qquad\forall i\in\NN$;
\item[(ii4)] $\sum_i \mu(R_i)\leq \kappa_0\,\frac{\|f\|_{L^1(\mu)}}{\alpha}\,.$
%\item[(ii4)] $\|b_i\|_{L^1(\mu)}\leq \kappa_0\,\alpha\,\mu(R_i)\qquad\forall i\in\NN$.
%\item[(ii4)]$\sum_i {\|b_i\|_{L^1(\mu)}}\leq \kappa_0\,\|f\|_{L^1(\mu)}\,.$ $\|b_i\|_{L^1(\mu)}\leq \kappa_0\,\alpha\,\mu(R_i)$.
\end{itemize}
\end{itemize}
The sets in the family $\mathcal R$ are called \emph{\CZ sets} and the decomposition $f=g+\sum_{i\in \NN} b_i$ is called \emph{\CZ decomposition of $f$ at height $\alpha$}.  
\end{defi}
Clearly spaces of homogeneous type are \CZ spaces. Note that in this case we may choose $\mathcal R$ as the family of balls. It is remarkable that some spaces which are not of homogeneous type are \CZ spaces. Indeed, Hebisch and Steger proved that $ax+b\,$-groups are \CZ spaces. They proved a boundedness theorem for integral operators on \CZ spaces \cite[Theorem 2.1]{HS}. We now give a formulation of their theorem, where the hypothesis is reminiscent of the classical H\"ormander's condition.
\begin{teo}\label{Teolim}
Let $(X,\mu,d)$ be a Calder\'on--Zygmund space.
Let $T$ be a linear operator which is bounded on $L^2(\mu)$ and admits a locally integrable kernel $K$ off the diagonal that satisfies the condition 
\begin{align}\label{stimaH}
\sup_R \sup_{y,\,z\in R}\int_{(R^*)^c}|K(x,y)-K(x,z)| \,{\rm{d}}\mu (x) &<\infty\,,
\end{align}
where the supremum is taken over all \CZ sets $R$ in $\mathcal R$. 
Then $T$ extends from $L^1(\mu)\cap L^2(\mu)$ to a bounded operator from $\lu{\mu}$ to $\lorentz{1}{\infty}{\mu}$ and on $L^p(\mu)$, for all $p$ in $ (1,2]$.
\end{teo}
\begin{ossn}\label{condHS}{\rm{It is easy to verify that if $T$ is a linear operator bounded on $L^2(\mu)$ such that $T=\sum_{j\in \ZZ}T_j$, where
\begin{itemize}
\item [(i)] the series converges in the strong topology of $L^2(\mu)$;
\item [(ii)] every $T_j$ is an integral operator with kernel $K_j$;
\item [(iii)] there exist positive constants $a,A,\varepsilon$ and $c>1$ such that
\begin{align}\label{stime1}
\int_X|K_j(x,y)|\,\big(1+c^j d(x,y)\big)^{\varepsilon}\,{\rm{d}}\mu (x) &\leq A\qquad\forall y\in X;
\end{align}
\begin{align}\label{stime2}
\int_X|K_j(x,y)-K_j(x,z)| \,{\rm{d}}\mu (x) &\leq A\,\big(c^j d(y,z)\big)^a\qquad\forall y,z\in X\,,
\end{align}
\end{itemize}
then $T$ satisfies the hypothesis (\ref{stimaH}) of Theorem \ref{Teolim}. The conditions (i)-(iii) are formulated by Hebisch and Steger in \cite[Theorem 2.1]{HS} and in some case they are more convenient than (\ref{stimaH}) to verify (as in the proof of Theorem \ref{moltiplicatori}). }}
\end{ossn}

\begin{ossn}\label{dualcondition}{\rm{Note that if the operator $T$ in Theorem \ref{TeolimH1} satisfies also the ``dual condition'' 
\begin{align}\label{stimaHdual}
\sup_R \sup_{y,\,z\in R}\int_{(R^*)^c}|K(y,x)-K(z,x)| \,{\rm{d}}\mu (x) &< \infty \,,
\end{align}
where the supremum is taken over all \CZ sets $R$ in $\mathcal R$, then $T$ is bounded on $L^p(\mu)$, for all $p$ in $ (1,\infty)$. 

Indeed,  by (\ref{stimaHdual}) the adjoint operator $T'$ satisfies the hypothesis of Theorem \ref{TeolimH1}. Thus $T'$ is bounded on $L^p(\mu)$, for $1<p<2$. By duality it follows that $T$ is bounded on $L^p(\mu)$, for $2<p<\infty$.}}
\end{ossn}

\bigskip
We now show that if a generic \CZ space $(X,\mu,d)$ satisfies an additional hypothesis, then we may introduce an Hardy space and a bounded mean oscillation  space on $X$.  

Throughout this section we shall work on \CZ spaces satisfying the following additional condition (C).

\vspace{0,3cm}
\emph{There exists a subfamily $\mathcal R'$ of $\mathcal R$ such that the following hold:
\begin{itemize}
\item[(i)] given $R_1,\,R_2$ in $\mathcal R'$ such that $R_2\cap  R_1\neq \emptyset$, then either $R_1\subseteq R_2$ or $R_2\subseteq R_1$;
\item[(ii)] for every set $R$ in $\mathcal R$ there exists a set $R'$ in $\mathcal R'$ which contains $R$.
\end{itemize} 
}
\vspace{0,3cm}

We now introduce the space $H^{1,q}$ on $X$, for $1<q<\infty$.
\begin{defi}
A $(1,q)$-{\rm{atom}} is a function $a$ in $L^1(\mu)$ such that
\begin{itemize}
\item [(i)] $a$ is supported in a \CZ set $R$;
\item [(ii)]$\Big(\frac{1}{\mu(R)}\int_R|a|^q\di\mu\Big)^{1/q}\leq \mu(R)^{-1}\,;$ 
\item [(iii)]$\int_S a\di\mu =0$\,.
\end{itemize}
\end{defi}
Note that if $a$ is a $(1,q)$-atom supported in $R$, then
$$\|a\|_1=\int_R|a|\dir\leq \|a\|_{q}\,\mu(R)^{1/q'}\leq 1\,.$$
\begin{defi}
The Hardy space $H^{1,q}$ is the space of all functions $f$ in $ L^1(\mu)$ such that $f=\sum_j \lambda_j\, a_j$, where $a_j$ are $(1,q)$-atoms and $\lambda _j$ are complex numbers such that $\sum _j |\lambda _j|<\infty$. We denote by $\|f\|_{H^{1,q}}$ the infimum of $\sum_j|\lambda_j|$ over all decompositions $f=\sum_j|\lambda_j|\,a_j$, where $a_j$ are $(1,q)$-atoms.
\end{defi}
We now introduce the bounded mean oscillation space. For every locally integrable function $f$ and every set $R$ we denote by $f_R$ the average of $f$ on $R$, i.e. $f_R=\frac{1}{\mu(R)}\int_Rf\di \mu$. 

\begin{defi}
The space $\mathcal{B}\mathcal{M}\mathcal{O}_p$, for $1<p<\infty$, is the space of all functions in $L^p_{\rm{loc}}(\mu)$ such that
$$\sup_R\Big(\frac{1}{\mu(R)}\int_R|f-f_R|^p\di\mu \Big)^{1/p}<\infty\,,$$
where the supremum is taken over all \CZ sets in the family $\mathcal R$. The space $BMO_p$ is the quotient of $\mathcal{B}\mathcal{M}\mathcal{O}_p$ module constant functions. It is a Banach space with the norm defined by
$$\|f\|_{BMO_p}=\sup\Big\{\Big(\frac{1}{\mu(R)}\int_R|f-f_R|^p\di\mu \Big)^{1/p}\,:~R\in\mathcal R \Big\}\,.$$
\end{defi}
\begin{ossn}\label{BMOP}{\rm{We summarize some properties of $BMO_p$.
\begin{itemize}
\item[(i)] $L^{\infty}(\mu)$ is contained in $BMO_p$ and $\|f\|_{BMO_p}\leq 2\|f\|_{\infty}$. 

Indeed, for each \CZ set $R$ we have that
\begin{align*}
\Big(\frac{1}{\mu(R)}\int_R|f-f_R|^p\di\mu \Big)^{1/p}&\leq \frac{1}{\mu(R)^{1/p}}\big(\|f\|_{L^p(R)}+|f_R|\,\mu(R)^{1/p}\big)\\
&\leq 2\,\|f\|_{\infty}\,.
\end{align*}
\item[(ii)] Suppose that there exists a constant $C$ such that for all \CZ set $R$ there exists a constant $c_R$ such that $$\Big(\frac{1}{\mu(R)}\int_R|f-c_R|^p\di\mu \Big)^{1/p}\leq C< \infty \,.$$ Then $f$ is in $ BMO_p$. Indeed,
\begin{align*}
\Big(\frac{1}{\mu(R)}\int_R|f-f_R|^p\di\mu \Big)^{1/p}&\leq \Big(\frac{1}{\mu(R)}\int_R|f-c_R|^p\di\mu \Big)^{1/p}\\
&+\Big(\frac{1}{\mu(R)}\int_R|c_R-f_R|^p\di\mu \Big)^{1/p}\\
&\leq C+|c_R-f_R|\\
&\leq C+\frac{1}{\mu(R)^{1/p}}\left(\int_R |c_R-f|^p\di\mu \right)^{1/p}\\
&\leq 2\,C\,.
\end{align*}
\item[(iii)] If $f,g$ are in $ BMO_p$, then $|f|,~\max\{f,g\},~\min\{f,g\}$ are in $BMO_p$.
\item[(iv)] If $f$ is in $BMO_p$ and $a$ is a $(1,p')$-atom supported in $R$, then by H\"older's inequality we obtain that 
\begin{align*}
\Big|\int_S f\,a\di\mu \Big|&=\Big|\int_R(f(x)-f_R)a(x)\di\mu (x)\Big|\\
&\leq \Big(\int_R |f(x)-f_R|^p\di\mu (x)\Big)^{1/p}\Big(\int_R|a|^{p'}\di\mu\Big)^{1/{p'}}\\
&\leq \mu(R)^{1/p}\,\|f\|_{BMO_p}\,\mu(R)^{1/{p'}}\,\mu(R)^{-1}\\
&\leq \|f\|_{BMO_p}\,.
\end{align*}
\end{itemize}
Property (iv) is key to prove the duality between $BMO_p$ and $H^{1,p'}$.}}
\end{ossn}
\begin{teo}\label{duality}
For all $1<p<\infty$ the space $BMO_p$ is the dual of $H^{1,p'}$.
\end{teo}
\begin{proof}
We first prove that each $f$ in $ BMO_p$ represents a bounded linear functional $\ell_f$ on $H^{1,p'}$, in the sense that
$$\ell_f(g)=\int_S f\,g\di\mu\,,$$
and $|\ell_f(g)|\leq  C\,\|f\|_{BMO_p}\,\|g\|_{H^{1,p'}}$ for all functions $g$ in 
$$H^{1,p'}_{\rm{fin}}=\Big\{g\in L^1(\mu):~g=\sum_{j=1}^N\lambda_j\,a_j,\,a_j~ (1,p')-atoms, N\in\NN \Big\}\,.$$ 
Since $H^{1,p'}_{\rm{fin}}$ is dense in $H^{1,p'}$, it suffices to identify a unique bounded linear functional on $H^{1,p'}$. We distinguish three cases.

{\it{Case~$f\in L^{\infty}(\mu)$}}. For each $g$ in $ H^{1,p'}$ and for all $\varepsilon>0$, there exists a decomposition $g=\sum_j \lambda_j a_j$, where $\sum_j |\lambda _j|<\|g\|_{H^{1,p'}}+\varepsilon$ and $a_j$ are $(1,p')$-atoms. Since $f$ is in $ L^{\infty}(\mu)$ and the series converges to $g$ in $L^1(\mu)$, we have that
$$ \int_S f\,g\di\mu=\sum_j\lambda_j\int_S f\,a_j\di\mu\,,$$
and then, by using property (iv) above, we obtain that
\begin{align*}
\Big|\int_S f\,g\di\mu\Big|&\leq \sum_j|\lambda_j|\Big|\int_S f\,a_j\di\mu\Big|\\
&\leq \|f\|_{BMO_p}\,\sum_j|\lambda_j|\\
&\leq \|f\|_{BMO_p}\,(\|g\|_{H^{1,p'}}+\varepsilon)\,.
\end{align*}
By considering the infimum over $\varepsilon>0$, we deduce that the functional $\ell_f$ is bounded on $H^{1,p'}$ and has norm $\leq \|f\|_{BMO_p}$.

{\it{Case~$f\in BMO_p$ real valued.}} In this case we define for each $k\in\NN$
$$f_k(x)=\begin{cases}
k & \text{if $f(x)>k$}\\
f(x) & \text{if $|f(x)|\leq k$}\\
-k & \text{if $f(x)<-k$\,.}
\end{cases}
$$
Each function $f_k$ is in $L^{\infty}(\mu)$ and $\|f_k\|_{BMO_p}\leq C\,\|f\|_{BMO_p}$. By the previous case, for all $g$ in $ H^{1,p'}_{\rm{fin}}$ and $k\in\NN$ 
$$\Big|\int_S f_k\,g\di\mu\Big|\leq \|f_k\|_{BMO_p}\,\|g\|_{H^{1,p'}}\leq C\,\|f\|_{BMO_p}\,\|g\|_{H^{1,p'}}\,.$$
Since $g\in H^{1,p'}_{\rm{fin}}$ and $f\in BMO_p$, we have that $g$ and $f$ belong locally to $L^{p'}(\mu)$ and $L^p(\mu)$, respectively. Thus, by the dominated convergence theorem, since $f_k$ converges to $f$ almost everywhere, we deduce that 
$$\Big|\int_S f\,g\di\mu\Big|\leq C\,\|f\|_{BMO_p}\,\|g\|_{H^{1,p'}}\,,$$
as required.

{\it{Case~$f\in\Bp$ complex valued.}} It suffices to write $f={\rm{Re}}f+i{\rm{Im}}f$, use the previous case and the fact that $\|{\rm{Re}}f\|_{BMO_p}\leq C\,\|f\|_{BMO_p}$ and $\|{\rm{Im}}f\|_{BMO_p}\leq C\,\|f\|_{BMO_p}$.

So far we have proved that $BMO_p$ is contained in $(H^{1,p'})'$. Now we prove the converse inclusion. 

Let $\ell$ be in $ (H^{1,p'})'$. Our purpose is to define a function $f$ in $ BMO_p$ such that $\ell_f =\ell$ and $\|f\|_{BMO_p}\leq C\,\|\ell\|_{(H^{1,p'})'}$. For simplicity we first prove this fact in the case where $p=2$. 

For each \CZ set $R$ we denote by $L^2_{R}$ the set of all functions in $L^{2}(\mu)$ supported in $R$ and by $L^{2}_{R,\,0}$ the subspace of functions whose integral is zero. If $g$ is in $L^{2}_{R,\,0}$, then the function $a=\mu(R)^{-1/2}\,\|g\|_{2}^{-1}\,g$ is a $(1,2)$-atom. Then
$$|\langle \ell,g\rangle |\leq \mu(R)^{1/2}\,\|g\|_{2}\,\|\ell\|_{(H^{1,2})'}\,.$$
This shows that $\ell$ is in $(L^{2}_{R,\,0})'$. Thus, since $L^2_{R,\,0}$ is an Hilbert space, there exists a function $f^R\in L^{2}_{R,0}$ such that $\|f^R\|_2=\mu(R)^{1/2}\,\|\ell\|_{(H^{1,2})'}$ and $\langle \ell,g\rangle =\int_Rf^R\,g\di\mu$ for all $g\in L^2_{R,0}$. 

To define the function $f$ which represents the functional $\ell$ we proceed in the following way. Let $\mathcal R'$ be the family of \CZ sets which satisfies condition (C). We now define the function $f$ by
$$f(x)=f^{R}(x)\qquad{\rm{if}}~x\in R\,,$$
where $R$ is in $\mathcal R'$ (note that by (ii) in condition (C) there exists a set $R$ in $\mathcal R'$ which contains $x$). First we observe that this definition makes sense, because if $R_1\cap R_2\neq\emptyset$, then either $R_2\subseteq R_1$ or $R_1\subseteq R_2$. If $R_2\subseteq R_1$, then 
$$\int_{R_2}(f^{R_1}-f^{R_2})g\di\mu =0\qquad \forall g\in L^{2}_{R_2,0}\,,$$
which implies that $f^{R_2}$ and $f^{R_1}$ are equal on $R_2$. Otherwise we deduce that $f^{R_2}$ and $f^{R_1}$ are equal on $R_1$.

It remains to prove that $f\in BMO_2$ and $\ell_f=\ell$. For each \CZ set $R$ there exists a set $R'$ in $\mathcal R'$ such that $R\subseteq R'$. For every function $g\in L^{2}_{R,\,0}$ we have that
$$\int_R f\,g\di\mu=\,\langle \ell,g\rangle \,,$$
and then the restriction of $f$ on $R$ is equal to $f^R$. In particular it follows that
\begin{align*}
\Big(\frac{1}{\mu(R)}\int_R |f|^{2}\di\mu\Big)^{1/2}&=\Big(\frac{1}{\mu(R)}\int_R |f^R|^{2}\di\mu\Big)^{1/2}\\
&\leq \frac{1}{\mu(R)^{1/2}}\,\|f^R\|_2\\
&\leq C\,\|\ell\|\,.
\end{align*}
This shows that $f\in BMO_2$ and $\|f\|_{BMO_2}\leq C\, \|\ell\|\,,$ as required.

The proof in the case $p\neq 2$ is similar and then it is omitted.
\end{proof}
We now prove a boundedness theorem on a \CZ space for integral operators which satisfy the same hypotheses of Theorem \ref{Teolim}.
\begin{teo}\label{TeolimH1}
Let $(X,\mu,d)$ be a Calder\'on--Zygmund space which satisfies condition (C). Let $T$ be a linear operator as in Theorem \ref{Teolim}. Then $T$ is bounded from $H^{1,q}$ to $L^1(\mu)$, for all $1<q\leq 2$.
\end{teo} 
\begin{proof}
Let $q$ be in $(1,2]$: we claim that there exists a constant $A$, which depends only on the norm of $T$, such that $\|Ta\|_1\leq A$ for each $(1,q)$-atom $a$.

Let $a$ be a $(1,q)$-atom supported in the \CZ set $R$. Recall that $R\subseteq {B(x_0,\kappa_0 \,r)}$, for some $x_0$ in $X$ and $r>0$, and denote by $R^*$ the dilated set of $R$. We need to estimate the integral $\int_S |Ta|\di\mu$.

We first estimate the integral on $R^*$ by H\"older's inequality:
\begin{align}\label{suR^*}
\int_{R^*} |Ta|\di\mu&\leq \|Ta\|_{q}\,\mu(R^*)^{1/q'}\nonumber\\
&\leq\kappa_0^{1/q'}\, \opnorm T\opnorm_{q}\,\|a\|_q\,\mu(R)^{1/q'}\nonumber\\
&\leq \kappa_0^{1/q'}\,\opnorm T\opnorm_{q}\,\mu(R)^{-1+1/q}\,\mu(R)^{1/q'}\nonumber\\
&=\kappa_0^{1/q'}\,\opnorm T\opnorm_{q}\,.
\end{align}
We consider the integral on the complementary set of $R^*$ by using the fact that $a$ has average zero: 
\begin{align*}
\int_{R^{*c}} |Ta|\di\mu&\leq \int_{(R^{*})^c}\Big|\int_R K(x,y)\,a(y)\di\mu(y)  \Big|\di\mu(x)\\
&=
\int_{(R^{*})^c}\Big|\int_R [K(x,y)-K(x,x_0)]\,a(y)\di\mu(y)  \Big|\di\mu(x)\\
&\leq \int_{(R^{*})^c}\int_R |K(x,y)-K(x,x_0)|\,|a(y)|\di\mu(y)\di\mu(x)\\
&=\int_R|a(y)|\Big( \int_{(R^{*})^c} |K(x,y)-K(x,x_0)|\di\mu(x) \Big)\di\mu(y)\\
&\leq  \|a\|_1\,\sup_{y\in R}\int_{(R^{*})^c}|K(x,y)-K(x,x_0)|\di\mu(x)\\
&\leq  C\,,
\end{align*}
as required by the claim.

Now let $g$ be in $ H^{1,q}$. There exists a decomposition $g=\sum_j \lambda_j a_j$ such that $\sum_j |\lambda_j|<\|g\|_{H^{1,q}}+\varepsilon$ and $a_j$ are $(1,q)$-atoms. Since the operator $T$ is of weak type $(1,1)$ we have that $Tg=\sum_j\lambda_j\,Ta_j$, whose series converges in $L^{1,\infty}(\mu)$. Thus,
\begin{align*}
\|Tg\|_1&\leq \sum_j |\lambda_j|\,\|Ta_j\|_1\\
&\leq A\,\sum_j |\lambda_j|\\
&\leq A (\|g\|_{H^{1,q}}+\varepsilon)\,.
\end{align*}
This shows that $T$ is bounded from $H^{1,q}$ to $L^1(\mu)$. 
\end{proof}
\begin{coro}\label{TeolimBMO}
Let $(X,\mu,d)$ and $T$ be as in Theorem \ref{TeolimH1}. If $T$ satisfies estimate (\ref{stimaHdual}), then it is bounded from $H^{1,q}$ to $L^1(\mu)$, for all $2\leq q<\infty$, and from $L^{\infty}(\mu)$ to $BMO_p$, for all $1< p<\infty$.
\end{coro}
\begin{proof}
Suppose that $2\leq q<\infty$. By Remark \ref{dualcondition} it follows that $T$ is bounded on $L^q(\mu)$. By arguing as in the proof of Theorem \ref{TeolimH1}, we may prove that $T$ is bounded from $H^{1,q}$ to $L^1(\mu)$. 

Suppose that $1< p<\infty$. Since the adjoint operator $T'$ is bounded from $H^{1,p'}$ to $L^1(\mu)$, by duality it follows that $T$ is bounded from $L^{\infty}(\mu)$ to $BMO_p$.
\end{proof}
\section{The CZ decomposition in \DR spaces}\label{CZdecextH}
In this chapter we prove that \DR spaces are \CZ spaces. 

The simplest example of \DR spaces are the $ax+b\,$-groups, whose \CZ theory has been studied by Hebisch and Steger \cite{HS}. We now recall their main ideas.

Let $S=\RR^d\times \RR^+$ be an $ax+b\,$-group. Hebisch and Steger introduced a family of admissible sets which are products of dyadic cubes in $\RR^d$ and intervals in $\RR^+$. More precisely, a rectangle $Q^k_m\times (a_0/r,a_0\,r)$, where $Q^k_m$ is a dyadic cube in $\RR^d$, $a_0\in \RR^+$, $r>1$, is \emph{admissible} if
\begin{align}\label{admb}
a_0^{1/2}r^{\beta}&\leq 2^k < a_0^{1/2}r^{4\beta}\qquad{\rm{if}}~r\geq\nep\,,\end{align}
and
\begin{align}\label{adms}
a_0^{1/2}\log r&\leq 2^k < \nep^{4\beta}\,a_0^{1/2}\log r\qquad{\rm{if}}~1<r<\nep\,,\end{align}
and $\beta$ is a constant $>1+\log 2\,$.

Note that admissible sets are rectangles either {\emph{big}}, if $r\geq \nep$, or {\emph{small}}, if $r<\nep$. By using the properties of the metric, one can prove that these sets satisfy property (i) of Definition \ref{CZdef}. More precisely, each admissible set $R=Q^k_m\times (a_0/r,a_0\,r)$ is contained in a ball of radius $\kappa_0\,\log r$ and $\rho(R^*)\leq \kappa _0\,\rho(R)$, where $R^*=\{(n,a)\in S: d\big((n,a),R\big)<\log r\}$, for an appropriate constant $\kappa _0$.

A noteworthy property of admissible sets is that every admissible set, either big or small, can be split up in a finite number of admissible sets. More precisely, given an admissible set $R=Q^k_m\times (a_0/r,a_0\,r)$, we may split up $R$ either in the union of $2^d$ admissible subsets $Q^{k-1}_i\times (a_0/r,a_0\,r)$, where $Q^{k-1}_i$ are $2^d$ dyadic subsets of $Q^k_m$, or in the union of two admissible subsets $Q^k_m\times (a_0/r,a_0)$ and $Q^k_m\times (a_0,a_0\,r)$. This ``splitting property'' has an important r\^ole in the \CZ decomposition: it allows to reproduce the standard ``stopping argument'' of the classical \CZ decomposition in this context. 

Now let us consider a partition $\mathcal P$ of $S$ in admissible sets. If we split up each set in $\mathcal P$ as described above, then, by iterating the process, we find a sequence $\{\mathcal P_j\}_j $ of partitions of $S$ in admissible sets such that:
\begin{itemize}
\item[(i)] $\mathcal P_0=\mathcal P$;
\item[(ii)] each set in $\mathcal P_{j}$ is the union of at most $2^d$ sets of the partition $\mathcal P_{j+1}$ of equal measure;
\item[(iii)] if $\{  R_j\}_j$ is a sequence of sets such that $R_j\in\mathcal P_j$ and
$$R_0\supseteq  R_1\supseteq  ...\supseteq 
R_j\supseteq  R_{j+1}\supseteq ...\,,$$
then the diameter of $R_j$ tends to zero as $j$ tends to $\infty$.
\end{itemize} 

Let $\mathcal R$ denote $\bigcup_j{{\mathcal P} _j}$. It is easy to see that for all $x\in S$ and $j\in \NN$ there exists a unique set $R_j^x\in {\mathcal P}_j$ such that $x\in R_j^x$. The sequence $\{ R_j^x\}_j$ is decreasing. For every $f$ in $ L^1_{\rm{loc}}(\rho)$ and $j\in \NN$ we define the operator $\mathcal E_j$ by
\begin{equation}\label{martingale}
\mathcal E_jf(x)=\frac{1}{\rho(R^x_j)}\int_{R^x_j}f\dir \qquad \forall x\in S\,.
\end{equation}
By a standard argument, one can show that the maximal operator $M^{{\mathcal R}}$, which is defined by
$$M^{{\mathcal R}}f(x)=\sup_{R\in\mathcal R,\,x\in R}\frac{1}{\rho(R)}\int_R|f|\dir \qquad\forall f\in
L_{\rm{loc}}^1(\rho)\,,$$
is of weak type $(1,1)$ and that for every locally integrable function $f$
$$\lim_{j\to +\infty}\mathcal E_jf(x)=f(x)\qquad{\rm{ for~almost~every~}}x\in S\,.$$

By using the maximal operator $M^{\mathcal R}$, the operators $\mathcal E_j$ and a classical stopping argument, Hebisch and Steger \cite{HS} defined a \CZ decomposition of integrable functions on $ax+b\,$-groups.

\bigskip
Now let $S$ be a \DR space. We shall prove that $(S,\rho,d)$ is a \CZ space. The idea is to generalize the \CZ decomposition in \cite{HS} to this context, but this is not an obvious generalization. It is tempting to extend the definition of admissible sets from $ax+b\,$-groups to the space $S$ by strict analogy. This would lead us to define an admissible set $R$ as a product  $\Qka\times (a_0/r,a_0\,r)$, where $\Qka$ is a dyadic set in $N$ and $a_0\in A$, which satisfies conditions that are strictly analogous to (\ref{admb}) and (\ref{adms}). In particular for small sets we would obtain the condition
\begin{align*}
a_0^{1/2}\log r&\leq \rd^k < \nep^{4\beta}\,a_0^{1/2}\log r\qquad{\rm{if}}~1<r<\nep\,,\end{align*}
where $\beta$ is an appropriate constant and $\rd$ is as in Theorem \ref{dyadic}. Unfortunately this does not work.

Indeed, given a small set $R$ defined as above, it is contained in a ball of radius $C\,\log r$ and $\rho(R)\asymp a_0^{Q}\,(\log r)^{2Q+1}$. It is clear that $B\big((\nka,\,a_0),\log r)\subseteq R^*$, where $R^*=\{x\in S:~d(x,R)<\log r\} $, and then $\rho(R^*)\geq \gamma_1\,a_0^{Q}\, (\log r)^n$. Thus,
$$\frac{\rho(R^*)}{\rho(R)}\geq C\, (\log r)^{n-2Q-1}=C\,(\log r)^{-m_{\zg}}\qquad \forall r\in (1,\nep)\,.$$
The quantity above is bounded by a constant $\kappa_0$ (which is required by Definition \ref{CZdef}) if and only if $m_{\zg}=0$, i.e. if $\n$ is abelian. This holds only in the case where $S$ is an $ax+b\,$-group. 

%Since $2Q+1=m_{\vg}+2\,m_{\zg}+1\geq n=m_{\vg}+m_{\zg}+1$ and $1<r<\nep$, this is compatible with the condition $\rho(R^*)\leq \kappa_0\,\rho(R)$ if and only if $2Q+1=n$, i.e. if $\n$ is abelian ($~\zg=(0)~$). This holds only in the case where $S$ is an $ax+b\,$-group. Moreover in $ax+b\,$-groups one can also prove that $R^*\subseteq B_{\RR^d}(x^k_{\alpha},C\,a_0^{1/2}\log r)\times (a_0/r,a_0\,r)$ and then $\rho(R^*)\asymp a_0^Q\,(\log r)^{2Q+1}= a_0^Q\,(\log r)^{2Q+1}$ which is comparable with the measure of $R$.

Then, in \DR spaces we need a new definition of small admissible sets. A {\it{small admissible set}} is a set in the family $\mathcal R^0$, i.e. a geodesic ball of radius less than $1/2$.

On the contrary the definition of big admissible sets is an extension of big admissible sets of Hebisch and Steger.

A {\it{big admissible set}} is a set in $\mathcal D^{\infty}$, i.e. a set $R= \Qka\times (a_0/r,a_0\,r)\,,$ where $\Qka$ is a dyadic set in $N$, $a_0\in A$, $r\geq {\rm{e}}$,
\begin{align}\label{admbis}
a_0^{1/2}r^{\beta}&\leq \rd^k < a_0^{1/2}r^{4\beta}\,,\end{align}
and $\beta$ is a constant $>{\rm  max}\left\{ 3/2,1/4+\log \rd , 1+\log\big(c_3/c_N\big)\right\}\,$, where $c_3$, $\rd$, $c_N$ are the constants which appear in Proposition \ref{palleS} and Theorem \ref{dyadic}. 

We now investigate some geometric properties of big admissible sets, which correspond to property (i) in Definition \ref{CZdef}.
\begin{prop}\label{proprieta}
Let $R$ denote the big admissible set $\Qka\times (a_0/r,a_0\,r)$ and let $c_3$, $c_N$, $C_N$, $\rd$ be as in Proposition \ref{palleS} and Theorem \ref{dyadic} (ii). The following hold:
\begin{itemize}
\item[(i)] there exists a constant $C_{N,\beta}$ such that $R\subseteq B\big((n_{\alpha}^k, a_0),C_{N,\,\be}\,\log r\big)$;
\item[(ii)] $c_N^{2Q}\,|B_N(0_N,1)|\,(a_0^{1/2}r^{\be})^{2Q}\,\log r \leq \rho(R)\leq C_N^{2Q}\,|B_N(0_N,1)|\,(a_0^{1/2}r^{4\be})^{2Q}\,\log r\,;$
\item[(iii)] let $R^*$ be the set $\{(n,a)\in S: d\big((n,a),R\big)<\log r\}$; then 
$$\rho\big(R^*\big)\leq \left(\frac{c_3+C_N}{c_N}\right)^{2Q}\rho\big(R\big)\,.$$
\end{itemize}
\end{prop}
\begin{proof}
To prove (i), note that by Theorem \ref{dyadic} (ii)
\begin{align*}
R&\subseteq B_N\big(\nka,C_N\,\rd^k\big)\times (a_0/r,a_0\,r)\,,
\end{align*}
which, in turn, is contained in $B_N\big(\nka,C_Na_0^{1/2}r^{4\be})\times (a_0/r,a_0\,r)$ by the admissibility condition (\ref{admbis}). By the left invariance of the metric and Proposition \ref{palleS} (ii), 
\begin{align}\label{Rinpalla}
R&\subseteq (\nka,a_0)\cdot [B_N\big(0_N,C_N\,r^{4\be})\times (1/r,r)]\nonumber\\
&\subseteq (\nka, a_0)\cdot \big[B(e,C_{N,\be}\,\log r)\big]\nonumber\\
&=B\big((\nka, a_0), C_{N,\be}\,\log r \big)\,,
\end{align}
as required.

We now prove (ii). Since $\rho(R)=|\Qka|\,\log r$ and by Theorem \ref{dyadic} (ii) 
\begin{align*}
c_N^{2Q}\,|B_N(0_N,1)|\,\rd^{2Qk}\,\log r&\leq \rho(R)\leq C_N^{2Q}\,|B_N(0_N,1)|\,\rd^{2Qk}\,\log r\,.
\end{align*}
Since $a_0^{1/2}r^{\beta}\leq \rd^k < a_0^{1/2}r^{4\beta}$, (ii) follows.

To prove (iii), we observe that
$$R^*=\bigcup_{(n,a)\in R} B\big((n,a),\log r\big)\,.$$
By using the left invariance of the metric and Proposition \ref{palleS} (i), we obtain that\begin{align*}
B\big((n,a),\log r\big)&=(n,a)\cdot \big[B(e,\log r)\big]\\
&\subseteq (n,a)\cdot \big[B_N\big(0_N,c_3\,r\big)\times (1/r,r)\big]\\
&=B_N\big(n,c_3\,a^{1/2}r\big)\times (a/r,a\,r)\,\qquad\forall (n,a)\in R\,.
\end{align*}
Since $(n,a)$ is in $R$ and $R$ is admissible, we see that
\begin{align*}
(a/r,ar)&\subseteq (a_0/r^2,a_0r^2)
\end{align*}
and 
\begin{align*}
B_N\big(n,c_3\,a^{1/2}r\big)&\subseteq B_N\big(n,c_3\,a_0^{1/2}r^{3/2}\big)\\&\subseteq B_N\big(\nka,c_3\,a_0^{1/2}r^{\be}+C_N\,\rd^k\big)\\
&\subseteq B_N\big(\nka,(c_3+C_N)\,\rd^k\big) \,.
\end{align*}
Thus
$$R^*\subseteq B_N\big(\nka,(c_3+C_N)\,\rd^k\big)\times (a_0/r^2,a_0\,r^2)\,.$$
Finally,
\begin{align*}
\rho\big(R^*\big)&\leq \Big(\frac{c_3+C_N}{c_N}\Big)^{2Q}\,|B_N\big(\nka,c_N\,\rd^k\big)|\,\log r\\
&\leq  \Big(\frac{c_3+C_N}{c_N}\Big)^{2Q}\,\rho\big(R\big)\,,
\end{align*}
as required.
\end{proof}
\begin{ossn} {\rm{Let $R=\Qka\times (a_0/r,a_0\,r)$ be a big admissible set. We have defined either the dilated set $\tR=\Qka\times (a_0/{r^{\gamma}},a_0\,{r^{\gamma}})$ or the dilated set $R^*=\{x\in S:~d(x,R)<\log r  \} $. They contain $R$ and their measures are comparable. Note that they are different and $\tR$ is not contained in $R^*$ and $R^*$ is not contained in $\tR$. Nevertheless there exists a set $R^{**}= B_N\big(\nka,(c_3+C_N)\,\rd^k\big)\times (a_0/r^{\gamma},a_0\,r^{\gamma})$ which contains either $\tR$ or $R^*$.}}
\end{ossn}

We remarked that an important property of admissible sets in $ax+b\,$-groups is their ``splitting property''. In  generic \DR spaces only big admissible sets satisfy a ``splitting property'' which is analogue, but not equal, to the $ax+b\,$-case. Indeed, most (but not all) big admissible sets may be split up in a finite number of mutually disjoint smaller subsets which are still admissible. 
\begin{lem}\label{tagli}
Let $R$ denote the big admissible set $\,\Qka\times (a_0/r_0,a_0r_0)$ and let $\rd$, $M$,$\nka$, $c_N$, $C_N$ be as in Theorem \ref{dyadic}. The following hold:\begin{itemize}
\item[(i)] if $\rd^{k-1}\geq a_0^{1/2}r^{\be}$, then there exist $J$ mutually disjoint big admissible sets $R_1,...,R_J$ such that $2\leq J\leq M$, $R=\bigcup_{i=1}^J R_i$ and
$$\big(c_N/{(\rd \,C_N)}\big)^{2Q}\rho(R)\leq\rho(R_i)\leq \rho(R) \qquad i=1,...,J\,;$$
\item[(ii)] if $\rd^{k-1}< a_0^{1/2}r^{\be}$ and $r\geq \nep ^2$, then there exist two disjoint big admissible sets $R_1$  and $R_2$ such that $R=R_1\cup R_2$ and $\rho(R_i)=\rho(R)/2$, for $i=1,2$;
\item[(iii)] if $\rd^{k-1}< a_0^{1/2}r^{\be}$ and $r< \nep ^2$, then there exists a constant $\Rnd$ such that
\begin{align}\label{grandeallimite}
B\big((\nka,a_0),1\big)&\subseteq R\subseteq B\big((\nka,a_0),\Rnd\big)\,.
\end{align}
\end{itemize}
\end{lem}
\begin{proof}
To prove (i), suppose that $\rd^{k-1}\geq a_0^{1/2}r^{\be}$. We split up $R$ in the following way: let $Q_i^{k-1}$, $\,i=1,...,J$ be the subsets of $\Qka$ as in Theorem \ref{dyadic} ($2\leq J\leq M$). Define
$$R_i=Q_i^{k-1}\times (a_0/r_0,a_0\,r_0)\qquad i=1,...,J\,.$$
Since $\rd^{k-1}\geq a_0^{1/2}r^{\be}$, the sets $R_i$ are admissible. Obviously $R=\bigcup _{i=1}^JR_i$ and $\rho(R_i)\leq \rho(R)$. By Theorem \ref{dyadic} (ii) 
\begin{align*}
\rho(R_i)&=|Q_i^{k-1}|\,\log r\\
&\geq |B_N(0_N,c_N\,\rd ^{k-1})|\,\log r\\
&=\Big|B_N\Big(0_N,\big(c_N/(\rd \,C_N)\big)\,C_N\,\rd ^k\Big)\Big|\,\log r\\
&\geq \big(c_N/(\rd \,C_N)\big)^{2Q} \big|B_N(0_N,C_N\,\rd^k)\big|\,\log r\\
&\geq \big(c_N/(\rd \,C_N)\big)^{2Q}\rho(R)\,,
\end{align*}
as required.

To prove (ii), suppose that $\rd^{k-1}< a_0^{1/2}r^{\be}$ and $r\geq \nep ^2$. Then by the admissibility condition (\ref{adm}), 
\begin{align}\label{admis}
a_0^{1/2}r^{\be}&\leq \rd^k<\rd\, a_0^{1/2}r^{\be}\,.
\end{align}
Define $R_1$ and $R_2$ by
$$R_1=\Qka\times (a_0/r,a_0) \qquad{\rm and}\qquad R_2=\Qka\times (a_0,a_0\,r)\,.$$
Clearly $R_1$ and $R_2$ are ``centred'' at $(\nka,{a_0}/{\sqrt r})$ and $(\nka,a_0\sqrt r)$ respectively. Note that $\sqrt r\geq \textrm e$. To prove that $R_1$ and $R_2$ are admissible we use (\ref{admis}):
\begin{align*}
({a_0}/{\sqrt r})^{1/2}(\sqrt r)^{\be}&\leq a_0^{1/2}r^{\be}\\
&\leq \rd^k\,;\\
({a_0}/{{\sqrt r}})^{1/2}(\sqrt r)^{4\be}&=\rd^{-1}r^{\be-1/4}\,\rd\,a_0^{1/2}r^{\be}\\
&>\rd^{-1}\textrm e^{\be-1/4}\,\rd^k\\
&>\rd^k\,.
\end{align*}
This proves that $R_1$ is admissible. The proof of the admissibility of $R_2$ is similar and is omitted. %For the set $R_2$ we proceed similarly:
%\begin{align*}
%\big(a_0{\sqrt r}\big)^{1/2}(\sqrt r)^{\be}&\leq a_0^{1/2}r^{\be}\\
%&\leq \rd^k\,;\\
%\big(a_0{\sqrt r}\big)^{1/2}(\sqrt r)^{4\be}&=\rd^{-1}r^{\be+1/4}\,\rd \,a_0^{1/2}r^{\be}\\
%&\geq\rd^{-1}\textrm e^{\be+1/4}\,\rd^k\\
%&>\rd^k\,.
%\end{align*}
%This proves that also $R_2$ is admissible. \\
Obviously $R=R_1\cup R_2$ and $\rho(R_i)=\rho(R)/2$, $i=1,2$, as required.

We now consider (iii). Suppose that $\rd^k\leq \rd\,a_0^{1/2} r^{\be}$ and $\textrm e\leq r<\textrm e ^2$. By the admissibility condition (\ref{adm}) and the left invariance of the metric we have that
\begin{align*}
R&\subseteq B_N\big(\nka,C_N\,\rd^k\big)\times (a_0/r,a_0\,r)\\
&\subseteq B_N\big(\nka,C_N\,a_0^{1/2} r^{4\be}\big)\times (a_0/r,a_0r\,)\\
&=(\nka,a_0)\cdot \big[B_N\big(0_N,C_N \,r^{4\be}\big)\times (1/r,r)\big]\,.
\end{align*}
Since $r<\textrm{e}^2$ and by Proposition \ref{palleS} (ii) we conclude that
\begin{align*}
R&\subseteq (\nka,a_0)\cdot \big[B_N\big(0_N,C_N\,\textrm e^{8\be}\big)\times (1/{\textrm e^2},\textrm e^2)\big]\\
&\subseteq B\big((\nka,a_0),\Rnd\big)\,,
\end{align*}
where $\Rnd$ depends only on $\beta$ and $C_N$. Similarly, (\ref{adm}) and the left invariance of the metric imply that
\begin{align*}
R&\supseteq B_N\big(\nka,c_N\,\rd^k\big)\times (a_0/r,a_0\,r)\\
&\supseteq B_N\big(\nka,c_N\,a_0^{1/2} r^{\be}\big)\times (a_0/r,a_0r)\\
&=(\nka,a_0)\cdot \big[B_N\big(0_N,c_N \,r^{\be}\big)\times (1/r,r)\big]\,.
\end{align*}
Since $r\geq\textrm{e}$ and by Proposition \ref{palleS} we conclude that
\begin{align*}
R&\supseteq (\nka,a_0)\cdot \big[B_N\big(0_N,c_N \,\textrm e^{\be}\big)\times (1/{\textrm e},{\textrm e})\big]\\
&\supseteq B\big((\nka,a_0),1\big)\,,
\end{align*}
as required.
\end{proof}
For later developments it is useful to distinguish big admissible sets that satisfy condition (i) or (ii) in Lemma \ref{tagli}, which may be split up in a finite number of smaller big admissible sets, and big admissible sets that satisfy condition (iii) in Lemma \ref{tagli}, which cannot be split up in that way. 
\begin{defi}
A big admissible set $\Qka\times (a_0/r,a_0\,r)$ is said to be {\emph{divisible}} if either $\rd^{k-1}\geq a_0^{1/2}r^{\be}$ or $r\geq\nep^2$.\\
A big admissible set $\Qka\times (a_0/r,a_0\,r)$ is said to be {\emph{nondivisible}} if $\rd^{k-1}< a_0^{1/2}r^{\be}$ and  $r<\nep^2$.
\end{defi}
Next we show that there exists a partition of $S$ which consists of  big admissible sets whose measure is as large as needed. 
\begin{lem}\label{partizionegrande}
For all $\sigma>0$ there exists a partition $\mathcal P_{\sigma}$ of $S$ which consists of big admissible sets whose measure is $>\sigma$.
\end{lem}
\begin{proof}
As a first step, we choose $r_0\geq \textrm{e}$ and $k_0\in\ZZ$ such that $r_0^{2\be Q}\,\log r_0>\frac{\sigma}{c_N^{2Q}\,|B_N(0_N,1)|}$ and $r_0^{\be}\leq \rd^{k_0}<r_0^{4\be}$. The sets $R_{\alpha}^0=Q_{\alpha}^{k_0}\times (1/{r_0},r_0)$, $\alpha\in I_{k_0}$,  are big admissible sets and

\begin{align*}
\rho(R_{\alpha}^0)&=|Q^{k_0}_{\alpha}|\,\log r_0\\
&\geq c_N^{2Q}\,|B(0_N,1)|\,\rd^{2k_0Q}\log r_0\\
&\geq  c_N^{2Q}\,|B(0_N,1)|\,r_0^{2\be Q}\log r_0\\
&>\sigma\qquad \forall\alpha\in I_{k_0}\,.
\end{align*}
Then the sets $R_{\alpha}^0\,,\alpha\in I_{k_0}$, give a partition of the strip $N\times (1/{r_0},r_0)$ which consists of big admissible sets whose measure is $>\sigma$.

Next suppose that a partition of a strip $N\times (a_n/{r_n},a_n\,r_n)$ which consists of admissible sets whose measure is $>\sigma$ has been chosen. Then we choose $r_{n+1}\geq \nep$ and $k_{n+1}\in\ZZ$ such that $(a_{n+1}^{1/2}\,r_{n+1}^{\be})^{2 Q}\,\log r_{n+1}>\frac{\sigma}{ c_N^{2Q}\,|B(0_N,1)|}$ and $a_{n+1}^{1/2}\,r_{n+1}^{\be}\leq \rd^{k_1}<a_{n+1}^{1/2}\,r_{n+1}^{4\be}$, where $a_{n+1}=a_n\,r_n\,r_{n+1}$. The sets $R_{\alpha}^{n+1}=Q_{\alpha}^{k_{n+1}}\times (a_n\,r_n,a_{n+1}\,r_{n+1})$, $\alpha\in I_{k_{n+1}}$,  are big admissible sets whose measure is $>\sigma$. They give a partition of the strip $N\times (a_n\,r_n,a_{n+1}\,r_{n+1})$.

By iterating this process we obtain a partition of $N\times (r_0,\infty)$. By a similar procedure, we define a partition of $N\times (0,1/{r_0})$ which consists of big admissible sets with the required property. 
\end{proof}

We need a geometric lemma concerning  intersection properties between  balls and ``big \nondivisible sets''. 
\begin{lem}\label{BintQ}
Let $B$ be a ball of radius $1/2\leq R\leq\gamma/2$, where $\gamma$ is the constant which appears in Section \ref{small}. Let $\{F_{\ell}\}_{\ell}$ be a family of mutually disjoint \nondivisible big admissible sets. Then:
\begin{itemize}
\item[(i)] if $B\cap F_{\ell}\neq \emptyset$, then $\rho(B)\geq 2^{-n}\,\big(\gamma_1/{\gamma_2}\big)\,\nep^{-Q(2\Rnd+\gamma/2)}\,\rho(F_{\ell})$;
\item[(ii)] the ball $B$ intersects at most $\big(\gamma_2/{\gamma_1}\big)\,\nep^{Q(1+\Rnd+\gamma/2)}$ sets of the family $\{F_{\ell}\}_{\ell}\,$, where $\Rnd$ is the constant which appears in Lemma \ref{tagli}.
%\item[(ii)] the ball $B$ intersects a number of sets of the family $\{F_{\ell}\}\,$ which is $\leq C\nep^{Q(1+\Rnd+\gamma/2)}$, where $\Rnd$ is the constant which appears in Lemma \ref{tagli}.
\end{itemize}
\end{lem}
\begin{proof}
Let $x_0$ be the centre of the ball $B$. Note that $B(x_0,1/2)\subseteq B\subseteq B(x_0,\gamma/2)$. By (\ref{grandeallimite}) there exist points $y_{\ell}$ such that $B\big(y_{\ell},1\big)\subseteq F_{\ell}\subseteq B\big(y_{\ell},\Rnd\big)$.\\
To prove (i), note that 
$$\rho(B)\geq \gamma_1\,\delta^{-1}(x_0)\,R^n\geq \gamma_1\,\delta^{-1}(x_0)\,(1/{2})^n\,,$$
while
$$\rho(F_{\ell})\leq\delta^{-1}(y_{\ell})\,\rho\big(B(e,\Rnd)\big)\leq \gamma_2\,\delta^{-1}(y_{\ell})\,\nep^{Q\,\Rnd}\,.$$
If $B\cap F_{\ell}\neq\emptyset$, then $d(x_0,y_{\ell})<\gamma/2+\Rnd$, and so $\delta(y_{\ell}x_0^{-1})\geq \nep^{-Q(\Rnd+\gamma/2)}$. Therefore
\begin{align*}
\rho(B)/\rho(F_{\ell})&\geq 2^{-n}\,\big(\gamma_1/{\gamma_2}\big)\,\nep^{-Q(2\Rnd+\gamma/2)} \,,
\end{align*} 
as required in (i).\\
To prove (ii), note that if $\ell\neq k$, then $B\big(y_{\ell},1\big)\cap B\big(y_{k},1\big)=\emptyset$, since $F_{\ell}\cap F_k=\emptyset$. Now let $\mathcal I=\{\ell:~B\cap F_{\ell}\neq \emptyset\}$. Obviously, if $\ell$ is in $\mathcal I$, then $B(y_{\ell},1)\subseteq B\big(x_0, \gamma/2+1+\Rnd\big)$, so that
$$\bigcup_{\ell\in \mathcal I} B(y_{\ell},1)\subseteq B\big(x_0, \gamma/2+1+\Rnd\big)\,.$$
Now consider the left invariant measure of the sets above:
$$\sharp\mathcal I\,\cdot\,\lambda\big(B(e,1)\big)\leq \lambda\big(B(e,\gamma /2+1+\Rnd)\big)\,.$$
Then $B$ intersects at most 
\begin{align*}
\sharp\mathcal I&\leq{\lambda\big(B\big(e,\gamma /2+1+\Rnd\big)\big)}/{\lambda\big(B(e,1)\big)}\\
&\leq \big(\gamma_2/{\gamma_1})\,\nep^{Q(1+\Rnd+\gamma/2)}
\end{align*} 
sets of the family $\{F_{\ell}\}_{\ell}$, as required. 
\end{proof}
We now prove that $(S,\rho,d)$ is a \CZ space. The ingredients are similar to those in the proof of \cite[Lemma 5.1 ]{HS}: the admissible sets, the maximal function, a stopping argument. Our proof is much more complicated because big and small admissible sets have a different structure and the ``splitting property'' is more complicated.
\begin{teo}\label{CZd}
Let $S$ be a \DR space. Then $(S,\rho,d)$ is a \CZ space. 
\end{teo}
\begin{proof}
Let $f$ be in $L^1(\rho)$ and $\alpha > 0$. Our purpose is to define a Cal\-der\'on--Zyg\-mund  decomposition of $f$ at height $\alpha$ .

Let ${\mathcal P}$ be a partition of $S$ which consists of big admissible sets whose measure is $>{\|f\|_{L^1(\rho)}}/{\alpha}$  (it does exist by Lemma \ref{partizionegrande}). For each $R$ in $ {\mathcal P}$ we have that $\frac{1}{\rho(R)}\int_R|f|\dir < \alpha$.

Now we split up each divisible set $R$ in ${\mathcal P}$ into big admissible disjoint subsets $R_i$, $i=1,...,J$, such that $2\leq J\leq M$, as in Lemma \ref{tagli}. If $\frac{1}{\rho(R_i)}\int_{R_i}|f|\dir \geq\alpha\,,$ then we stop, otherwise, if $R_i$ is divisible, then we split up $R_i$ and stop when we find a subset $E$ such that $\frac{1}{\rho(E)}\int_{E}|f|\dir\geq\alpha$.

By iterating this process, we obtain the family $\{E_i\}_i$ of the stopping sets. The sets $E_i$ have the following properties:
\begin{itemize}
\item[(i)] $E_i$ are mutually disjoint big admissible sets;
\item[(ii)] $\frac{1}{\rho(E_i)}\int_{E_i}|f|\dir \geq\alpha$;
\item[(iii)] for each set $E_i$, there exists a set $E_i'$ such that $\frac{1}{\rho(E_i')}\int_{E_i'}|f|\dir <\alpha$ and $\rho(E_i')\leq\max\big\{2,\,\big(\rd \,C_N/{c_N}\big)^{2Q}\big\}\,\rho(E_i)$. Then
\begin{align*}
\frac{1}{\rho(E_i)}\int_{E_i}|f|\dir&\leq \max\big\{2,\,\big(\rd \,C_N/{c_N}\big)^{2Q}\big\}\,\frac{1}{\rho(E_i')}\int_{E_i'}|f|\dir \\
&<\max\big\{2,\,\big(\rd \,C_N/{c_N}\big)^{2Q}\big\}\,\alpha\,;
\end{align*}
\item[(iv)]the complementary of the $\bigcup_i E_i$ is the union of mutually disjoint \nondivisible big admissible sets $\{F_{\ell}\}_{\ell}$ such that $\frac{1}{\rho(F_{\ell})}\int_{F_{\ell}}|f|\dir < \alpha$.
\end{itemize}
By Proposition \ref{proprieta} for each set $E_i$ there exists a poistive number $r_i$ such that $E_i$ is contained in a ball of radius at most $C_{N,\,\beta}\,\log r_i$. Moreover the sets $E_i^*=\{x\in S:~d(x,E_i)<\log r_i\}$ have measures $\rho\big(E_i^*\big)\leq \Big(\frac{c_3+C_N}{c_N}\Big)^{2Q}\,\rho(E_i)$. 

Let $g_f$, $b_f^i$ and $h$ be defined by
\begin{align*}
g_f&=\sum_i\Big(\frac{1}{\rho(E_i)}\int_{E_i}f\dir \Big)\,\chi_{E_i}\,,\\
b_f^i&=\Big(f-\frac{1}{\rho(E_i)}\int_{E_i}f\dir \Big)\,\chi_{E_i}\,,\\
h&=f\,\chi_{(\cup_iE_i)^c}\,.
\end{align*}
By (iii), $|g_f|\leq \max\big\{2,\,\big(\rd \,C_N/{c_N}\big)^{2Q}\big\}\,\alpha$. Each function $b_f^i$ is supported in $E_i$ and its integral vanishes. By (iii) the $L^1$-norm of $b_f^i$ is
\begin{align*}
\|b_f^i\|_{L^1(\rho)}&\leq 2\,\int_{E_i}|f|\dir\\
&\leq 2\,\max\big\{2,\,\big(\rd \,C_N/{c_N}\big)^{2Q}\big\}\,\alpha\,\rho(E_i)\,,
\end{align*}
and
\begin{align*}
\sum_i\rho(E_i)&\leq \frac{1}{\alpha}\,\sum_i\int_{E_i}|f|\dir\\
&\leq\frac{\|f\|_{L^1(\rho)}}{\alpha}\,.
\end{align*}
It remains to define a suitable decomposition of the function $h$. 

Let $O^0_{\alpha}=\{x\in S:~M^{{\mathcal R}^0}h(x)>\alpha  \}$. For each point $x\in O^0_{\alpha}$ we choose a ball $B_x$ in $\mathcal R^0$ such that $\frac{1}{\rho(B_x)}\int_{B_x}|h|\dir >\alpha$ and 
\begin{align}\label{maximality}
\rho(B_x)&>(1/2)\sup\Big\{\rho(B):~B\in\mathcal R^0,\,x\in B,\,\frac{1}{\rho(B)}\int_B|h|\dir>\alpha\Big\}\,.
\end{align}
Now we select a disjoint subfamily of $\{B_x\}_x$.

We choose $B_{x_1}$ such that $\rho(B_{x_1})>(1/2)\,\sup\{\rho(B_x):~x\in O^0_{\alpha}\}\,.$ Next, suppose that $B_{x_1},...,B_{x_n}$ have been chosen. Then $B_{x_{n+1}}$ is chosen so that it is disjoint from $B_{x_1},...,B_{x_n}$ and $\rho(B_{x_{n+1}})>(1/2)\,\sup\{\rho(B_x):~x\in O^0_{\alpha},\, B_x\cap B_{x_i}=\emptyset,\,i=1,...,n\}\,.$\\
We have that $\bigcup_jB_{x_j}\subseteq O^0_{\alpha}\subseteq \bigcup_j\tilde{B}_{x_j}$, where $\tilde{B}_{x_j}=B(c_{x_j},\gamma \,r_{x_j})$ and $\gamma$ is the constant which appear in Section \ref{small}.\\
Indeed, each set $B_{x_j}$ is contained in $O^0_{\alpha}$ by construction. Moreover, for each point $x\in O^0_{\alpha}$ either $B_x=B_{x_{j_0}}\subset \tilde{B}_{x_{j_0}}$ for some index $j_0$ or $B_x\neq B_{x_j}$ for all $j$. In this case there exists an index $j_0$ such that $B_x\cap B_{x_{j_0}}\neq\emptyset$ and $\rho(B_x)\leq 2\,\rho (B_{x_{j_0}})$. Since $\mathcal R^0$ is nicely ordered, $x\in B_x\subseteq B_{x_{j_0}}^*$.

Now set $G_j=\tilde{B}_{x_j}\cap\big(\bigcup_{k<j}G_k\big)^c\cap\big(\bigcup_{\ell>j}B_{x_{\ell}}\big)^c\cap O^0_{\alpha}$. It is easy to check that the sets $G_j$ are mutually disjoint and that $B_{x_j}\subseteq G_j\subseteq \tilde{B}_{x_j}$, so that their measure are comparable. Moreover $\bigcup_jG_j=O^0_{\alpha}\,$.

Indeed, on the one hand, $ \bigcup_jG_j\subseteq O^0_{\alpha}\,$ by construction; on the other hand, if $x$ is in $O^0_{\alpha}$, then there exists an index $j_0$ such that $x$ is in $\tilde{B}_{x_{j_0}}$. Now either $x$ is in $B_{x_{\ell}}$ for some index $\ell>j_0$ (and then $x$ is in $G_{\ell}$) or $x$ is in $G_{k}$ for some index $k<j_0$ or $x$ is in $G_{j_0}$. 

We claim that
\begin{align}\label{claim}
\frac{1}{\rho(G_j)}\int_{G_j}|h|\dir&\leq \tilde{C}\,2^n\,\big(\gamma_2/{\gamma_1}\big)^2\,\nep^{Q(3\,\Rnd+\gamma+1)} \,\alpha\,,
\end{align}
where $\tilde{C}$ is the constant which appears in Section \ref{small}. To see this fact, we first observe that
\begin{align}\label{FjBj*}
\frac{1}{\rho(G_j)}\int_{G_j}|h|\dir&\leq \tC\,\frac{1}{\rho\big(\tB_{x_j}\big)}\int_{\tB_{x_j}}|h|\dir\,.
\end{align}
To estimate this average we shall distinguish two cases. 

First suppose that $\tB_{x_j}$ is in ${\mathcal R}^0$. Since $\rho(B_{x_j})\leq 2\,\rho(\tB_{x_j})$, by (\ref{maximality})
\begin{align}\label{caso1}
\frac{1}{\rho\big(\tB_{x_j}\big)}\int_{\tB_{x_j}}|h|\dir&\leq \alpha\,.
\end{align}
Next suppose that $\tB_{x_j}$ is not in ${\mathcal R}^0$. Then $1/2\leq\gamma\,r_{x_j}\leq\gamma/2$. Hence we may apply Lemma \ref{BintQ} to the ball $B_{x_j}^*$ and the family of \nondivisible big admissible sets $\{F_{\ell}\}_{\ell}$. Let $\mathcal I=\{\ell:~\tB_{x_j}\cap F_{\ell}\neq\emptyset\}$. Since $h$ is supported in $\bigcup_{\ell} F_{\ell}\,$, by Lemma \ref{BintQ} we obtain that
\begin{align}\label{caso2}
\frac{1}{\rho(\tB_{x_j})}\int_{\tB_{x_j}}|h|\dir&=\sum_{\ell\in\mathcal I}\frac{1}{\rho(\tB_{x_j})}\int_{\tB_{x_j}\cap F_{\ell}}|h|\dir\nonumber\\
&\leq \sum_{\ell\in\mathcal I}2^n\,\big(\gamma_2/{\gamma_1}\big)\,\nep^{Q(2\Rnd+\gamma/2)}\,\frac{1}{\rho(F_{\ell})}\int_{F_{\ell}}|h|\dir\nonumber\\
&\leq \,{\sharp\,\mathcal I}\,\cdot\,2^n\,\big(\gamma_2/{\gamma_1}\big)\,\nep^{Q(2\Rnd+\gamma/2)}\,\alpha\nonumber\\
&\leq 2^n\,\big(\gamma_2/{\gamma_1}\big)^2\,\nep^{Q(3\Rnd+\gamma+1)}\,\alpha\,.
\end{align}
By (\ref{FjBj*}), (\ref{caso1}) and (\ref{caso2}) the claim follows.

Each set $G_j$ is contained in the ball $\tB_{x_j}$ and $G_j^*=\{x\in S:~d(x,G_j)<r_{x_j}\}\subseteq B\big(c_{x_j},(\gamma +1) r_{x_j}\big)$; then there exists a constant $C^{*}$ such that  $\rho(G_j^*)\leq C^{*}\,\rho(G_j)$.

We now define the decomposition of $h$:
\begin{align}\label{secondadec}
g_h&=h\,\chi_{(O^0_{\alpha})^c}+\sum_j\Big(\frac{1}{\rho(G_j)}\int_{G_j}h\dir \Big)\,\chi_{G_j}\,,\nonumber\\
b_h^j&=\Big(h-\frac{1}{\rho(G_j)}\int_{G_j}h\dir \Big)\,\chi_{G_j}\,.
\end{align} 
By (\ref{claim}), $|g_h|\leq  \tC\,2^n\,\big(\gamma_2/{\gamma_1}\big)^2\,\nep^{Q(3\,\Rnd+\gamma+1)}\,\alpha$ on each set $G_j$ and $|g_h|=|h|\leq\alpha$ on $(O^0_{\alpha})^c$. Each function $b_h^j$ is supported in $G_j$ and its integral vanishes. The $L^1$-norm of the functions $b_h^j$ is\begin{align*}
\|b_h^j\|_{L^1(\rho)}&\leq 2\,\int_{G_j}|h|\dir\\
&\leq 2\,  \tilde{C}\,2^n\,\big(\gamma_2/{\gamma_1}\big)^2\,\nep^{Q(3\,\Rnd+\gamma+1)} \,\alpha       \,\rho(G_j)\,.
\end{align*}
and
\begin{align*}
\sum_j\rho(G_j)&\leq \rho(O^0_{\alpha})\\
&\leq \frac{1}{\alpha}\,\opnorm M^{\mathcal R^0}\opnorm_{L^1(\rho);\,L^{1,\infty}(\rho)}\,\|h\|_{L^1(\rho)}\\
&\leq \frac{1}{\alpha}\,\opnorm M^{\mathcal R^0}\opnorm_{L^1(\rho);\,L^{1,\infty}(\rho)}\,\|f\|_{L^1(\rho)}\,,
\end{align*}
since $M^{\mathcal R^0}$ is bounded from $L^1(\rho)$ to $L^{1,\infty}(\rho)$.

Then $f=g_f+g_h+\sum_ib_f^i+\sum_jb_h^j$ is a \CZ decomposition of the function $f$ at height $\alpha$, where the \CZ sets are the sets $\{E_i\}_i$ and $\{G_j \}_j$ . The \CZ constant of the space is 
\begin{align*}
\kappa_0=\max \Big\{&2,\,\big(\rd\,C_N/c_N\big)^{2Q},\,C_{N,\beta},\, \gamma, \,\Big(\frac{c_3+C_N}{c_N}\Big)^{2Q},\\
&\tC\,2^n\,\big(\gamma_2/{\gamma_1}\big)^2\,\nep^{Q(3\,\Rnd+\gamma+1)},\,C^{*},\,\opnorm M^{\mathcal R^0}\opnorm_{L^1(\rho);\,L^{1,\infty}(\rho)}\Big\}\,.
\end{align*}
\end{proof}

\section{The Hardy spaces and $BMO$ spaces on $ax+b\,$-groups}\label{Hardyax+b}
In Section \ref{CZdecgen} we introduced the Hardy spaces and the $BMO$ spaces in a generic \CZ space. 

Let $S=\RR^d\times \RR^+$ be an $ax+b\,$-group and let $\mathcal R$ denote the family of \CZ sets, which are admissible sets satisfying conditions (\ref{admb}) and (\ref{adms}). The family $\mathcal R$ satisfies the additional condition (C). 

Indeed, for all $k\geq 2$ we choose a number $r_k\geq \nep$ such that $r_k^{\be}\leq 2 ^k<r_k^{4\be}$. We define 
\begin{equation}\label{Rkm}
R^k_{m}=Q^k_m\times (1/{r_k},r_k)\qquad\forall m\in \ZZ^d\,,
\end{equation}
where $Q^k_m$ are the dyadic cubes in $\RR^d$. Set $\mathcal R'=\bigcup_{k\geq 2, m\in\ZZ^d}R^k_m$. It satisfies the following properties:
\begin{itemize}
\item[(i)] if $R^k_m\cap R^{\ell}_n\neq\emptyset$ and $k>\ell$, then $R^{\ell}_n\subseteq R^k_m$ ;
\item[(ii)] if $R=Q_{n}^{\ell}\times (a_0/r,a_0\,r)$ is an admissible set, then there exist  $k>\ell$ and $m\in\ZZ^d$ such that $R\subseteq R^k_m$. Indeed, we may choose $k\geq\ell$ such that $(a_0/r,a_0\,r)\subseteq (1/{r_k},r_k)$. In this case, there exists $m\in\ZZ^d$ such that $Q_{n}^{\ell}\subseteq Q^k_m$. 
\end{itemize}
Thus the condition (C) is satisfied.
 
So far we have defined spaces $H^{1,q}$ for all $1<q< \infty$. In this context we can also define a space $H^{1,\,\infty}$.
\begin{defi}
A $(1,\infty )$-atom on $S$ is a function $a$ in $L^1(\rho)$ such that
\begin{itemize}
\item [(i)] $a$ is supported in an admissible set $R$;
\item [(ii)] $\|a\|_{\infty}\leq  \rho(R)^{-1}$ if $q=\infty$;
\item [(iii)]$\int_S a\dir =0$\,.
\end{itemize}
\end{defi}
\begin{ossn}{\rm{
\begin{itemize}
\item[(1)] If $a$ is a $(1,\infty)$-atom supported in $R$, then
$$\|a\|_p=\Big(\int_R|a|^p\dir\Big)^{1/p}\leq \|a\|_{\infty}\,\rho(R)^{1/p}\leq \rho(R)^{-\frac{1}{p'}}\,.$$
\item[(2)] If $a$ is a $(1,\infty)$-atom, then $a$ is a $(1,q)$-atom for all $q\in (1,\infty)$.
\end{itemize}
}}
\end{ossn}
\begin{defi}
$H^{1,\infty }$ is the space of all functions $f$ in $ L^1(\rho)$ such that $f=\sum_j \lambda_j\, a_j$, where $a_j$ are $(1,\infty)$-atoms and $\lambda _j$ are complex numbers such that $\sum _j |\lambda _j|<\infty$. We denote by $\|f\|_{H^{1,\infty}}$ the infimum of $\sum_j|\lambda_j|$ over all decompositions $f= \sum_j\lambda_j\,a_j$, where $a_j$ are $(1,\infty)$-atoms. 
\end{defi}
Next we prove that all the spaces $H^{1,q}$, $1<q\leq \infty$, are equivalent. We have already remarked that each $(1,\infty)$-atom is a $(1,q)$-atom, for $q\in (1,\infty)$, hence it is in $H^{1,q}$. The following proposition shows that the converse is true: the proof is similar to the proof of \cite[Theorem A]{CW2} given by Coifman and Weiss in the case of spaces of homogeneous type.
\begin{prop}\label{atomo}
Let $q\in(1,\infty)$ and $a$ be a $(1,q)$-atom. Then $a$ is in $ H^{1,\infty}$ and there exists a constant $C_q$, which depends only on $q$, such that
$$\|a\|_{H^{1,\infty}}\leq C_q\,.$$
\end{prop}
\begin{proof}
Let $a$ be a $(1,q)$-atom supported in the admissible set $R$. We define $b:=\rho(R)\,a$. Note that $b\in L^q(\rho)$ and $\|b\|_q\leq \rho(R)^{1/q}$. 

Let $\alpha$ be a positive number such that $\alpha> \max\{ 1, M^{-1/q}\,2^{\frac{1}{q-1}}\} $, where $M$ is the number of admissible sets in which we may split up a given admissible set ($M=2^d$). It suffices to prove the following lemma.
\begin{lem}\label{nuovadec}
For all $n\in \NN$ there exist functions $a_{j_{\ell}}$, $h_{j_n}$ and admissible sets $R_{j_{\ell}}$, $j_{\ell}\in \NN ^{\ell}$, $\ell=0,...,n$ such that
$$b=\sum_{\ell=0}^{n-1}M^{\frac{\ell +1}{q}}\,2^{\ell}\,\alpha^{\ell +1}\sum _{j_{\ell}}\rho(R^*_{j_{\ell}})\,\,a_{j_{\ell}}+\sum _{j_n}h_{j_n}\,,$$
where the following properties are satisfied:
\begin{itemize}
\item[(i)] $a_{\jl}$ is a $(1,\infty)$-atom supported in $R_{\jl}$;
\item[(ii)] $h_{j_n}$ is supported in $R_{j_n}$ and $\int_S  h_{j_n}\dir=0$;
\item[(iii)] $\Big(\frac{1}{\rho(R_{j_n})}\int_{R_{j_n}}|h_{j_n}|^q\dir\Big)^{1/q}\leq M^{n/q}\,2^n\,\alpha^n$;
\item[(iv)]$\sum_{j_n}\|h_{j_n}\|^q_q\leq 2^{qn}\,\|b\|^q_q$;
\item[(v)] $|h_{j_n}(x)|\leq |b(x)|+M^{n/q}\,2^n\,\alpha ^n\,\chi _{R_{j_n}}(x)\qquad \forall x\in S$;
\item[(vi)] $\sum_{j_n}\rho(R^*_{j_n})\leq \kappa_0\, M^{-n+1}\,\alpha ^{-nq}\,\|b\|^q_q$,
\end{itemize}
where $\kappa_0$ is the \CZ constant of the space $S$.
\end{lem}
Before proving the Lemma \ref{nuovadec}, we conclude the proof of the Proposition \ref{atomo}. Let $H_n=\sum_{j_n}h_{j_n}$. We prove that $H_n\in L^1(\rho)$ and that its $L^1(\rho)$-norm tends to zero when $n$ tends to $\infty$. Indeed, by H\"older's inequality
\begin{align*}
\|H_n\|_1&\leq\sum_{j_n}\|h_{j_n}\|_1\\
&\leq\sum_{j_n}\rho(R_{j_n})^{1/{q'}}\,\|h_{j_n}\|_q\,,
\end{align*}
where $q'$ is the conjugate exponent of $q$. Now by properties (iii) and (vi) in Lemma \ref{nuovadec} we have that
\begin{align*}
\|H_n\|_1&\leq\sum_{j_n}\rho(R_{j_n})^{1/{q'}}\rho(R_{j_n})^{1/q}\,  M^{n/q}\,2^n\,\alpha ^n\\
&=\sum_{j_n}\rho(R_{j_n}) M^{n/q}\,2^n\,\alpha ^n\\
&\leq \kappa_0 \,M^{-n+1}\,\alpha ^{-nq}\,\|b\|^q_q\,M^{n/q}\,2^n\,\alpha^n\\
&\leq \kappa_0\,M \big(2\,M^{\frac{1-q}{q}}\,\alpha^{1-q}\big)^n \rho(R)\,,
\end{align*}
which tends to zero when $n$ tends to $\infty$, since $\alpha >M^{-1/q}\,2^{\frac{1}{q-1}}$.
Thus, 
$$b=\sum_{\ell=0}^{\infty}M^{\frac{\ell +1}{q}}\,2^{\ell}\,\alpha^{\ell +1}\sum _{j_{\ell}}\rho(R^*_{j_{\ell}})\,a_{j_{\ell}}\,,$$
where the series converges in $L^1(\rho)$. It follows that
$$a=\rho(R)^{-1}\,b=\rho(R)^{-1}\,\sum_{\ell=0}^{\infty}M^{\frac{\ell +1}{q}}\,2^{\ell}\,\alpha^{\ell +1}\sum _{j_{\ell}}\rho(R^*_{j_{\ell}})\,a_{j_{\ell}}\,.$$
By Lemma \ref{nuovadec} (vi) we have that
\begin{align*}
\rho(R)^{-1}\,\sum_{\ell=0}^{\infty}M^{\frac{\ell +1}{q}}\,2^{\ell}\,\alpha^{\ell +1}\sum _{j_{\ell}}\rho(R^*_{j_{\ell}})&\leq \rho(R)^{-1}\,\kappa_0\,\sum_{\ell=0}^{\infty}M^{\frac{\ell +1}{q}}\,2^{\ell}\,\alpha^{\ell +1}\,\rho(R)^{-1} \,M^{-\ell+1}\,\alpha ^{-\ell q}\,\|b\|^q_q\\
&\leq \kappa_0\,M^{\frac{q+1}{q}}\,\alpha \sum_{\ell=0}^{\infty}\big(2\,M^{\frac{1-q}{q}}\,\alpha^{1-q}\big)^{\ell}\\
&=C_q\,,
\end{align*}
because $\alpha >M^{-1/q}\,2^{\frac{1}{q-1}}$, where $C_q$ depends only on $\kappa_0,\, M,\, q,\, \alpha$.

Since $a_{j_{\ell}}$ are $(1,\infty)$-atoms, this shows that $a\in H^{1,\infty}$ and that
$$\|a\|_{H^{1,\infty}}\leq C_q \,,$$
as required.
\end{proof}
Thus, to conclude the proof of Proposition \ref{atomo} it remains to prove Lemma \ref{nuovadec}. 

\begin{proof}
Let $\mathcal P$ be a partition of $S$ in admissible sets which contains the set $R$. 
We prove the Lemma  by induction on $n$. 

{\bf{Step~$n=1$.}} We choose $R_0=R$. Since $b=\rho(R)\,a$ and $a$ is a $(1,q)$-atom we have that
\begin{align*}
\frac{1}{\rho(R)}\int_R |b|^q \dir &\leq \frac{1}{\rho(R)} \,\rho(R)^q\int_R |a|^q \dir\leq 1\leq \alpha^q\,.
\end{align*}
Now we split up the set $R$ in at most $M$ admissible subsets. If the average of $|b|^q$ on a subset is greater than $\alpha^q$, then we stop; otherwise we divide again the subset until we find sets on which the average of $|b|^q$ is greater than $\alpha ^q$. We denote by $\mathcal{C}$ the collection of stopping sets. Now we distinguish two cases.

{\it{Case~$\mathcal C \neq \emptyset$.}} Let $\mathcal C=\{R_i:~i\in\NN\}$. The average of $|b|^q$ on each set $R_i$ is comparable with $\alpha^q$. Indeed, by construction we have that
$$\frac{1}{\rho(R_i)}\int_{R_i}|b|^q\dir>\alpha^q\,.$$
On the other hand, there exists a set $R_i'$ which contains $R_i$ such that $\rho(R_i)\geq\frac{\rho(R_i')}{M}$ and $\frac{1}{\rho(R_i')}\int_{R_i'}|b|^q\dir\leq \alpha^q$. It follows that
$$\frac{1}{\rho(R_i)}\int_{R_i}|b|^q\dir\leq \frac{M}{\rho(R_i')}\int_{R_i'}|b|^q\dir\leq M \,\alpha^q\,.$$
Now we define 
\begin{align*}
g(x)&=\begin{cases}
b(x) & {\rm{if}}~ x\notin\bigcup _i{ R_i}\\
\frac{1}{\rho(R_i)}\int_{R_i}b\dir & {\rm{if}}~ x\in R_i
\end{cases}\\
h_i(x)&=\begin{cases}
0 & {\rm{if}}~ x\notin R_i\\
b(x)-\frac{1}{\rho(R_i)}\int_{R_i}b\dir & {\rm{if}}~x\in R_i\qquad\forall i\in \NN\,.
\end{cases}
\end{align*}
Obviously 
\begin{align}\label{passouno}
b=g+\sum_i h_i&=M^{1/q}\,\alpha\,\rho(R_0^*)\big(M^{-1/q}\,\alpha^{-1}\,\rho(R_0^*)^{-1}g\big)+\sum_i h_i\nonumber\\
&=M^{1/q}\,\alpha\,\rho(R_0^*)\,a_0+\sum_i h_i\,,
\end{align}
where $a_0=M^{-1/q}\,\alpha^{-1}\,\rho(R_0^*)^{-1}\,g$. 

First we prove that $a_0$ is a $(1,\infty)$-atom. Obviously $a_0$ is supported in $R$ and its average is equal to zero. Now we find pointwise estimates of $g$. For all $x\in R_i$
$$|g(x)|\leq \frac{1}{\rho(R_i)}\int_{R_i}|b|\dir\leq \frac{1}{\rho(R_i)}\rho(R_i)^{1/{q'}}\Big(\int_{R_i}|b|^q\dir\Big)^{1/q}\leq M^{1/q}\,\alpha\,.$$
In order to estimate $g$ on the complementary of $\bigcup_i{R_i}$, we may consider the operators $\mathcal E_j$ associated to the partition $\mathcal P$ as in (\ref{martingale}). For each $x\notin\bigcup_i {R_i}$ we have that
$$\mathcal E_j(|b|^q)(x)\leq \alpha ^q\qquad \forall j\in \NN\,;$$
so we obtain that 
\begin{align*}
|g(x)|&=|b(x)|\\
&=\lim_{j\to +\infty}|\mathcal E_jb(x)|\\
&\leq \lim_{j\to +\infty} \mathcal E_j|b|(x)\\
&\leq \lim_{j\to +\infty}(\mathcal E_j(|b|^q)(x))^{1/q}\\
&\leq \alpha\qquad {\rm{ for~almost~every~}}x\notin \bigcup_i{R_i}\, .
\end{align*}
This allows us to conclude that
$$\|a_0\|_{\infty}\leq M^{-1/q}\,\alpha^{-1}\,\rho(R_0^*)^{-1}\,M^{1/q}\,\alpha\leq \rho(R_0)^{-1}\,.$$
Thus $a_0$ is a $(1,\infty)$-atom. 

Now we verify that properties (ii),...,(vi) are satisfied by functions $h_i$. Each function $h_i$ is supported in $R_i$ and has average zero. Moreover
\begin{align}\label{hi}
\|h_i\|_q&\leq \|b\|_{L^q(R_i)}+\rho(R_i)^{1/q}\,\frac{1}{\rho(R_i)}\int_{R_i}|b|\dir\nonumber\\
&\leq \|b\|_{L^q(R_i)}+\rho(R_i)^{1/q}\,\rho(R_i)^{-1}\,\rho(R_i)^{1/{q'}}\,\|b\|_{L^q(R_i)}\nonumber\\
&=2\,\|b\|_{L^q(R_i)}\,.
\end{align}
By summing estimates (\ref{hi}) over $i\in\NN$ we obtain property (iv):
$$\sum_i\|h_i\|^q_q\leq 2^q\,\sum_i\|b\|^q_{L^q(R_i)}\leq 2^q\,\|b\|^q_q\,,$$
since the sets $R_i$ are mutually disjoint. From estimate (\ref{hi}) follows also property (iii):
$$\frac{1}{\rho(R_i)}\int_{R_i}|h_i|^q\dir\leq 2^q\,\frac{1}{\rho(R_i)}\int_{R_i}|b|^q\dir\leq M\,2^q\,\alpha^q\,.$$
The pointwise estimate (v) of $h_i$ is an easy consequence of H\"older's inequality, since for all $x\in R_i$
\begin{align*}
|h_i(x)|&\leq |b(x)|+\frac{1}{\rho(R_i)}\int_{R_i}|b|\dir\\
&\leq|b(x)|+\rho(R_i)^{-1}\,\rho(R_i)^{1/{q'}}\Big(\int_{R_i}|b|\dir\Big)^{1/q}\\
&\leq |b(x)|+M^{1/q}\,\alpha\\
&\leq |b(x)|+M^{1/q}\,2\,\alpha\,\chi _{R_{i}}(x)\,.
\end{align*}
It remains to prove property (vi):
\begin{align*}
\sum_i\rho(R^*_i)&\leq \kappa_0\sum_i\rho(R_i)\\
&\leq \kappa_0\,\alpha^{-q}\sum_i\int_{R_i}|b|^q\dir\\
&\leq \kappa_0\,\alpha^{-q}\,\|b\|^q_q\,.
\end{align*}
This concludes the proof of the first step in the case $\mathcal{C}\neq \emptyset$.

{\it{Case}} ~$\mathcal C=\emptyset$. In this case it suffices to define $R_0=R$, $g=b$ and $h_i=0$ for all $i\in\NN$. It follows that
$$b=M^{1/q}\,\alpha\,\rho(R_0^*)\,a_0\,,$$
where $a_0=M^{-1/q}\,\alpha^{-1}\,\rho(R_0^*)^{-1}\,b$. It is easy to verify that $|b(x)|\leq \alpha$ almost everywhere on $S$. Then
$$\|a_0\|_{\infty}\leq\rho(R_0^*)^{-1}\,M^{-1/q}\,\alpha^{-1}\,\alpha\leq \rho(R_0)^{-1}\,,$$and $a_0$ is a $(1,\infty)$-atom.

{\bf{Inductive step}}. Suppose that
$$ b=\sum_{\ell=0}^{n-1}M^{\frac{\ell +1}{q}}\,2^{\ell}\,\alpha^{\ell +1}\,\sum _{j_{\ell}}\rho(R^*_{j_{\ell}})a_{j_{\ell}}+\sum _{j_n}h_{j_n}\,,$$
where functions $a_{\jl},~h_{\jl}$ and sets $R_{\jl}$ satisfy properties (i),\ldots,(vi). We shall prove that a similar decomposition of $b$ holds with $n+1$ in place of $n$. To do this we decompose each function $h_{j_n}$. Let us choose a multiindex $j_n$. By property (iv) the average of $|h_{j_n}|^q$ on $R_{j_n}$ is less or equal to $M^n\,2^{nq}\,\alpha^{(n+1)q}\,.$ Now we split up the set $R_{j_n}$ in at most $M$ admissible subsets and we stop if we find a set on which the average of $|h_{j_n}|^q$ is greater than $M^n\,2^{nq}\,\alpha^{(n+1)q}\,.$ Let ${\mathcal{C}}_{j_n}$ be the collection of stopping sets. We distinguish two cases.

{\it{Case}}~${\mathcal{C}}_{j_n} \neq \emptyset$. Let ${\mathcal{C}}_{j_n}=\{R_{j_n,i}:~i\in \NN\}$. The average of $|h_{j_n}|^q$ on each set $R_{j_n,i}$ is comparable with $M^n\,2^{nq}\,\alpha^{(n+1)q}$. Indeed, by construction we have that
$$\frac{1}{\rho(R_{j_n,i})}\int_{R_{j_n,i}}|h_{j_n}|^q\dir>M^n\,2^{nq}\,\alpha^{(n+1)q}\,.$$
On the other hand, there exists a set $R_{j_n,i}'$ which contains $R_{j_n,i}$ such that $\rho(R_{j_n,i})\geq\frac{\rho(R_{j_n,i}')}{M}$ and $\frac{1}{\rho(R_{j_n,i}')}\int_{R_{j_n,i}'}|h_{j_n,i}|^q\dir\leq M^n\, 2^{nq}\, \alpha^{(n+1)q}$. It follows that
$$\frac{1}{\rho(R_{j_n,i})}\int_{R_{j_n,i}}|h_{j_n}|^q\dir\leq \frac{M}{\rho(R_{j_n,i}')}\int_{R_{j_n,i}'}|h_{j_n}|^q\dir\leq M^{n+1}\, 2^{nq} \,\alpha^{(n+1)q}\,.$$
Now we define 
\begin{align*}
g_{j_n}(x)&=\begin{cases}
h_{j_n}(x) & {\rm{if}}~x\notin \bigcup_i R_{j_n,i}\\
\frac{1}{\rho(R_{j_n,i})}\int_{R_{j_n,i}}h_{j_n}\dir & {\rm{if}}~ x\in R_{j_n,i}
\end{cases}\\
h_{j_n,i}(x)&=\begin{cases}
0 & {\rm{if}}~x\notin R_{j_n,i}\\
h_{j_n}(x)-\frac{1}{\rho(R_{j_n,i})}\int_{R_{j_n,i}}h_{j_n}\dir & {\rm{if}}~x\in R_{j_n,i}\,.
\end{cases}
\end{align*}
Obviously 
\begin{align}\label{passoinduttivo}
b=\sum_{\ell=0}^{n-1}M^{\frac{\ell +1}{q}}\,2^{\ell}\,\alpha^{\ell +1}\,\sum _{j_{\ell}}\rho(R^*_{j_{\ell}})a_{j_{\ell}}
&+M^{\frac{n +1}{q}}\,2^{n}\,\alpha^{n +1}\,\sum _{j_{n}}\rho(R^*_{j_{n}})a_{j_n}+\sum _{j_n,i}h_{j_n,i}\,,
\end{align}
where $a_{j_n}=M^{-\frac{n+1}{q}}\,2^{-n}\,\alpha^{-n-1}\,\rho(R_{j_n}^*)^{-1}\,g_{j_n}$. \\
First we prove that $a_{j_n}$ are $(1,\infty)$-atoms. Obviously $a_{j_n}$ is supported in $R_{j_n,i}$ and has average equal to zero. We now find pointwise estimates of $g_{j_n}$. For all $x\in R_{j_n,i}$
$$|g_{j_n}(x)|\leq \frac{1}{\rho(R_{j_n,i})}\int_{R_{j_n,i}}|h_{j_n}|\dir\leq M^{\frac{n+1}{q}}\,2^{n}\,\alpha^{n+1}\,.$$
In order to estimate $g_{j_n}$ on the complementary of $\bigcup_i{R_{j_n,i}}$, we observe that for each $x\notin\bigcup_i {R_{j_n,i}}$ we have that
$$\mathcal E_j( |h_{j_n}|^q )(x)\leq M^n\,2^{nq}\,\alpha ^{(n+1)q}\qquad\forall j\in \NN\,,$$
and so we obtain that 
\begin{align*}
|g_{j_n}(x)|&=|h_{j_n}(x)|\\
&=\lim_{j\to +\infty}|\mathcal E_jh_{j_n}(x)|\\
&\leq \lim_{j\to +\infty}(\mathcal E_j|h_{j_n}|^q(x))^{1/q}\\
&\leq M^{\frac{n+1}{q}}\,2^n\,\alpha^{n+1}\qquad{\rm{ for~almost~every~}}x\notin \bigcup_i R_{j_n,i}\, .
\end{align*}
This allows us to conclude that
$$\|a_{j_n}\|_{\infty}\leq M^{-\frac{n+1}{q}}\,2^{-n}\,\alpha^{-(n+1)}\,\rho(R_{j_n}^*)^{-1}\,M^{\frac{n+1}{q}}\,2^n\,\alpha^{n+1}\leq \rho(R_{j_n})^{-1}\,,$$
and then that $a_{j_n}$ are $(1,\infty)$-atoms. 

Now we verify that properties (ii),...,(vi) are satisfied by functions $h_{j_n,i}$. Each function $h_{j_n,i}$ is supported in $R_{j_n,i}$ and has average zero. Moreover
\begin{align}\label{hjn}
\|h_{j_n,i}\|_q&\leq \|h_{j_n}\|_{L^q(R_{j_n,i})}+\rho(R_{j_n,i})^{1/q}\,\frac{1}{\rho(R_{j_n,i})}\int_{R_{j_n,i}}|h_{j_n}|\dir\nonumber\\
&\leq \|h_{j_n}\|_{L^q(R_{j_n,i})}+\rho(R_{j_n,i})^{1/q}\,\rho(R_{j_n,i})^{-1}\,\rho(R_{j_n,i})^{1/{q'}}\,\|h_{j_n}\|_{L^q(R_{j_n,i})}\nonumber\\
&=2\,\|h_{j_n}\|_{L^q(R_{j_n,i})}\,.
\end{align}
Thus, by summing estimates (\ref{hjn}) over $(j_n,i)\in\NN^n \times \NN$, and using inductif hypothesis, we obtain property (iv):
$$\sum_{(j_n,i)}\|h_{j_n,i}\|^q_q\leq 2^q\,\sum_{j_n}\|h_{j_n}\|^q_{L^q(R_{j_n,i})}\leq 2^{q(n+1)}\,\|b\|^q_q\,,$$
since the sets $R_{j_n,i}$ are mutually disjoint. From estimate (\ref{hjn}) follows also property (iii):
\begin{align*}
\frac{1}{\rho(R_{j_n,i})}\int_{R_{j_n,i}}|h_{j_n,i}|^q\dir&\leq 2^q\,\frac{1}{\rho(R_{j_n,i})}\int_{R_{j_n,i}}|h_{j_n}|^q\dir\\
&\leq M^{n+1}\,2^{(n+1)q}\,\alpha^{(n+1)q}\,.
\end{align*}
The pointwise estimate (v) of $h_{j_n,i}$ is an easy consequence of the H\"older's inequality and inductif hypothesis, since for all $x\in R_{j_n,i}$
\begin{align*}
|h_{j_n,i}(x)|&\leq |h_{j_n}(x)|+\frac{1}{\rho(R_{j_n,i})}\int_{R_{j_n,i}}|h_{j_n}|\dir\\
&\leq|h_{j_n}(x)|+\rho(R_{j_n,i})^{-1}\,\rho(R_{j_n,i})^{1/{q'}}\Big(\int_{R_{j_n,i}}|h_{j_n}|\dir\Big)^{1/q}\\
&\leq |h_{j_n}(x)|+M^{\frac{n+1}{q}}\,2^{n}\,\alpha^{n+1}\\
&\leq |b(x)|+M^{{n}/{q}}\,2^{n}\,\alpha^{n}+M^{\frac{n+1}{q}}\,2^{n}\,\alpha^{n+1}\\&\leq |b(x)|+M^{\frac{n+1}{q}}\,2^{n+1}\,\alpha^{n+1}\,\chi _{R_{j_n,i}}(x)\,.
\end{align*}
It remains to prove property (vi):
\begin{align*}
\sum_{j_n,i}\rho(R^*_{j_n,i})&\leq\kappa_0\,\sum_{j_n,i}\rho(R_{j_n,i})\\
&\leq \kappa_0\,M^{-n}\,2^{-nq}\,\alpha^{-(n+1)q}\,\sum_{j_n,i}\int_{R_{j_n,i}} |h_{j_n}|^q\dir\\
&\leq \kappa_0\,M^{-n}\,2^{-nq}\,\alpha^{-(n+1)q}\,\sum_{j_n} \|h_{j_n}\|^q_q\\
&\leq \kappa_0\,M^{-n}\,\alpha^{-(n+1)q}\,\|b\|^q_q\,.
\end{align*}

{\it{Case ${\mathcal{C}}_{j_n}= \emptyset$.}} In this case it suffices to define $g_{j_n}=h_{j_n}$ and $h_{j_n,i}=0$ for all $i\in \NN$. 

This concludes the proof of the fact that (\ref{passoinduttivo}) gives a decomposition of $b$ which satisfies the required properties.

Then the lemma is proved by induction.
\end{proof}
{F}rom Proposition \ref{atomo} the following corollary follows.
\begin{coro}\label{coincidono}
For all $q\in (1,\infty)$, we have that $H^{1,q}=H^{1,\infty}$ and norms $\|\cdot\|_{H^{1,q}}$ and $\|\cdot\|_{H^{1,\infty}}$ are equivalent.
\end{coro}
Finally we can define the Hardy space on $S$ as follows.
\begin{defi}
The {\rm{Hardy space}} $H^1$ is the space $H^1=H^{1,q}$, for all $q\in (1,\infty]$. We denote by $\|\cdot\|_{H^1}$ the norm $\|\cdot\|_{H^{1,\infty}}$, which is equivalent to each norm $\|\cdot\|_{H^{1,q}}$, for all $q\in (1,\infty)$.
\end{defi}
We have defined $\Bp$ spaces on $S$ for $1< p<\infty$. We also define the space $BMO_1$.
\begin{defi}
$\mathcal{B}\mathcal{M}\mathcal{O}_1$ is the space of all functions in $L^1_{\rm{loc}}(\rho)$ such that
$$\sup_R\frac{1}{\rho(R)}\int_R|f-f_R|\dir <\infty$$
where the supremum is taken over all admissible sets. The space $BMO_1$ is the quotient of $\mathcal{B}\mathcal{M}\mathcal{O}_1$ module constant functions. It is a Banach space with the norm defined by
$$\|f\|_{BMO_1}=\sup\Big\{\frac{1}{\rho(R)}\int_R|f-f_R|\dir \,:~R ~{\rm{admissible~set}}\Big\}\,.$$
\end{defi}
\begin{ossn}{\rm{It is clear that $BMO_1$ possesses the properties (i)-(iv) of Remark \ref{BMOP}.}}
\end{ossn}
In Theorem \ref{duality} we proved that the space $\Bp$ is the dual of $H^{1,p'}$, for $1<p<\infty$. The same duality result holds for $BMO_1$ and $H^{1,\infty}$. To prove this fact we use the equivalence between $H^{1,\infty}$ and $H^{1,2}$, which holds in the particular case of $ax+b\,$-groups. This is the reason why we did not define the spaces $BMO_1$ and $H^{1,\infty}$ in a general \CZ space.
\begin{teo}\label{dualitybis}
The space $BMO_1$ is the dual of $H^{1,\infty}$.
\end{teo}
\begin{proof}
The proof that each $f$ in $ BMO_1$ represents a bounded linear functional $\ell_f$ on $H^{1,\infty}$ is the same as the proof of Theorem \ref{duality} and it is omitted.

Now we prove the converse inclusion. 

Let $\ell$ be in $ (H^{1,\infty})'$. Our purpose is to define a function $f$ in $BMO_1$ such that $\ell_f =\ell$ and $\|f\|_{BMO_1}\leq C\,\|\ell\|_{(H^{1,\infty})'}$. 

For each admissible set $R$ we denote by $L^2_{R}$ the set of all functions in $L^{2}(\rho)$ supported in $R$ and by $L^{2}_{R,\,0}$ the subspace of functions whose integral is zero. If $g$ is in $L^{2}_{R,\,0}$, then the function $a=\rho(R)^{-1}\,\|g\|_{2}^{-1}\,g$ is a $(1,2)$-atom. Then, since $H^{1,2}$ and $H^{1,\infty}$ are equivalent, the functional $\ell$ is bounded also on $H^{1,2}$ and
$$|\langle \ell,g\rangle |\leq \rho(R)\,\|g\|_{2}\,\|\ell\|_{(H^{1,2})'}\,.$$
This shows that $\ell$ is in $(L^{2}_{R,\,0})'$ and, since $L^{2}_{R,\,0}$ is an Hilbert space, there exists a function $f^R\in L^{2}_{R,\,0}$ such that $\|f^R\|_2\leq \rho(R)\,\|\ell\|_{(H^{1,2})'}$ and $\langle \ell,g\rangle =\int_Rf^R\,g\dir$ for all $g\in L^2_{R,\,0}$. 

In order to define the function $f$ which represents the functional $\ell$, let
$$f(x)=f^{R^k_m}(x)\qquad{\rm{if}}~x\in R^k_m\,,$$
where $R^k_m$ are the sets defined by (\ref{Rkm}). First we observe that this definition makes sense.

It remains to prove that $f\in BMO_1$ and that $\ell_f=\ell$. For each admissible set $R$, there exists a set $R^k_{m}$ such that $R\subseteq R^k_{m}$. For all function $g\in L^{2}_{R,0}$ we have that
$$\int_R f\,g\dir=\,\langle \ell,g\rangle \,,$$
and then the restriction of $f$ on $R$ is equal to $f^R$. In particular it follows that
\begin{align*}
\frac{1}{\rho(R)}\int_R |f|\dir &=\frac{1}{\rho(R)}\int_R |f^R|\dir\\
&\leq \frac{1}{\rho(R)}\|f^R\|_{1}\\
&\leq {\rho(R)}^{-1}\,\rho(R)^{1/2}\,\|f^R\|_{2}\\
&\leq C\,\|\ell\|\,.
\end{align*}
This shows that $f\in BMO_1$ and that $\|f\|_{BMO_1}\leq C\, \|\ell\|\,,$ as required.
\end{proof}
{F}rom Corollary \ref{coincidono} and Theorem \ref{dualitybis} the following result follows immediately.
\begin{coro}
For all $p\in [1,\infty)$, $\Bp=BMO_1$ and norms $\|\cdot\|_{\Bp}$ and $\|\cdot\|_{BMO_1}$ are equivalent.
\end{coro}
\begin{defi}
We denote by $BMO$ the space $BMO=\Bp$, for all $p\in [1,\infty)$ and by $\|\cdot\|_{BMO}$ the norm $\|\cdot\|_{BMO_1}$ which is equivalent to $\|\cdot\|_{BMO_p}$, for all $p\in (1,\infty)$.
\end{defi}

\section{The CZ decomposition in the product case}\label{CZdecompositionp}
Let $S=S''\times S''$ be the direct product of two \DR spaces. In this section we shall prove that $(S,\rho,d_{max})$ is a \CZ space: we shall generalize the results which we proved in Section \ref{CZdecextH} in the nonproduct case to the product case. This is a natural and simple generalization, except for the proof of Lemma \ref{partizionegrandep} below.

We first define admissible sets as products of admissible sets in $S'$ and $S''$. We start from big admissible sets.
\begin{defi}\label{admissibilityp}
A \emph{big admissible set} in $S$ is a set of the form $R'\times R''$, where $R'$ and $R''$ are big admissible sets in $S'$ and $S''$ of the form
\begin{align*}
R'&=Q_{\alpha}^{'\,k} \times (a_0'/r,a_0'\,r)\qquad {\rm{and}}\qquad R''= Q_{\beta}^{''\,\ell}\times (a_0''/r,a_0''\,r)\,,
\end{align*}
where $Q_{\alpha}^{'\,k}$ and $Q_{\beta}^{''\,\ell}$ are dyadic sets in $N'$ and $N''$ respectively (see Theorem \ref{dyadic}),
\begin{align*}%\label{admp}
a_0'^{1/2}r^{\beta}&\leq \rd ^{'\,k} < a_0'^{1/2}r^{4\beta}\,,\\
a_0''^{1/2}r^{\beta}&\leq \rd ^{''\,{\ell}} < a_0''^{1/2}r^{4\beta}\,,\end{align*}
and $\beta$ is a constant greater than $\max\{3/2,\,1/4+\log\rd ',\, 1/4+\log\rd '',\,1+\log\big(c_3'/c_N'\big),\,1+\log\big({c_3''}/{c_N''}\big) \}$ . 
\end{defi}
We now investigate some geometric properties of big admissible sets.
\begin{prop}\label{proprietap}
Let $R$ denote the big admissible set $R'\times R''$. The following hold:
\begin{itemize}
\item[(i)] $R\subseteq B_{d_{max}}\big((n_{\alpha }^{'\,k} , a_0')(n_{\beta }^{''\,\ell} , a_0''),\,C\,\log r\big)$;%, where $C_N=\max\{C_{N',\,\beta}, C_{N'',\,\beta}$;
\item[(ii)] %$c_{N'}^{2Q'}\,c_{N''}^{2Q''}\,|B_{N'}(0_{N'},1)|\,|B_{N''}(0_{N''},1)|\,(a_0'^{1/2}r^{\be})^{2Q'}\,(a_0''^{1/2}r^{\be})^{2Q''}\,\log^2 r \leq \rho(R)\leq C_{N'}^{2Q'}\,C_{N''}^{2Q''}\,|B_{N'}(0_{N'},1)|\,|B_{N''}(0_{N''},1)|\,(a_0'^{1/2}r^{4\be})^{2Q'}\,(a_0''^{1/2}r^{4\be})^{2Q''}\,\log^2 r;$
there exist constants $c_{N',N''}, C_{N',N''}$ such that
\begin{align*}
\rho(R)&\geq c_{N',N''}\,(a_0'^{1/2}r^{\be})^{2Q'}\,(a_0''^{1/2}r^{\be})^{2Q''}\,\log^2 r \,;\\
\rho(R)&\leq C_{N',N''}\,(a_0'^{1/2}r^{4\be})^{2Q'}\,(a_0''^{1/2}r^{4\be})^{2Q''}\,\log^2 r;
\end{align*}
\item[(iii)] let $R^*$ be the set $\{x\in S: d_{max}\big(x,R\big)<\log r\}$; there exists a constant $C^*$ such that 
$$\rho\big(R^*\big)\leq C^*\,\rho\big(R\big)\,.$$
\end{itemize}
\end{prop}
\begin{proof}
Recall that by Proposition \ref{proprieta}
$$R\subseteq B_{S'}\big((n_{\alpha }^{'\,k} , a_0'),C_{N',\be}\,\log r\big)\times B_{S''}\big((n_{\beta}^{''\,\ell} , a_0''),C_{N'',\be}\,\log r\big)\,,$$
and (i) follows.

Property (ii) is an immediate consequence of Proposition \ref{proprieta}.

To prove (iii) note that
\begin{align*}
R^*&=\{x'\in S':~d_{S'}(x',R')<\log r\}\times \{x''\in S'':~d_{S''}(x'',R'')<\log r\}\\
&=(R')^*\times (R'')^*\,.
\end{align*}
Again (iii) follows by Proposition \ref{proprieta}.
\end{proof}
We observe that most big admissible sets may be split up in a finite number of mutually disjoint smaller subsets which are still admissible. More precisely the following lemma holds. 
\begin{lem}\label{taglip}
Let $R$ denote the big admissible set $R'\times R''$.

Let $\,\rd'\,,\rd ''\,,\,M'\,,\,M''\,,\,n^{'\,k}_{\alpha}\,,\,n^{''\,\ell}_{\beta}\,,\,c_{N'}\,,\,c_{N''}\,,\,C_{N'}\,,\,C_{N''}\,$ be as in Theorem \ref{dyadic}. The following hold:
\begin{itemize}
\item[(i)] if either $\rd '^{k-1}\geq a_0'^{1/2}r^{\be}$ or $\rd ''^{\ell-1}\geq a_0''^{1/2}r^{\be}$, then there exist $J$ mutually disjoint big admissible sets $R_1,...,R_J$ such that $2\leq J\leq M'M''$, $R=\bigcup_{i=1}^J R_i$ and
$$\big(c_{N'}/{(\rd ' \,C_{N'})}\big)^{2Q'}\,\big(c_{N''}/{(\rd '' \,C_{N''})}\big)^{2Q''}\,\rho(R)\leq\rho(R_i)\leq \rho(R) \qquad i=1,...,J\,;$$
\item[(ii)] if $\rd'^{k-1}< a_0'^{1/2}r^{\be}$ and $\rd''^{\ell-1}< a_0''^{1/2}r^{\be}$ and $r\geq \nep ^2$, then there exist four disjoint big admissible sets $R_1,...,R_4$ such that $R=\bigcup_{i=1}^{4}R_i$ and $\rho(R_i)=\rho(R)/4$, for $i=1,...,4$;
\item[(iii)] if $\rd'^{k-1}< a_0'^{1/2}r^{\be}$ and $\rd''^{\ell-1}< a_0''^{1/2}r^{\be}$ and $r< \nep ^2$, then there exists a constant $\Rndp$ such that
\begin{align*}
B_{d_{max}}\big((n^{'\,k}_{\alpha},a_0')(n_{\beta}^{''\,\ell},a_0''),1\big)&\subseteq R\subseteq B_{d_{max}}\big((n^{'\,k}_{\alpha},a_0'),(n_{\beta}^{''\,\ell},a_0''),\Rndp\big)\,.
\end{align*}
\end{itemize}
\end{lem}
\begin{proof}
To prove this lemma it suffices to apply Lemma \ref{tagli} to the admissible sets $R'$ and $R''$.
\end{proof}
For later developments it is useful to distinguish between big admissible sets that satisfy condition (i) or (ii) in Lemma \ref{taglip}, which may be split up in a finite number of smaller big admissible sets, and big admissible sets that satisfy condition (iii) in Lemma \ref{taglip}, which cannot be split up in that way. 
\begin{defi}
A big admissible set $R=Q^{'\,k}_{\alpha}\times (a_0'/r,a_0'\,r)\times Q^{''\,\ell}_{\beta}\times (a_0''/r,a_0''/,r)$ is said to be {\emph{divisible}} if either $\rd'^{k-1}\geq a_0'^{1/2}r^{\be}$ or $\rd''^{\ell-1}\geq a_0''^{1/2}r^{\be}$ or $r\geq\nep^2$.

A big admissible set is said to be {\emph{nondivisible}} if $\rd'^{k-1}< a_0'^{1/2}r^{\be}$, $\rd''^{\ell-1}< a_0''^{1/2}r^{\be}$ and  $r<\nep^2$.
\end{defi}
Next we show that there exists a partition of $S$ which consists of  big admissible sets whose measure is as large as needed. The proof of this fact, as we will see below, is not a simple generalization of the proof of the analogue result in the nonproduct case. 
\begin{lem}\label{partizionegrandep}
For all $\sigma>0$ there exists a partition $\mathcal P_{\sigma}$ of $S$ which consists of big admissible sets whose measure is $>\sigma$.
\end{lem}
\begin{proof}
The proof of this lemma is quite technical. The idea is to fix a number $a>1$ large enough and define a partition of $\RR^2_+$ in ``squares of size $a^{2^{n-1}}$'', for $n\geq 0$. Then we consider the product of each ``square'' of $\RR^2_+$ with dyadic sets in $N'\times N''$ of suitable size, to obtain admissible sets in $S$.

We first introduce some more notation. Given an interval $I=(b,c)$ in $\RR^+$ we denote by $I^{\lor}$ the interval $(1/c,1/b)$. Given a ``square'' $E=(b',c')\times (b'',c'')$ in $\RR^2_+$ we denote by $\tilde{E}$ the ``transposed square'' $(b'',c'')\times (b',c')$. 

Now we define a partition of $\RR^2_+$ step by step.

The first step is to split up $(1/a^2,a^2)\times (1/a^2,a^2)$ in $16$ ``squares'' of the form $(a'/r,a'\,r)\times (a''/r,a''\,r)$, where $r=\sqrt a$ and $a',\,a''\in [r^{-3},r^{3}]$. To do it set $I_0=(1,a)$ and $I_1=(a,a^2)$ and consider $I_i\times I_j$, $I_i\times I^{\lor}_j$, $I_i^{\lor}\times I_j$ and $I_i^{\lor}\times I^{\lor}_j$, with $i,j=0,1$. 

As a second step, we set $I_0=(1,a^2),\,I_1=(a^2,a^4)$ and $I_{-1}=I_1^{\lor}$. Then we consider sets $E$ of the form $I_i\times I_0$, $I_i\times I_0^{\lor}$,  $I_i\times I_j$, with $i,j=1,-1$ and their transposed sets $\tilde{E}$. We obtain $12$ sets of the type $(a'/r,a'\,r)\times (a''/r,a''\,r)$, where $r=a$ and $a',\,a''\in [r^{-3},r^{3}]$.

Now suppose we have defined a partition of $(1/{a^{2^n}},\,{a^{2^n}})$ in squares of the type $(a'/r,a'\,r)\times (a''/r,a''\,r)$ such that $r=a^{2^{\ell-1}}$, $\ell\leq n-1$, and $a',\,a''\in [r^{-3},\,r^3]$. We define $I_0=(1,a^{2^{n}}),\,I_1=(a^{2^n},a^{2^{n+1}})$ and $I_{-1}=I_1^{\lor}$. We consider sets $E$ of the form $I_i\times I_0$, $I_i\times I_0^{\lor}$, $I_i\times I_j$, wth $i,j=1,-1$ and their transposed sets $\tilde{E}$. These are $12$ sets of the type $(a'/r,a'\,r)\times (a''/r,a''\,r)$, where $r=a^{2^{n-1}}$ and $a',\,a''=[r^{-3},r^3]$.

If we collect all the sets defined above, then we obtain a partition of $\RR^2_+$.

Now for each square of the partition $E=(a'/r,a'\,r)\times (a''/r,a''\,r)$, with $r\geq \sqrt a$ and $a',\,a''\in [r^{-3},\,r^3]$, we choose two integers $k$ and $\ell$ such that
$$a'^{1/2}\,r^{\beta}\leq\rd ^{'\,k}\leq a'^{1/2}\,r^{4\beta}\,,$$
$$a''^{1/2}\,r^{\beta}\leq\rd ^{''\,\ell}\leq a''^{1/2}\,r^{4\beta}\,.$$
The sets $Q^{'\,k}_{\alpha}\times Q^{''\,\ell}_{\beta}\times  E$ give a partition of $N'\times N''\times E$ in big admissible sets whose measure is greater than
\begin{align*}
a'^{Q'}\,a''^{Q''}\,r^{2\beta(Q'+Q'')}\,(\log r)^2&\geq \,r^{-3(Q'+Q'')}\,r^{2\beta(Q'+Q'')}\,(\log r)^2\\
&\geq  \,a^{(\beta-3/2)(Q'+Q'')}\,1/4\,\log ^2 a\\
&> \,\sigma\,,
\end{align*}
if $a$ is sufficently large.

Thus we have defined a partition of $S$ in big amissible sets whose measure is greater than $\sigma$.
\end{proof}
We now define small admissible sets.

A \emph{small admissible set} is a ball with radius $<1/2$. More precisely small admissible sets are sets of the family $\mathcal R^0$ which we defined in Section \ref{maxoperatorsp}.

Before proving the main result of this section, we prove a geometric lemma concerning  intersection properties between  balls and big \nondivisible sets. 
\begin{lem}\label{BintQp}
Let $B$ be a ball of radius $1/2\leq R\leq\gamma/2$. Let $\{F_{\ell}\}_{\ell}$ be a family of mutually disjoint \nondivisible big admissible sets. Then:
\begin{itemize}
\item[(i)] if $B\cap F_{\ell}\neq \emptyset$, then there exists a constant $C$ such that
$$\rho(B)\geq C\,\rho(F_{\ell})\,;$$
\item[(ii)] there exists a constant $L$ such that the ball $B$ intersects at most $L$ sets of the family $\{F_{\ell}\}_{\ell}\,$.
\end{itemize}
\end{lem}
\begin{proof}
The ball $B$ is equal to the product $B'\times B''$ of two balls in $S'$ and $S''$ and each set $F_{\ell}$ is the product $F_{\ell}'\times F_{\ell}''$ of two big nondivible sets in $S'$ and $S''$ respectively. It suffices to apply Lemma \ref{BintQ} to the ball $B'$ and the family $\{F_{\ell}'\}$ and to the ball $B''$ and the family $\{F_{\ell}''\}$. 
\end{proof}
\begin{teo}
Let $S=S'\times S''$ be the product of two \DR spaces. The space $(S,\rho,d_{max})$ is a \CZ space. 
\end{teo}
\begin{proof}
The proof is almost verbatim the same as the proof of Theorem \ref{CZd}. we just need to use Lemma \ref{partizionegrandep} and Lemma \ref{BintQp} in place of Lemma \ref{partizionegrande} and Lemma \ref{BintQ}.
\end{proof}

\chapter{Spectral multipliers for the Laplacian $\D$}
\begin{intro*}
In this chapter we introduce a distinguished left invariant Laplacian $\D$ on a \DR space $S$ which is essentially selfadjoint on $L^2(\rho)$. Then we study spectral multipliers for $\D$. 

Specifically we prove that if a bounded measurable function $M$ on $\RR^+$ satisfies a \MH condition both at infinity and locally, then the operator 
$M(\D)$ is bounded from $L^1(\rho)$ to $\lorentz{1}{\infty}{\rho}$ and on 
$L^p(\rho)$, for $1<p<\infty$. 

The strategy of the proof, which is similar to that of \cite[Theorem 2.4]{HS},
is to show that $M(\D)$ may be realized as a singular integral operator,
and that such operators are bounded 
from $L^1(\rho)$ to $\lorentz{1}{\infty}{\rho}$ and on 
$L^p(\rho)$, for $1<p<\infty$. 

In Section \ref{multipliersp} we consider a left invariant Laplacian $\D$ on the product of two \DR spaces $S=S'\times S''$: we prove a multiplier theorem for $\D$ which is analogue to the theorem which holds in the nonproduct case. 
\end{intro*}

\section{The Laplacian $\D$}
Let $S$ be the harmonic extension of an $H$-type group $N$.
 
Let $E_0,...,E_{n-1}$ be an orthonormal basis of the algebra $\s$ such that 
$E_0=H$, $E_1,...,E_{m_{\vg}}$ is an orthonormal basis of $\vg$ 
and $E_{m_{\vg}+1},...,E_{n-1}$ is an orthonormal basis of $\zg$. 
Let $X_0,X_1,...,X_{n-1}$ be the left invariant vector fields 
on $S$ which agree with $E_0, E_1,...,E_{n-1}$ at the identity.
Let $\Delta$ be the operator defined by  
$$
\Delta= -\sum_{i=0}^{n-1} X_i^2\,.
$$
The \emph{Laplacian} $\D$ is left invariant
and essentially selfadjoint on 
$C^{\infty}_c(S)\subset L^2(\rho).$ 
Therefore there exists a spectral resolution $E_{\D}$ of the identity for which
$$
\D f = \int_0^{\infty}t\di E_{\D}(t)f
\qquad\forall f \in \rm{Dom}(\D).
$$
By the spectral theorem, 
for each bounded measurable function $M$ on $\RR^+$ the operator
$M(\D)$ defined by
$$
M(\D) f 
= \int_0^{\infty}M(t)\di E_{\D}(t)f
\qquad\forall f \in \ld\rho,
$$
is bounded on $L^2(\rho)$; $M(\D)$ is called the \emph{spectral operator}
associated to the \emph{spectral multiplier}~$M$. 

A classical problem is to find conditions on 
$M$ which ensure that $M(\D)$ extends to a bounded operator 
from $L^1(\rho)$ to the Lorentz 
space $\lorentz{1}{\infty}{\rho}$ and on $L^p(\rho)$, for $1<p<\infty$. In this case we say that
$M$ is a $L^p(\rho)$ spectral multiplier for $\D$.

To solve this problem it is useful to consider the relationship between the Laplacian $\D$ and the Laplace-Beltrami operator $\LB$ on $S$. 

More precisely, let $\LQ$ denote the shifted operator $\LB-{Q^2}/{4}$; it is known \cite[Proposition 2]{A1} that
\begin{align}\label{relationship}
\delta^{-1/2}\Delta\,\delta^{1/2}f=\LQ f\,,
\end{align}
for all smooth compactly supported radial functions $f$ on $S$.\\
The spectra of $\LQ$ on $L^2(\lambda)$ and $\Delta$ on $L^2(\rho)$ are $[0,+\infty)$. Let $E_{\LB _Q}$ and $E_{\Delta}$ be the spectral resolution of the identity for which 
$$\LB _Q=\int_0^{+\infty}t\, \di E_{\LB _Q}(t)\qquad{\rm and}\qquad\Delta=\int_0^{+\infty}t\, \di E_{\Delta}(t)\,.$$
For each bounded measurable function $M$ on $\RR^+$ the operators $M(\LB _Q)$ and $M(\Delta)$, spectrally defined by  
$$M(\LB _Q)=\int_0^{+\infty}M(t) \di E_{\LB _Q}(t)\qquad{\rm and}\qquad M(\Delta)=\int_0^{+\infty}M(t) \di E_{\Delta}(t)\,,$$
are bounded on $L^2(\lambda)$ and $L^2(\rho)$ respectively. By (\ref{relationship}) and the spectral theorem, we see that 
$$\delta^{-1/2}M(\Delta)\,\delta^{1/2}f=M(\mathcal L _Q)f\,,$$ 
for smooth compactly supported radial functions $f$ on $S$.\\
Let $k_{M(\D)}$ and $k_{M(\LQ)}$ denote the convolution kernels of $M(\D)$ and $M(\LQ)$ respectively; then 
$$M(\LQ) f=f\ast k_{M(\LQ)}\qquad{\rm {and}}\qquad M(\D) f=f\ast k_{M(\D)}\qquad\forall f\in C^{\infty}_c(S)\,,$$
where $\ast$ denotes the convolution on $S$.
\begin{prop}\label{relazionenuclei}
Let $M$ be a bounded measurable function on $\RR^+$. Then $k_{M(\LQ)}$ is radial and $k_{M(\D)}=\du\,k_{M(\LQ)}$. The spherical transform of $k_{M(\LQ)}$ is
$$\mathcal H k_{M(\LQ)}(s)=M(s^2)\qquad\forall s \in\RR^+\,.$$
\end{prop}
\begin{proof}
See \cite{A1}, \cite{ADY}.
\end{proof}

\section{The multiplier theorem}
In this section we formulate our main result about spectral multipliers for the Laplacian $\D$. We need some more notation.

Let $M$ be a bounded measurable function on $\RR^+$. We denote by $K_{M(\D)}$ the integral kernel of the operator $M(\D)$ defined by
$$K_{M(\D)}(x,y)=k_{M(\D)}(y^{-1}x)\,\delta(y)\qquad\forall x,y\in S\,.$$
The reason for this definition is that if the convolution kernel $k_{M(\D)}$ is smooth, then 
\begin{align*}
M(\D)f(x)&=\int_SK_{M(\D)}(x,y)\,f(y)\dir (y)\\
&=f\ast k_{M(\D)}(x)\qquad \forall f\in C^{\infty}_c(S)\quad\forall x\in S\,.
\end{align*}
Note that
\begin{align*}
f\ast k_{M(\D)}(x)&=\int_Sf(xy^{-1})\,k_{M(\D)}(y)\dir (y)\\
&=\int_Sf(xy)\,k_{M(\D)}(y^{-1})\,\delta(y)\dir (y)\\
&=\int_SK_{M(\D)}(x,y)\,f(y)\dir (y)\qquad \forall f\in C^{\infty}_c(S)\quad\forall x\in S\,.
\end{align*}
Thus,
$$K_{M(\D)}(x,y)=k_{M(\D)}(y^{-1}x)\,\delta(y)\qquad\forall x,y\in S\,.$$

Now let $\psi$ be a function in $C^{\infty}_c(\RR^+)$, supported in $[1/4,4]$, such that
\begin{align}\label{sumpsi}
\sum_{j\in\ZZ}\psi(2^{-j}\lambda)&=1\qquad \forall \lambda\in\RR^+\,.
\end{align} 
We define $\|M\|_{0,s}$ and $\|M\|_{\infty,s}$ thus:
\begin{align*}
\|M\|_{0,s}&=\sup_{t<1}\|M(t\cdot)\,\psi(\cdot)\|_{H^s(\RR)}\,,\\
\|M\|_{\infty,s}&=\sup_{t\geq 1}\|M(t\cdot)\,\psi(\cdot)\|_{H^s(\RR)}\,,
\end{align*}
where $H^s(\RR)$ denotes the $L^2$-Sobolev space of order $s$ on $\RR$, i.e. the space of all measurable functions $f$ on $\RR$ such that
$$\|f\|_{H^s(\RR)}=\Big(\int_{\RR}|\hat{f}(\xi)|^2\,(1+|\xi|^2)^{s/2}\di \xi\Big)^{1/2}<\infty\,.$$
\begin{defi}
We say that a bounded measurable function $M$ defined on $\RR^+$ satisfies a {\rm{mixed \MH condition of order $(s_0,s_{\infty})$}} if $\|M\|_{0,s_0}<\infty$ and $\|M\|_{\infty,s_{\infty}}<\infty$.
\end{defi}
The next result, which is the main contribution of this chapter, gives a sufficient condition of \MH type for spectral multipliers of the Laplacian $\D$ to be bounded from $L^1(\rho)$ to $\lorentz{1}{\infty}{\rho}$ and on $L^p(\rho)$, for $1<p<\infty$.
\begin{teo}\label{moltiplicatori}
Let $S$ be a \DR space. Suppose that $s_0>\frac{3}{2}$ and $s_{\infty}>\max\left\{\frac{3}{2},\frac{n}{2}\right\}$, where $n$ denotes the dimension of $S$. If $M$ satisfies a mixed \MH condition of order $(s_0,s_{\infty})$, then $M(\D)$ extends to a bounded operator from $\lu{\rho}$ to $\lorentz{1}{\infty}{\rho}$ and on $L^p(\rho)$, for all $p$ in $ (1,\infty)$.
\end{teo}
{\bf{Sketch of the proof. }} Let $\varepsilon$ be such that $s_0>\frac{3}{2}+\varepsilon$ and $s_{\infty}>\max\left\{\frac{3}{2},\frac{n}{2}\right\}+\varepsilon$. We split up the proof into three steps.

{\bf{Step 1.}} Let $m$ be in $H^{s_0}(\RR)\cap H^{s_{\infty}}(\RR)$ supported in $[1/4,4]$. We shall prove that there exists a constant $C$ such that the integral kernel $K_{m(t\D)}$ satisfies the following estimate:
\begin{align}\label{stimaA}
\int_S|K_{m(t\D)}(x,y)|\,\big(1+t^{-1/2}d(x,y)\big)^{\varepsilon}\dir(x)&\leq  \begin{cases}
C\|m\|_{H^{s_0}(\RR)}& \forall~t\in [1,\infty)\\
C\|m\|_{H^{s_{\infty}}(\RR)}& \forall~t\in (0,1)\,\qquad \forall y\in S\,.
\end{cases}
\end{align}

{\bf{Step 2.}} Let $m$ and $K_{m(t\D)}$ be as in Step 1. Then there exists a constant $C$ such that
\begin{align}\label{stimaB}
&\int_S|K_{m(t\D)}(x,y)-K_{m(t\D)}(x,z)|\dir(x)\nonumber\\
\leq&\, \begin{cases}
C\,t^{-1/2}d(y,z)\,\|m\|_{H^{s_0}(\RR)}& \forall t\in [1,\infty)\\
C\,t^{-1/2}d(y,z)\,\|m\|_{H^{s_{\infty}}(\RR)}& \forall t\in (0,1)\,\qquad\forall y,z\in S\,.
\end{cases}
\end{align}
The proof of (\ref{stimaB}) hinges on an $L^1$-estimate of the gradient of the heat kernel associated to the Laplacian $\D$.

{\bf{Step 3.}} We show how Step 1 and Step 2 imply the conclusion of the theorem. Let $M$ be as in the statement of the theorem. Define 
$$m_j(\lambda)=M(2^j\lambda)\,\psi(\lambda)\qquad\forall j\in\ZZ\quad\forall\lambda\in\RR^+\,,$$
where $\psi$ is as in (\ref{sumpsi}). We observe that, at least formally,
$$M(\D)=\sum_{j\in\ZZ}m_j(2^{-j}\D)\,.$$
Since each function $m_j$ is supported in $[1/4,4]$ we may apply estimates (\ref{stimaA}) and (\ref{stimaB}) to $m_j$ and $t=2^{-j}$, to obtain that
\begin{align}\label{stimaAD}
\int_S|K_{m_j(2^{-j}(\D)}(x,y)|(1+2^{j/2}d(x,y))^{\varepsilon}\dir(x)&\leq  \begin{cases}
C\,\|M\|_{0,s_0}& \forall~j\leq 0\\
C\,\|M\|_{\infty,s_{\infty}}& \forall~j>0\quad\forall y\in S\,,
\end{cases}
\end{align}
and
\begin{align}\label{stimaBD}
&\int_S|K_{m_j(2^{-j}(\D)}(x,y)-K_{m_j(2^{-j}(\D)}(x,z)|\dir(x)\nonumber \\
\leq& \begin{cases}
C\,2^{j/2}\,d(y,z) \,\|M\|_{0,s_0}& \forall~j\leq 0\\
C\,2^{j/2}\,d(y,z) \,\|M\|_{\infty,s_{\infty}}& \forall~j>0\quad\forall y,z\in S\,.
\end{cases}
\end{align}
Then by Remark \ref{condHS} the operator $M(\D)$ satisfies the hypotheses of Theorem \ref{Teolim}, then it is bounded from $L^1(\rho)$ to $L^{1,\,\infty}(\rho)$, on $L^p(\rho)$ for all $p$ in $(1,2]$ and, by duality, for all $p$ in $[2,\infty)$. 

To conclude the proof of the theorem, we need to give full details of Step 1 and Step 2. This will be done in Subsections \ref{stima1section} and \ref{stima2section} below respectively.

\subsection{Step 1}\label{stima1section}
The proof of (\ref{stimaA}) is based on some technical lemmata and follows \cite{HEB}. The weight function $w$ on $S$ defined by
$$w(x)=\dum (x)\,\nep^{{Qd(x,e)}/2}\qquad\forall x\in S\,,$$
will play an important r\^ole in the sequel. 
\begin{lem}\label{peso}
There exists a constant $C$ such that the following hold:
\begin{itemize}
\item[(i)]$\int_{B_r}w^{-1}\,{\rm{d}}\rho\leq \begin{cases}
C\,r^2& \forall r\in [1,\infty)\\
C\,r^n& \forall r\in (0,1)\,;
\end{cases}$
\item[(ii)] for every compactly supported function  $f$ on $\RR^+$
\begin{align*}
\int_{B_r}|k_{f(\D)}|^2\,w\,{\rm{d}}\rho&\leq C\, (1+r)\, \int_{B_r}|k_{f(\D)}|^2\di\rho\,.
\end{align*}
\end{itemize}
\end{lem}
\begin{proof}
If $r<1$, then
$$\int_{B_r}w^{-1}\dir\leq C\,\rho(B_r)\leq C\,r^n\,.$$
If $r\geq 1$, then by Lemma \ref{intduf}
\begin{align*}
\int_{B_r}w^{-1}\dir&=\int_{B_r}\du (x)\,\nep^{{-Qd(x,e)}/2}\dir(x)\\
&=\int_0^{r}\phi _0(t)\,\nep^{-Qt/2}\,A(t)\di t\,,
\end{align*}
which, by (\ref{stimafi0}) and (\ref{pesoA}), is bounded above by
\begin{align*}
&C\int_0^{r}(1+t)\,\nep^{-Qt/2}\nep^{-Qt/2}\Big(\frac{t}{1+t}\Big)^{n-1}\nep^{Qt}\di t\\
\leq\,& C\,r^2\,.
\end{align*}
This concludes the proof of (i).

To prove (ii), let $f$ be compactly supported on $\RR^+$  and let $k_{f(\LQ)}$ denote the convolution kernel of the operator $f(\LQ)$. By Proposition \ref{relazionenuclei}, $k_{f(\D)}=\du \,k_{f(\LQ)}$.

We split up the ball $B_r$ into the annuli $A_R=\{x\in S:~R-1<d(x,e)<R\}$, $R=1,...,[r]$ and $C_r=\{x\in S:~[r]<d(x,e)<r\}$ and estimate the integral of $|k|^2\,w$ on $A_R$ and $C_r$ separately. 

By Lemma \ref{intduf} 
\begin{align*}
\int_{A_R}|k_{f(\D)}|^2\,w\,\dir&=\int_{A_R}|\du (x)\,k_{f(\LQ)}(x)|^2\,\dum (x)\,\nep^{{Qd(x,e)}/2}\dir(x)\\
&=\int_{R-1}^R\phi_0(t)\,|k_{f(\LQ)}(t)|^2\,\nep^{Qt/2}\,A(t)\di t\,,
\end{align*}
which, by (\ref{stimafi0}), is bounded above by
\begin{align}\label{AR}
& C\int_{R-1}^R (1+t)\,|k_{f(\LQ)}(t)|^2\,A(t)\di t\nonumber\\
\leq&\, C \,(1+r)\, \int_{A_R}|k_{f(\LQ)}|^2\dil\nonumber\\
=&\,C\,(1+r)\,\int_{A_R}|k_{f(\D)}|^2\dir\,.
\end{align}
The proof of the estimate
\begin{equation}\label{CR}
\int_{C_r}|k_{f(\D)}|^2\,w\dir\leq C\,(1+r)\,\int_{C_r}|k_{f(\D)}|^2\dir\,
\end{equation}
is similar and is omitted. Now (ii) follows immediately from (\ref{AR}) and (\ref{CR}).
\end{proof}
Our main purpose is to prove an $L^1$-estimate for the convolution kernel of a multiplier of the Laplacian $\D$. The following result gives an estimate of this type in a particular case.
\begin{lem}\label{norma1}
Let $f$ be an even function on $\RR$ such that its Fourier transform $\hat{f}$ is supported in $[-r,r]$. Then $k_{f(\sqrt{\D})}$ satisfies the following estimate:
\begin{align*}
\int_S|k_{f(\sqrt{\D})}|{\rm{d}}\rho&\leq\, \begin{cases} 
Cr^{n/2}\big(\int_0^{\infty}|f(s)|^2\,(s ^2+s^{n-1})\di s\big)^{1/2}& \forall r\in (0,1)\\
Cr^{3/2}\big(\int_0^{\infty}|f(s)|^2\,(s ^2+s^{n-1})\di s\big)^{1/2}& \forall r\in [1,\infty).
\end{cases}
\end{align*}
\end{lem}
\begin{proof}
Let $k_{f(\sqrt{\LQ})}$ denote the convolution kernel of the operator $f(\sqrt{\LQ})$. By Proposition \ref{relazionenuclei}, $k_{f(\sqrt{\D})}=\du\, k_{f(\sqrt{\LQ})}$ and $\mathcal Hk_{f(\sqrt{\LQ})}(s)=f(\sqrt{s ^2})=f(s)$ for all $s\in \RR^+\,.$ The hypothesis that the Fourier transform of $f$ is contained in $[-r,r]$ implies that $k_{f\left(\sqrt{\LQ}\right)}$ is supported in the ball $B_r$.

By H\"older's inequality we obtain that
\begin{align}\label{holder}
\int_S|k_{f(\sqrt{\D})}|\dir&=\int_{B_r}|k_{f(\sqrt{\D})}|\,w^{1/2}\,w^{-1/2}\dir\nonumber\\
&\leq \Big( \int_{B_r}w^{-1}\dir \Big)^{1/2}\Big( \int_{B_r}|k_{f(\sqrt{\D})}|^2\,w\dir   \Big)^{1/2}\,.
\end{align}
By Lemma \ref{peso} (ii) we have that
\begin{align*}
\Big( \int_{B_r}|k_{f(\sqrt{\D})}|^2\,w\dir   \Big)^{1/2}&\leq  C\, (1+r)^{1/2}\,\Big( \int_{B_r}|k_{f(\sqrt{\D})}|^2\dir  \Big)^{1/2}\\
&= C \,(1+r)^{1/2}\,\Big( \int_{B_r}|k_{f(\sqrt{\LQ})}|^2\dil   \Big)^{1/2}\\
&= C \,(1+r)^{1/2}\,\Big(\int_0^{\infty}|f(s)|^2\,|{\bf{c}}(s)|^{-2}\di s\Big)^{1/2}\\
&\leq C\,(1+r)^{1/2}\,\Big(\int_0^{\infty}|f(s)|^2\,(s^2+s^{n-1})\di s\Big)^{1/2}\,.
\end{align*}
Here we have used Plancherel formula and estimate (\ref{HC}) for the Plancherel measure. Thus, by (\ref{holder}) and 
Lemma \ref{peso} (i) we deduce that
\begin{align*}
\int_S|k_{f(\sqrt{\D})}|{\rm{d}}\rho&\leq\, \begin{cases} 
Cr^{n/2}\big(\int_0^{\infty}|f(s)|^2\,(s ^2+s^{n-1})\di s\big)^{1/2}& \forall r\in (0,1)\\
Cr^{3/2}\big(\int_0^{\infty}|f(s)|^2\,(s ^2+s^{n-1})\di s\big)^{1/2}& \forall r\in [1,\infty)\,,
\end{cases}
\end{align*}
as required.
\end{proof}
The previous result gives a good estimate for the $L^1$-norm of the kernel of a multiplier whose Fourier transform has compact support. Next lemma shows that every function $f$ supported in $[1/2,2]$ may be written as sum of functions whose Fourier transform has compact support (for the proof see \cite[Lemma 1.3]{HEB}).\begin{lem}\label{decomp}
Let $q$ and $Q$ be real numbers such that $0<q\leq Q$ and $f$ be a function in $H^s(\RR)$ supported in $[1/2,2]$. Then there exist even functions $f_{\ell}\,,~\ell\geq 0$, such that
\begin{itemize}
\item[(i)]$f=\sum_{\ell=0}^{\infty} f_{\ell}\,;$
\item[(ii)] ${\rm{supp}}(\hat{f}_{\ell})\subset[-2^{\ell},2^{\ell}]\,;$
\item[(iii)]$\int_0^{\infty}|f_{\ell}(s)|^2\,(s^{2q}+s^{2Q})\,{\rm{d}} s\leq C\, 2^{-2sl}\,\|f\|_{H^s(\RR)}\,.$
\end{itemize}
Let $f_t$ denote the dilated of $f$ defined by $f_t(\cdot)=f(t\cdot)$. Then
\begin{itemize}
\item[(i')]$f_t=\sum_{\ell} f_{\ell,t}\,$, where $f_{\ell,t}(\cdot)=f_{\ell}(t\cdot)\,;$
\item[(ii')] ${\rm{supp}}(\hat{f}_{\ell,t})\subset[-2^{\ell}t,2^{\ell}t]\,;$
\item[(iii')] $\int_0^{\infty}|f_{\ell,t}(s)|^2\,(s^{2q}+s^{2Q})\,{\rm{d}} s\leq \begin{cases}
C\,t^{-(2q+1)}\, 2^{-2s\ell}\,\|f\|_{H^s(\RR)}& \forall t\in [1,\infty)\\
C\,t^{-(2Q+1)} \,2^{-2s\ell}\,\|f\|_{H^s(\RR)}& \forall t\in (0,1).
\end{cases}$
\end{itemize}
\end{lem}

\bigskip
\begin{proof}{\bf{(of estimate (\ref{stimaA}))}}
First we observe that
\begin{align*}
&\int _S |K_{m(t\D)}(x,y)|\,\big(1+t^{-1/2}d(x,y)\big)^{\varepsilon}\dir(x)\\
=&\int_S |k_{m(t\D)}(y^{-1}\,x)|\,\delta(y)\,\big(1+t^{-1/2}d(y^{-1}\,x)\big)^{\varepsilon}\dir(x)   \\
=&\int_S|k_{m(t\D)}(x)|\,\big(1+t^{-1/2}d(x,e)\big)^{\varepsilon}\dir(x)\qquad\forall y\in S\,.
%&=\int_S |k_{m(t\D)}(y^{-1}x)|\,\delta(y)\,\big(1+t^{-1/2}d(y^{-1}x,e)\big)^{\varepsilon}\,\delta ^{-1}(x)\dil(x)\\
%=&\int_S|k_{m(t\D)}(x)|(1+t^{-1/2}d(x,e))^{\varepsilon}\delta^{-1}(x)d_lx\nonumber\\
%&=\int_S|k_{m(t\D)}(x)|\,\big(1+t^{-1/2}d(x,e)\big)^{\varepsilon}\dir(x)\qquad\forall y\in S\,.
\end{align*}
Now it suffices to define $f(s)=m(s^2)$ for all $s\in \RR^+$. The function $f$ is supported in $[1/2,2]$ and the operator $m(t\D)$ agrees with the operator $f(t^{1/2}\sqrt{\D})$. By applying Lemma \ref{decomp} with $q=1$ and $Q={(n-1)}/{2}$, we find functions $f_{\ell,t^{1/2}}$ such that $f(t^{1/2}\cdot)=\sum_{\ell} f_{\ell,t^{1/2}}(\cdot)$ and ${\rm{supp}}(\hat{f}_{\ell,t^{1/2}})\subset [-2^{\ell}t^{1/2}, 2^{\ell}t^{1/2}]$. Then we can apply Lemma \ref{norma1} to each function $f_{\ell,t^{1/2}}$ and sum these esti\-ma\-tes up. We distinguish the cases $t\geq 1$ and $t<1$.

{\it{Case}}~$t< 1$. In this case the quantity $2^{\ell}t^{1/2}$ is $\geq 1 $ if $\ell\geq (1/2)\,\log (1/t)$ and $<1$ otherwise. By applying again Lemma \ref{norma1} we have that
\begin{align*}
&\int_S|k_{f_{\ell,t^{1/2}}(\sqrt{\D})}(x)|\,\big(1+t^{-1/2}d(x,e)\big)^{\varepsilon} \dir(x)\\
=\,& \int_{B(e,2^{\ell}t^{1/2})}|k_{f_{\ell,t^{1/2}}(\sqrt{\D})}(x)|\,\big(1+t^{-1/2}d(x,e)\big)^{\varepsilon} \dir(x)\\
\leq\,& C\,(1+t^{-1/2}2^{\ell}t^{1/2})^{\varepsilon}\,(2^{\ell}t^{1/2})^{\max \{3/2,n/2\}}\,\Big(\int_0^{\infty}|f_{\ell,t^{1/2}}(s)|^2\,(s^2+s^{n-1})\di s\Big)^{1/2}\,,
\end{align*}
which by Lemma \ref{decomp} (iii') is bounded above by
\begin{align*}
& C\,2^{\ell\varepsilon}\,(2^{\ell}t^{1/2})^{\max \{3/2,n/2\}}\,t^{-n/4}\,2^{-s_{\infty}\ell}\,\|f\|_{H^{s_{\infty}}(\RR)}\\
\leq\,  &C\,2^{\ell(\max\{3/2,n/2\}+\varepsilon-s_{\infty})}\,\|f\|_{H^{s_{\infty}}(\RR)}\,.
\end{align*}
If we sum over $\ell\geq 0$ we obtain that
\begin{align*}
\int_S  |k_{m(t\D)}(x)|\,\big(1+t^{-1/2}d(x,e)\big)^{\varepsilon}\dir(x)&=\int_S |k_{f(t^{1/2}\sqrt{\D})}(x)|\,\big(1+t^{-1/2}d(x,e)\big)^{\varepsilon}\dir(x)\\&\leq C \,\|f\|_{H^{s_{\infty}}(\RR)}\\
&=C\,\|m(\cdot^2 )\|_{H^{s_{\infty}}(\RR)}\\
&\leq C\,\|m\|_{H^{s_{\infty}}(\RR)} \qquad \forall t\in (0,1)\,,
\end{align*}
since $s_{\infty}>\max \{3/2,n/2\}+\varepsilon$.

{\it{Case}} ~$t\geq 1$. By arguing very much as in the case where $t<1$, we obtain that
\begin{align*}
\int_S  |k_{m(t\D)}(x)|\,\big(1+t^{-1/2}d(x,e)\big)^{\varepsilon}\dir(x)&\leq C\,\|m\|_{H^{s_0}(\RR)}\,,
\end{align*}
as required.
\end{proof}

\subsection{Step 2}\label{stima2section}
To prove (\ref{stimaB}) we define $\mu(s)=m(s)e^{-s}$ for all $s\in \RR$. An easy calculation shows that 
$$k_{m(t\D)}=h_t\ast k_{\mu(t\D)}\qquad\forall t\in\RR^+\,,$$
and
\begin{align}\label{calcolo}
K_{m(t\D)}(x,y)=\int_S K_{\mu(t\D)}(x,u)\,H_t(u,y)\dir(u)\qquad\forall x,y\in S\,.\end{align}
where $h_t$ and $H_t$ denote the convolution and integral heat kernels associated to the Laplacian $\D$, respectively. %Let $q_t$ denote the heat kernel of the operator $\LQ$. It is well known that $h_t=\du\,q_t$ and that $\mathcal Hq_t(s)=\nep^{-ts^2}$ for all $s\in\RR^+$. It follows that
%\begin{align*}
%h_t=\mathcal A^{-1}\mathcal F^{-1}(\nep^{-ts^2})=\mathcal A^{-1}(h_t^{\RR})\,,
%\end{align*}
%where $h_t^{\RR}$ denote the heat kernel on $\RR$. An easy calculation shows that
%\begin{align}\label{calcolo}
%K_{m(t\D)}(x,y)=\int_S K_{M(t\D)}(x,u)\,H_t(u,y)\dir(u)\qquad\forall x,y\in S\,.\end{align}
This fact is useful to prove the following proposition. 
\begin{prop}\label{stima2}
There exists a constant $C$ such that
\begin{align*}
\int_S|K_{m(t\D)}(x,y)-K_{m(t\D)}(x,z)| {\rm{d}}\rho (x)&\leq \begin{cases}
C\,d(y,z)\,\|m\|_{H^{s_0}(\RR)}\,\|\nabla h_t\|_{L^1(\rho)}& \forall t\in [1,\infty)\\
C\,d(y,z)\,\|m\|_{H^{s_{\infty}}(\RR)}\,\|\nabla h_t\|_{L^1(\rho)}& \forall t\in (0,1)\,.
\end{cases}
\end{align*}
\end{prop}
\begin{proof}
By (\ref{calcolo}) we have that for all $y,z$ in $S$
\begin{align}\label{primopasso}
&\int_S |K_{m(t\D)}(x,y)-K_{m(t\D)}(x,z)|\dir(x)\nonumber\\
\leq\,&\int_S|H_t(u,y)-H_t(u,z)|\left(\int_S |K_{\mu(t\D)}(x,u)|\dir(x)\right)\dir(u)\nonumber\\
=&\,\|k_{\mu(t\D)}\|_{L^1(\rho)}\int_S|H_t(u,y)-H_t(u,z)|\dir(u)\,.
\end{align}
%Using the expression of $K_{M(t\D)}$ we obtain that
%$$\int_S |K_{M(t\D)}(x,u)|\dir(x)=\|k_{M(t\D)}\|_{L^1(\rho)}\,.$$
%\begin{align*}
%\int_S |K_{M(t\D)}(x,u)|\dir(x)&=\int_S |k_{M
%(t\D)}(u^{-1}x)|\,\delta(u)\dir(x)\\
%&=\int_S |k_{M(t\D)}(x)|\dir(x)\\
%&=\|k_{M(t\D)}\|_{L^1(\rho)}\,.
%\end{align*}
By applying (\ref{stimaA}) to the operator $\mu(t\D)$ we obtain that
\begin{align}\label{KM}
\|k_{\mu(t\D)}\|_{L^1(\rho)}&\leq  \begin{cases}
C\,\|\mu\|_{H^{s_0}(\RR)}& \forall t\in [1,\infty)\\
C\,\|\mu\|_{H^{s_{\infty}}(\RR)}& \forall t\in (0,1)
\end{cases}\nonumber\\
&\leq \begin{cases}
C\,\|m\|_{H^{s_0}(\RR)}& \forall t\in [1,\infty)\\
C\,\|m\|_{H^{s_{\infty}}(\RR)}& \forall t\in (0,1)\,.
\end{cases}
\end{align}
Since the operator $\nep^{t\D}$ is symmetric, $H_t(u,y)=H_t(y,u)$ for all $u,y\in S$. Thus
\begin{align}\label{simm}
\int_S |H_t(u,y)-H_t(u,z)|\dir(u)
&=\int_S |H_t(y,u)-H_t(z,u)|\dir(u)\nonumber\\
&=\int_S |h_t(uv) -h_t(u) |\dir(u)\,,
\end{align}
where $v=z^{-1}y$. Let $\gamma\,:~[0,1]\to S$ be a curve in $S$ such that $\gamma(0)=e$, $\gamma(1)=v$ and
$|\dot{\gamma}(\sigma)|=d(v,e)$ for all $\sigma\in [0,1]$. Then
\begin{align*}
h_t(uv) -h_t(u)&=h_t(u\gamma(1))-h_t(u\gamma(0))\\
&=\int_0^1\frac{d}{d\tau}\Big\lvert_{ \tau =\sigma}
h_t(u\gamma(\tau))\di\sigma\\
&=\int_0^1\langle \nabla h_t(u\gamma(\sigma)),\dot{\gamma}(\sigma)\rangle_{\gamma(\sigma)}\di\sigma \,, 
\end{align*}
where the gradient and the inner product are calculated with respect to the Riemannian structure of $S$. Therefore
\begin{align*}
| h_t(uv) -h_t(u)|&\leq\int_0^1|\dot{\gamma}(\sigma)|\,|\nabla h_t(u\gamma(\sigma))|\di\sigma\\
&\leq d(v,e)\int_0^1|\nabla h_t(u\gamma(\sigma))|\di\sigma\,.
\end{align*}
By using this estimate and (\ref{simm}) we obtain that
\begin{align}\label{HT}
\int_S |H_t(y,u)-H_t(z,u)|\dir(u)&=\int_S | h_t(uv) -h_t(u)| \dir(u)\nonumber\\
&\leq d(v,e)\int_S\int_0^1 |\nabla h_t(u\gamma(\sigma))|\,\di\sigma\dir(u) \nonumber\\
&=d(y,z)\int_S |\nabla h_t(u)|\dir(u)\,.              
\end{align}
By combining (\ref{primopasso}), (\ref{KM}) and (\ref{HT}) we obtain the required estimate.
\end{proof}
To conclude the proof of (\ref{stimaB}) it suffices to prove the following $L^1$-estimate of the gradient of the heat kernel $h_t$. 
\begin{prop}\label{norma}
There exists a constant $C$ such that 
\begin{align*}
\int_S |\nabla h_t| \dir&\leq C \,t^{-1/2}\qquad \forall t\in\RR^+\,.
\end{align*}
\end{prop}
Before proving Proposition \ref{norma}, we need some technical results which are contained in Lemma \ref{pezzi}, Lemma \ref{X_i} and Lemma \ref{I,I'}.

For all $x$ in $S$ define $|x|=d(x,e)$ and $\chQ (x)=\cosh^{-Q}(|x|/2)$. 

Let $h_t$ and $q_t$ denote the heat kernels associated to the operators $\D$ and $\LQ$, respectively. By Proposition \ref{relazionenuclei}, $h_t=\du\,q_t$ and $\mathcal Hq_t(s)=\nep^{-ts^2}=\mathcal F \htR(s)$ for all $s\in\RR^+$, where $h_t^{\RR}$ denotes the heat kernel on $\RR$. Then,
\begin{align}\label{hthtR}
h_t(x)&=\du(x)\,\big(\mathcal A^{-1}\circ\mathcal F^{-1}\big)\,(\mathcal F \htR)(|x|)\nonumber\\
&=\du(x)\, \mathcal A^{-1}(h_t^{\RR})(|x|)\,\qquad \forall x\in S\,.
\end{align}
In the following lemma we recall various estimates of $ h_t^{\RR}$ and its derivatives (see \cite[Proposition 5.22]{ADY}). 
\begin{lem}\label{pezzi}
For all $r$ in $\RR^+$ and $t$ in $\RR^+$ the following hold:
\begin{itemize}
\item[(i)]for all integer $j\geq 1$, there exists a constant $C$, independent of $t$ and $r$, such that 
$$\,r^j\,\htR(r)\leq C\,t^{j/2}\,h_{2t}^{\RR}(r)\,;$$
\item[(ii)] let $\mathcal D_1$ and $\mathcal D_2$ be the differential operators defined by
\begin{align*}
\mathcal D_1=\,-\frac{1}{\sinh r}\,\frac{\partial}{\partial r}\,,\qquad \mathcal D_2=\,-\frac{1}{\sinh(r/2)}\,\frac{\partial}{\partial r} \,.
\end{align*} 
For all integer $p,q\geq 0$
$$\mathcal D_1^q\,\mathcal D_2^p\,(\htR)(r)=\sum_{j=1}^{p+q}\,t^{-j}\,a_j(r)\,\htR(r)\,,$$
where 
\begin{align*}
a_j(r)&=\cosh^{-(p+2q)} (r/2)\,\big(\alpha _j\,r^j+f_j(r)\big)\,,
%a_j'(r)&=\beta_j\,r^j\cosh^{-(p+2q)} (r/2)+\cosh^{-(p+2q)} (r/2)\,g_j(r)\,,
\end{align*}
$f_j$, $f_j'$ are bounded functions on $\RR^+$ and $\alpha_j$ are constants.
\end{itemize}
\end{lem}
%We also prove a technical lemma which concerns the derivatives along the vector fields $X_i$ of some quantities that we shall use in future developments. 
%For future developments we need to estimate the derivatives of some functions along the vector fields $X_i$ and some integrals on $S$. We group all these estimates in the following lemma.
In the two next lemmata we prove various integral estimates, which we need in the proof of Proposition \ref{norma}.
\begin{lem}\label{X_i}
For all $i=0,...,n-1$ the following hold:
\begin{itemize}
\item[(i)] $\left|X_i\big(|\cdot|\big)\right|\leq 1$;
\item[(ii)] $t^{-1/2}\int_Sh_{2t}^{\RR}(|(x)|)\,\big|X_i\big(\du\,\chQ\big)(x)  \big|\dir (x)\leq C\,t^{-1/2}\qquad\forall t\in\RR^+\,;$
\item[(iii)] $t^{-1}\int_{B_1^c}\du(x)\,\chQ(x)\,\htR(|x|)\dir(x)\leq C\,t^{-1/2}\qquad\forall t\in\RR^+\,.$
\end{itemize}
\end{lem}
\begin{proof}
For the proof of (i) see \cite{HU}.\\
To prove (ii), recall that $\du\big((X,Z,a)\big)=a^{- Q/2}$ and, by (\ref{distanza}) 
\begin{align*}
\chQ\big((X,Z,a)\big)&=2^Q\,a^{Q/2}\,\big[(a+1+|X|^2/4)^2+|Z|^2\big]^{-Q/2}\,.
\end{align*}
Thus
\begin{align*}
(\du \,\chQ)\big((X,Z,a)\big)&=2^Q\,\big[(a+1+|X|^2/4)^2+|Z|^2\big]^{-Q/2}\,.
\end{align*}
By deriving along the vector field $X_0$ we obtain that
\begin{align*}
\big|X_0\big(\du\,\chQ\big)(X,Z,a)  \big|&\leq C\,\,\frac{a\,(a+1+|X|^2/4)}{\big[(a+1+|X|^2/4)^2+|Z|^2\big]^{Q/2+1}}\\
&\leq C\,\frac{a\,(a+1+|X|^2/4)^{-Q-1}}{\big[1+ (a+1+|X|^2/4)^{-2} \,|Z|^2\big]^{Q/2+1}}\\
\end{align*}
Since $h_{2t}^{\RR}(r)=C\,t^{-1/2}\,\nep^{-r^2/{8t}}$ and $|\log a|<|(X,Z,a)|$, we have that
\begin{align*}
&t^{-1/2}\int_S h_{2t}^{\RR}(|(X,Z,a)|)\,\big|X_0\big(\du\,\chQ\big)(X,Z,a)  \big|\dir (X,Z,a)\\
\leq\, & C\,t^{-1}\int_{\RR^+}e^{-\frac{(\log a)^2}{8t}}\,\int_N \frac{a\,(a+1+|X|^2/4)^{-Q-1}}{\big[1+ (a+1+|X|^2/4)^{-2} \,|Z|^2\big]^{Q/2+1}}  \,a^{-1}{\di X \di Z \di a}\,,
\end{align*}
which, changing variables ($(a+1+|X|^2/4)^{-1}\,Z=W$), transforms into
\begin{align*}
&\,C\,t^{-1}\int_{\RR^+}a\,\nep^{-\frac{(\log a)^2}{8t}}\int_{\vg} (a+1+|X|^2/4)^{-Q-1-m_{\zg}}\di X\,\int_{\zg}\frac{\di W}{(1+|W|^2)^{Q/2+1}}\,a^{-1}{\di a}\\
\leq& \,C\,t^{-1}\int_{\RR^+}\nep^{-\frac{(\log a)^2}{8t}}\int_{\vg} (a+1+|X|^2/4)^{-m_{\vg}/2-1}\di X {\di a}\\
=& \,C\,t^{-1}\int_{\RR^+}\nep^{-\frac{(\log a)^2}{8t}}\,(a+1)^{-m_{\vg}/2-1}\int_{\vg} (1+(a+1)^{-1/2}|X|^2/4)^{-m_{\vg}/2-1}\di X {\di a}\\
=& \,C\,t^{-1}\int_{\RR^+}\nep^{-\frac{(\log a)^2}{8t}}\,(a+1)^{-1}\int_{\vg} (1+|X|^2/4)^{-m_{\vg}/2-1}\di X {\di a}\\
\leq& \,C\,t^{-1}\int_{\RR^+}\nep^{-\frac{(\log a)^2}{8t}}\,a^{-1} \di a\,.
\end{align*}
The last integral, changing variables ($\log a=s$), transforms into
\begin{align*}
&\,C\,t^{-1}\,\int_{\RR}e^{-{s^2}/{8t}}\di s \\
=&\,C\,t^{-1/2}\,,
\end{align*}
as required. The proof of (ii) for $i=1,\ldots,n-1$ is similar and omitted.

To prove (iii), we use Lemma \ref{intduf} and apply (\ref{stimafi0}) and (\ref{pesoA}):
\begin{align*}
t^{-1}\int_{B_1^c}\du(x)\,\chQ(x)\,\htR(|x|)\dir(x)&\leq C\,t^{-3/2}\int_1^{\infty}\phi _0(r)\,\cosh^{-Q}(r/2)\,\nep^{-{r^2}/{4t}}\,A(r)\di r\\
&\leq C\,t^{-3/2}\int_1^{\infty}r\,\nep^{-Qr/2}\,\nep^{-Qr/2}\,\nep^{-{r^2}/{4t}}\,\nep^{Qr}\di r\\
&\leq C\,t^{-3/2}\int_1^{\infty}r\,\nep^{-{r^2}/{4t}}\di r\\
&\leq C\,t^{-1/2}\,,
\end{align*}
as required.  
\end{proof}

\begin{lem}\label{I,I'}
Let $F_1$ be the function defined by $F_1(s)=\big(\alpha_1\,s+f_1(s)\big)\,\htR(s)$ for all $s$ in $\RR^+$, where $\alpha_1$ and $f_1$ are as in Lemma \ref{pezzi}. Set
$$I(r,t)=\int_{1}^{\infty}F_1\big(2\,\arch (\cosh (r/2)v\big)\,\,\frac{\di v}{v^Q\sqrt{2v^2-2}}\,.$$
The following hold:
\begin{itemize}
\item[(i)] $|I(r,t)|\leq C\,t^{1/2}\,h_{2t}^{\RR}(r)\qquad \forall t\in\RR^+\quad\forall r\in [1,\infty)\,;$
\item[(ii)] $|I'(r,t)\leq C\,\htR(r)\qquad \forall t\in\RR^+\quad\forall r\in [1,\infty)\,.$
\end{itemize}
\end{lem}
\begin{proof}
First we observe that by Lemma \ref{pezzi} (i) and (ii),
\begin{align*} 
|F_1(s)|&\leq C\,s\,\htR(s)\\
&\leq C\,t^{1/2}h_{2t}^{\RR}(s)\,,
\end{align*}
and
\begin{align*}
|F_1'(s)|&\leq C\,(1+s/2t)\,\htR(s)\\
&\leq C\,\htR(s)\,.
\end{align*}
We now prove (i):
\begin{align*}
|I(r,t)|&\leq C\,t^{1/2}\,h_{2t}^{\RR}(r) \int_1^{\infty}\frac{\di v}{v^Q\sqrt{v^2-1}}\\
&\leq C\,t^{1/2}\,h_{2t}^{\RR}(r)\,,
\end{align*}
as required.\\
To prove (ii), note that
\begin{align*}
|I'(r,t)|&\leq\int_{1}^{\infty}|F_1'(2\,\arch (\cosh (r/2)v))|\,\,\frac{\sinh (r/2)v}{2\sqrt{\cosh ^2 (r/2)v^2-1}}\,\,\frac{\di v}{v^Q\sqrt{v^2-1}}\\
&\leq C\, \htR(r)\int_{1}^{\infty}\frac{\di v}{v^Q\sqrt{v^2-1}}\\
&\leq C\,\htR(r)\,,
\end{align*}
as required.
\end{proof}
Now we may prove the $L^1$-estimate of the gradient of the heat kernel $h_t$ given in Proposition \ref{norma}.

\begin{proof}
Let $q_t$ denote the heat kernel associated to the operator $\LQ$. By Proposition \ref{relazionenuclei}, $h_t=\du\, q_t$ so that
$$|\nabla h_t|\leq C\, \du \,(|q_t|+|\nabla q_t|)\,\qquad\forall t\in\RR^+\,.$$
It is well known (\cite[Theorem 5.9]{ADY}, \cite[Corollary 5.49]{ADY}) that $q_t$ is  radial and 
\begin{align}\label{pointwise}
|q_t(x)|&\leq  C\,t^{-1}\,(1+|x|)\,\Big(1+\frac{1+|x|}{t}\Big)^{(n-3)/2}\,\nep^{-Q\,|x|/2}\,\htR (|x|)\nonumber\\
|\nabla q_t(x)|&\leq  C\,t^{-1}\,|x|\,\Big(1+\frac{1+|x|}{t}\Big)^{(n-1)/2}\,\nep^{-Q\,|x|/2}\,\htR(|x|)\qquad\forall t\in\RR^+\quad\forall x\in S\,.
\end{align}
Our purpose is to estimate
\begin{align}\label{nabla}
\int_S |\nabla h_t| \dir&\leq C\int_S\du\, (|q_t|+|\nabla q_t|)\dir\nonumber\\
&=C\int_0^{\infty}\phi _0(r)\,(|q_t(r)|+|\nabla q_t(r)|)\,A(r)\di r\,.
\end{align}
We study the cases where $t<1$ and $t\geq 1$ separately.

{\it{Case}}~ $t<1$. In this case it suffices to use pointwise estimates (\ref{pointwise}) of $q_t$ and its gradient in (\ref{nabla}).

{\it{Case}}~$t\geq 1$. In this case, by using (\ref{pointwise}) in (\ref{nabla}),  we estimate the integral of $|\nabla h_t|$ on the unit ball. % indeed:
%\begin{align*}
%\int_{B_1}|\nabla h_t|\dir&\leq  C\int_0^1\phi_0(r)\,(|q_t(r)|+|\nabla q_t(r)|)\,A(r)\di r\\
%&\leq C\,t^{-1}\int_0^1\htR(r)\di r\\
%&\leq C\,t^{-1/2}\,.
%\end{align*}
The estimate on the complementary of the unit ball is more difficult. We already noticed (\ref{hthtR}) that
\begin{align*}
h_t(x)=\du(x)\,(\mathcal A^{-1}\circ\mathcal F^{-1})(\mathcal F \htR)(|x|)=\du(x)\, \mathcal A^{-1}(h_t^{\RR})(|x|)\,.
\end{align*}
By the inverse formula for the Abel transform (\ref{inv1}) and (\ref{inv2}) we obtain that if $m_{\zg}$ is even, then
\begin{equation}\label{centropari}
h_t(x)=C\,\du (x)\,\mathcal D_1^{m_{\zg}/2}\,\mathcal D_2^{m_{\vg}/2}(\htR)(|x|)\,,
\end{equation}
while if $m_{\zg}$ is odd, then 
\begin{align}\label{centrodispari}
h_t(x)&=C\,\du (x)\int_{|x|}^{\infty}\mathcal D_1^{(m_{\zg}+1)/2}\mathcal D_2^{m_{\vg}/2}(\htR)(s)\,\di\nu(s)\,,
\end{align}
for all $x$ in $S$. We now consider the cases where $m_{\zg}$ is even and odd separately.

{\it{Case $m_{\zg}$ even}}. By applying Lemma \ref{pezzi} (with $q=m_{\zg}/2$ and $p=m_{\vg}/2$) and formula (\ref{centropari}) we see that
\begin{align*}
h_t(x)&=C\,\du (x) \sum_{j=1}^{{(n-1)}/{2}} t^{-j}a_j(|x|)\,\htR(|x|)\\
&=C\,\du (x)\,\sum_{j=1}^{{(n-1)}/{2}} H_j(|x|,t)\,.
\end{align*}
Now we estimate the gradient of each summand $\du H_j$. We study the cases $j\geq 2$ and $j=1$ separately.\\
First suppose $j\geq 2$ and $i=0,...,n-1$:
\begin{align*}
X_i(\du H_j)(x)&=C\,\du(x)\,H_j(|x|,t)+\du(x)\,H_j'(|x|,t)\,X_i(|\cdot|)(x)\,.
\end{align*} 
Then, since $\big|X_i(|\cdot|)\big|\leq 1$ by Lemma \ref{X_i} (i), 
\begin{align*}
|X_i(\du H_j)(x)|&\leq C\,\du(x)\,\big(\big|H_j(|x|,t)\big|+\big|H_j'(|x|,t)\big|\big)\,.
\end{align*}
By applying Lemma \ref{pezzi} we obtain that
\begin{align*}
|H_j(r,t)|+|H_j'(r,t)|&\leq t^{-j}\,\Big(|a_j(r)|+|a_j'(r)|+|a_j(r)|\,\frac{r}{2t}  \Big)\htR(r)\\
&\leq C\,t^{-j}\,\Big(r^{j-1}+\frac{r^{j}}{2t}  \Big)\,r\,\nep^{-Qr/2}\,\htR(r)\\
&\leq C\,t^{-j}\,t^{(j-1)/2}\,r\,\nep^{-Qr/2}h_{2t}^{\RR}(r)\,.
\end{align*}
It follows that
\begin{align*}
|X_i(\du H_j)(x)|&\leq C\,t^{-1/2}\,\du(x)\,t^{-1}a_1(|x|)\,h_{2t}^{\RR}(|x|)\\
&\leq C\,t^{-1/2}\,h_{2t}(x)\,.
\end{align*}
Thus, 
\begin{align}\label{Int1}
\int_{B_1^c}|X_i(\du H_j)(x)|\dir (x)&\leq C\,t^{-1/2}\,\int h_{2t}(x)\dir (x)\nonumber\\
&\leq C \,t^{-1/2}\qquad \forall j\geq 2\quad i=0,...,n-1\,.
\end{align}
If $j=1$ and $i=0,...,n-1$ the estimate is more delicate. Indeed, by Lemma \ref{pezzi} (ii)
$$H_1(|x|,t)=t^{-1}\,\chQ(x)\,\big(\alpha_1\,|x|+f_1(|x|)\big)\htR(|x|)\,,$$
where $\alpha_1$ and $f_1$ are as in Lemma \ref{pezzi} (ii). Then we have that
\begin{align*}
X_i(\du H_1)(x)&=t^{-1}\,X_i\big(\du \,\chQ  \big)\,\big(\alpha_1\,|x|+f_1(|x|)\big)\\
&+t^{-1}\,\du (x)\,\chQ(x)\big(\alpha_1+f_1'(|x|)\big)\,X_i(|\cdot|)(x)\,\htR(|x|)\\
&+\big(\alpha_1\, |x|^2/{2t}+f_1(|x|)\,|x|^2/{2t}\big)\,X_i(|\cdot|)(x)\,\htR(|x|)\,.
\end{align*}
By Lemma \ref{pezzi} (i) and Lemma \ref{X_i} (ii),
\begin{align*}
|X_i(\du H_1)(x)|&\leq C\,t^{-1}\,|X_i\big(\du \chQ  \big)(x)|\,|x|\,\htR(|x|)+C\,t^{-1}\,\du (x)\,\chQ(x)\Big(1+\frac{|x|^2}{2t}\Big)\,\htR(|x|)\\
&\leq C\,t^{-1/2}\,|X_i\big(\du \chQ  \big)(x)|\,h_{2t}^{\RR}(|(x)|+C\,t^{-1}\,\du (x)\,\chQ(x)\,\htR(|x|)\,.
\end{align*}
Integrating on the complementary of the unit ball we obtain that
\begin{align}\label{J1i0}
\int_{B_1^c} |X_i(\du H_1)(x)|\dir(x) &\leq C \,t^{-1/2}\int_{B_1^c} h_{2t}^{\RR}\big(|(x)|\big)\,|X_i\big(\du \chQ  \big)(x)|\dir(x)\nonumber\\
&+C\,t^{-1}\int_{B_1^c}\du(x)\,\chQ(x)\,\htR(|x|)\dir(x)\nonumber\\
&\leq C\,t^{-1/2}\qquad i=0,...,n-1\,,
\end{align}
where we have used Lemma \ref{pezzi} (ii) and (iii).
If we sum the estimates (\ref{Int1}) and (\ref{J1i0}) we obtain that
\begin{align*}
\int_{B_1^c}|\nabla h_t(x)|\dir(x)&\leq C\,\sum_{i=0}^{n-1}\,\sum_{j=1}^{{(n-1)}/{2}}\,\int_{B_1^c}|X_i(\du H_j)(x)|\dir(x)\\
&\leq C \,t^{-1/2}\,,
\end{align*}
which concludes the proof in the case $m_{\zg}$ even.

{\it{Case~$m_{\zg}$ ~odd.}} By applying Lemma \ref{pezzi} with $q=(m_{\zg}+1)/2$ and $p={m_{\vg}}/2$ and (\ref{centrodispari}) we obtain that
\begin{align*}
h_t(x)&=C\,\du (x)\,\sum_{j=1}^{{n}/{2}}t^{-j}\int_{|x|}^{\infty}a_j(s)\,\htR(s)\di \nu(s)\\
&=C\,\du (x)\,\sum_{j=1}^{{n}/{2}}\,H_j(|x|,t)\,.
\end{align*}
We estimate the gradient of each summand $\du H_j$. \\
For all $i=0,...,n-1$ and $j\geq 1$ we see that
$$X_i(\du H_j)(x)=C\,\du (x)\,H_j(|x|,t)+C\,\du (x)\,H_j'(|x|,t)\,X_i(|\cdot|)(x)\,.$$
Then, since $\big|X_i(|\cdot|)\big|\leq 1$, 
\begin{align}\label{partenza}
|X_i(\du H_j)(x)|\leq C \,\du (x)\,\big(\big|H_j(|x|,t)\big|+\big|H_j'(|x|,t)\big|\big)\,.  
\end{align}
We study the cases where $j\geq 2$ and $j=1$ separately.\\
First suppose $j\geq 2$. By integrating by parts we obtain that
$$H_j(r,t)=-2t^{-j}\int_r^{\infty}\partial _s\big(a_j\htR\big)(s)\,\sqrt{\cosh s-\cosh r}\,\di s\,,$$
and 
$$H_j'(r,t)=-t^{-j}\int_r^{\infty}\partial _s\big(a_j\htR\big)(s)\,\frac{\sinh r\di s}{\sqrt{\cosh s-\cosh r}}\qquad\forall r\in [1,\infty)\,.$$
By Lemma \ref{pezzi} (i) and (ii)
\begin{align}\label{derivata}
\big|\partial _s\big(a_j\htR\big)(s)\big|&\leq \Big|a_j'(s)-\frac{s}{2t}\,a_j(s)\Big|\,\htR(s)\nonumber\\
&\leq C\,\Big(\frac{s^j}{2t}+s^{j-1}\Big)\,s\, \nep^{-{(Q+1)}s/{2}}\,\htR(s)\nonumber\\
&\leq C\,t^{{(j-1)}/{2}}\,s \,\nep^{-(Q+1)s/2}\,h_{2t}^{\RR}(s) \,.
\end{align}
By using (\ref{derivata}) in (\ref{partenza}) we obtain that
\begin{align*}
|X_i(\du H_j)(x)|&\leq C\,\du (x) \,t^{-(j+1)/2}\int_{|x|}^{\infty}\,s \,\nep^{-(Q+1)s/2}\,h_{2t}^{\RR}(s)\di \nu(s).
\end{align*}
Since $j\geq 2$, $t^{-j/2}\leq t^{-1}$, so that the right hand side is bounded by
\begin{align*}
C\, t^{-1/2}\,\du (x)\,t^{-1}\int_{|x|}^{\infty}a_1(s)\,h_{2t}^{\RR}(s)\di\nu (s)\,&\leq C\,t^{-1/2}\,h_{2t}(x)\,.
\end{align*}
Thus
%It is easy to verify that for all $s>r>1$
%\begin{equation}\label{sincos}
%\sqrt{\cosh s-\cosh r}+\frac{\sinh s}{\sqrt{\cosh s-\cosh r}}\leq C \frac{\sinh s}{\sqrt{\cosh s-\cosh r}}\,.
%\end{equation}
\begin{align}\label{casojmaggiore1}
\int _{B_1^c}|X_i(\du H_j)(x)|\dir (x)&\leq C \,t^{-1/2}\int_{B_1^c}h_{2t}(x)\dir(x)\nonumber\\
&\leq C\,t^{-1/2}\qquad \forall j\geq 2\quad i=0,...,n-1\,.
\end{align}
If $j=1$ the estimate is more delicate. Note that
\begin{align}\label{A11}
H_1(r,t)&=t^{-1}\int_r^{\infty}a_1(s)\,h_{t}^{\RR}(s)\di\nu(s)\,\nonumber\\
&=t^{-1}\int_r^{\infty}\cosh^{-(Q+1)}(s/2)\,\big(\alpha_1 s +f_1(s)\big)\,\htR(s)\di\nu(s)\nonumber\\
&=t^{-1}\int_r^{\infty}\cosh^{-(Q+1)}(s/2)\,F_1(s)\di\nu(s)\,\,,
\end{align}
where $F_1(s)=\big(\alpha_1 s +f_1(s)\big)\,\htR(s)$ and $\alpha_1$, $f_1$ are as in Lemma \ref{pezzi} (ii).\\
By changing variables $\Big(u=\cosh(s/2)\cosh ^{-1}(r/2)\Big)$ the integral (\ref{A11}) transforms into
\begin{align*}
&t^{-1}\cosh^{-Q}(r/2)\int_{1}^{\infty}F_1(2\,\arch (\cosh (r/2)v)\,\,\frac{\di v}{v^Q\sqrt{2v^2-2}}\,.
\end{align*}
Set $I(r,t)=\int_{1}^{\infty}F_1\big(2\,\arch (\cosh (r/2)v\big)\,\,\frac{\di v}{v^Q\sqrt{2v^2-2}}$. Thus 
$$H_1(r,t)=t^{-1}\,\cosh^{-Q}(r/2)\,I(r,t)\,.$$
Then, for all $i=0,...,n-1$, 
\begin{align*}
X_i(\du H_1)(x)&=t^{-1}\,X_i(\du\,\chQ)(x)\,I(|x|,t)+t^{-1}\,\du(x)\,\chQ(x)\,I'(|x|,t)\,X_i(|\cdot|)(x)\,.
\end{align*}
{F}rom Lemma \ref{pezzi} (i) and Lemma \ref{I,I'} we obtain that
\begin{align*}
|X_i(\du H_1)(x)|&\leq t^{-1}\big| X_i(\du\,\chQ)(x)\big| \,\big|I(|x|,t)\big|+t^{-1}\,\du\,\chQ(x)\,\big|I'(|x|,t)\big|\\
&\leq C\,t^{-1/2}\,h_{2t}^{\RR}(|x|)\,\big| X_i(\du\chQ)(x)\big| +C\,t^{-1}\,\du\,\chQ(x)\,\htR(|x|)\,.
\end{align*}
We now integrate the last expression on the complementary of the unit ball and apply Lemma \ref{pezzi} (ii) and (iii):
\begin{align}\label{CISONO}
\int_{B_1^c}|X_i(\du H_1)|\dir\leq \,&\,C \,t^{-1/2}\int_S h_{2t}^{\RR}(|(x)|)\,\big| X_i(\du\chQ)(x)\big|\dir(x)\nonumber\\
&+C\,t^{-1}\int_{B_1^c}\,\du(x)\,\chQ(x) \,\htR(|x|)\dir(x)\nonumber\\
\leq \,&C\,t^{-1/2}\,\qquad \forall i=0,...,n-1\,.
\end{align}
We put (\ref{casojmaggiore1}) and (\ref{CISONO}) together and conclude that
\begin{align*}
\int_{B_1^c}|\nabla h_t(x)|\dir(x)&\leq C\,\sum_{i=0}^{n-1}\,\sum_{j=1}^{{n}/{2}}\int_{B_1^c}|X_i(\du H_j)(x)|\dir(x)\\
&\leq C \,t^{-1/2}\,,
\end{align*}
as required.

This concludes the proof of Proposition \ref{norma}.
\end{proof}
The estimate (\ref{stimaB}) follows immediately from Propositions \ref{stima2} and \ref{norma}.

\bigskip
In the particular case of $ax+b\,$-groups we deduce from (\ref{stimaA}) and (\ref{stimaB}) a boundedness theorem for multipliers of $\D$ from $H^{1}$ to $L^1(\rho)$ and from $BMO$ to $L^{\infty}(\rho)$. 
\begin{teo}\label{moltiplicatoriH1}
Let $S=\RR^d\times \RR^+$ be an $ax+b\,$-group. Suppose that $s_0>\frac{3}{2}$, and $s_{\infty}>\max\left\{\frac{3}{2},\frac{d+1}{2}\right\}$. If $M$ satisfies a mixed \MH condition of order $(s_0,s_{\infty})$, then $M(\D)$ is bounded from $H^{1}$ to $L^1(\rho)$ and from $BMO$ to $L^{\infty}(\rho)$.
\end{teo}
\begin{proof}
It follows from Theorem \ref{TeolimH1} and Corollary \ref{TeolimBMO}.
\end{proof}

\section{The product case}\label{multipliersp}
Let $S=S'\times S''$ be the product of two \DR spaces. Let $E_0',...,E'_{n'-1}$ and $E_0'',...,E''_{n''-1}$ be orthonormal basis of $\s'$ and $\s''$, respectively. We denote by $X_i'$ and $X_j''$ the left invariant vector fields on $S'$ and $S''$, which agree with $E_i'$ and $E_j''$ at the identity. We now consider the left invariant vector fields on $S$ defined by
\begin{align*}
(X_i'f)(x',x'')&=\big(X_i'f(\cdot\,,x'')\big)(x')\,,\\
(X_j''f)(x',x'')&=\big(X_j''f(x',\,\cdot)\big)(x'')\,\qquad \forall f\in C^{\infty}_c(S)\,.
\end{align*}
We denote by $\D$ the Laplacian $\D=-\sum_{i=0}^{n'-1}X_i^{'2} -\sum_{j=0}^{n''-1}X_j^{''2}$. By (\ref{relationship}) 
\begin{align*}
\D f(x',x'')&=\big(\D 'f(\cdot\,,x''\big)(x')+\big(\D ''f(x',\,\cdot\big)(x'')\\&=\delta '^{1/2}(x')\big((\mathcal L'-Q'^2/4)(\delta '^{-1/2}f(\cdot\,,x'')\big)(x')\\
&+\delta ''^{1/2}(x')\big((\mathcal L''-Q''^2/4)(\delta ''^{-1/2}f(x',\,\cdot)\big)(x'')\\
&=\delta '^{1/2}(x')\delta ''^{1/2}(x'')\big((\mathcal L'-Q'^2/4)(\delta '^{-1/2}\delta ''^{-1/2}f(\cdot\,,x'')\big)(x')\\
&+\delta '^{1/2}(x')\delta ''^{1/2}(x'')\big((\mathcal L''-Q''^2/4)(\delta '^{-1/2}\delta ''^{-1/2}f(x',\,\cdot)\big)(x'')\\
&=\delta^{1/2}\,\big(\mathcal L-(Q'^2+Q''^2)/4\big)\,\delta^{-1/2}\,f\,(x',x'')\,.
\end{align*}
Thus, $\D=\delta^{1/2}\,\LQ\,\delta^{-1/2}$, where $\LQ=\mathcal L-(Q'^2+Q''^2)/4$. 

Let 
$$\LB _Q=\int_0^{+\infty}t\, dE_{\LB _Q}(t){\hspace{1cm}}{\rm and}{\hspace{1cm}}\Delta=\int_0^{+\infty}t\, dE_{\Delta}(t)$$
be the spectral resolutions of $\LQ$ and $\Delta$ respectively. For all bounded measurable function $M$ on $[0,+\infty)$ the operators
$$M(\LB _Q)=\int_0^{+\infty}M(t) dE_{\LB _Q}(t){\hspace{1cm}}{\rm and}{\hspace{1cm}}M(\Delta)=\int_0^{+\infty}M(t) dE_{\Delta}(t)$$
are bounded on $L^2(\lambda)$ and $L^2(\rho)$ respectively. By the spectral theorem we have that
$$\delta^{-1/2}M(\Delta)\delta^{1/2}f=M(\mathcal L _Q)\,.$$
We denote by $\kD$ and $\kQ$ the distributional kernels of $M(\D)$ and $M(\LQ)$ respectively; we have that 
$$\MQ f=f\ast\kQ{\hspace{1cm}}{\rm {and}}{\hspace{1cm}}\MD f=f\ast\kD\qquad\forall f\in C^{\infty}_c(S)\,,$$
where $\ast$ denotes the convolution on $S$. It is easy to check that $\kD=\du\kQ$. Moreover the spherical transform of $k_{\LQ}$ is
$$\mathcal H k_{\LQ}(s',s'')=M(s^{'2}+s ^{''2})\qquad\forall s',s'' \in\RR^+\,.$$
Let $M$ be a bounded measurable function on $\RR^+$. Let $K_{M(\D)}$ and $k_{M(\D)}$ denote the integral kernel and the convolution kernel of the operator $M(\D)$ respectively. Our aim is to find sufficient conditions on $M$ that ensure the boundedness of the operator $M(\D)$ from $L^1(\rho)$ to $\lorentz{1}{\infty}{\rho}$ and on $L^p(\rho)$, for $1<p<\infty$. 
\begin{teo}\label{moltiplicatorip}
Let $S=S'\times S''$ be the product of two \DR spaces. Suppose that $s_0>3$ and $s_{\infty}>\max\left\{{3},\frac{n}{2}\right\}$, where $n$ denotes the dimension of $S$. If $M$ satisfies a mixed \MH condition of order $(s_0,s_{\infty})$, then $\MD$ is bounded from $\lu{\rho}$ to $\lorentz{1}{\infty}{\rho}$ and on $L^p(\rho)$, for all $p$ in $ (1,\infty)$.
\end{teo}
{\bf{Structure of the proof.}} 
The proof of this theorem follows the same line of the proof of Theorem \ref{moltiplicatori}. Let $\varepsilon$ be such that $s_0>3+\varepsilon$ and $s_{\infty}>\max\left\{{3},\frac{n}{2}\right\}+\varepsilon$.

{\bf{Step 1.}} Suppose that $m$ is in $H^{s_0}(\RR)\cap H^{s_{\infty}}(\RR)$ and is supported in $[1/4,4]$. Then the integral kernel $K_{m(t\D)}$ satisfies the following estimate:
\begin{align}\label{stimaAp}
\int_S|K_{m(t\D)}(x,y)|\,\big(1+t^{-1/2}d(x,y)\big)^{\varepsilon}\dir(x)&\leq  \begin{cases}
C\,\|m\|_{H^{s_0}(\RR)}& \forall t\in [1,\infty)\\
C\,\|m\|_{H^{s_{\infty}}(\RR)}& \forall t\in (0,1)\,\qquad \forall y\in S\,.
\end{cases}
\end{align}

{\bf{Step 2.}} Let $m$ and $K_{m(t\D)}$ be as in Step 1. Then
\begin{align}\label{stimaBp}
\int_S|K_{m(t\D)}(x,y)-K_{m(t\D)}(x,z)|\dir(x)\nonumber\\
\leq& \begin{cases}
C\,t^{-1/2}d(y,z)\,\|m\|_{H^{s_0}(\RR)}& \forall t\in [1,\infty)\\
C\,t^{-1/2}d(y,z)\,\|m\|_{H^{s_{\infty}}(\RR)}& \forall t\in (0,1)\,\qquad\forall y,z\in S\,.
\end{cases}
\end{align}

{\bf{Step 3.}} This is verbatim the same as Step 3 in the proof of Theorem \ref{moltiplicatori} and therefore is omitted.

In the following subsections we go into details of Step 1 and Step 2.

\subsection{Step 1}
The weight function $w$ defined by
$$w(x',x'')=(w'\otimes w'')(x',x'')=\delta^{-1/2} (x',x'')\,\nep^{{[Q'\,d(x',e')+Q''\,d(x'',e'')]}/2}\qquad\forall (x',x'')\in S\,,$$
will play an important r\^ole in the sequel. 
\begin{lem}\label{pesop}
There exists a constant $C$ such that the following hold:
\begin{itemize}
\item[(i)]$\int_{B_r}w^{-1}\,{\rm{d}}\rho\leq \begin{cases}
C\,r^4& \forall r\in [1,\infty)\\
C\,r^n& \forall r\in (0,1)\,;
\end{cases}$
\item[(ii)] for every compactly supported function  $f$ on $\RR^+$
\begin{align*}
\int_{B_r}|k_{f(\D)}|^2\,w\,{\rm{d}}\rho&\leq C\, (1+r)^2\, \int_{B_r}|k_{f(\D)}|^2\di\rho\,.
\end{align*}
\end{itemize}
\end{lem}
\begin{proof}
To prove (i) we note that 
\begin{align*}
\int_{B_r}w^{-1}\dir&=\int_{B_r'}w'(x')\dir '(x')\cdot\int_{B_r''}w''(x'')\dir ''(x'')\,,
\end{align*}
and apply lemma \ref{peso}.

To prove (ii) let $f$ be compactly supported on $\RR^+$  and denote by $k_{f(\LQ)}$ the convolution kernel of the operator $f(\LQ)$. We know that $k_{f(\D)}=\du \,k_{f(\LQ)}$. Now define $A_{R'}=\{x'\in S':~R'-1<d(x',e)<R'\,\}$, $A_{R''}=\{x''\in S'':~R''-1<d(x'',e)<R''\,\}$, with $R',\,R''=1,...,[r]$ and $C_r'=\{x'\in S':~[r]<d(x',e)<r\,\}$ and $C_r''=\{x''\in S'':~[r]<d(x'',e)<r\,\}$. We split up the ball $B_r$ in $A_{R'}\times A_{R''}$, $A_{R'}\times C_r''$, $C_r'\times A_{R''}$ and $C_r'\times C_r''$ and estimate the integral of $|k|^2\,w$. By Lemma \ref{intdufp} 
\begin{align*}
&\int_{A_{R'}\times A_{R''}}|k_{f(\D)}|^2\,w\,\dir\\
=&\int_{A_{R'}\times A_{R''}}|\delta (x',x'')\,k_{f(\LQ)}(x',x'')|^2\,\dum (x',x'')\,\nep^{{Q'd(x',e)+Q''d(x'',e)}/2}\dir(x',x'')\\
=&\int_{R'-1}^{R'}\int_{R''-1}^{R''}\phi '_0(r')\,\phi ''_0 (r'')\,|k_{f(\LQ)}(r',r'')|^2\,\nep^{Q'\,r'/2}\,\nep^{Q''\,r''/2}\,A'(r')\, A''(r'')\di r'\di r''\,,
\end{align*}
which, by (\ref{stimafi0}), is bounded above by
\begin{align*}
\leq C\,&\int_{R'-1}^{R'} \int_{R''-1}^{R''}(1+r')\,(1+r'')\,|k_{f(\LQ)}(r',r'')|^2\, A'\otimes A''(r',r'')\di r'\di r''\\
\leq C \,(1+r)^2 &\int_{A_{R'}\times A_{R''}}|k_{f(\LQ)}|^2\dil\\
=C\,(1+r)^2&\int_{A_{R'}\times A_{R''}}|k_{f(\D)}|^2\dir\,.
\end{align*}
The proof of the estimates on $A_{R'}\times C_r''$, $C_r'\times A_{R''}$ and $C_r'\times C_r''$ is similar and is omitted. The proof of (ii) is complete.
\end{proof}

Our purpose is to prove an $L^1$-estimate for the convolution kernel of a multiplier of the Laplacian $\D$. The following result gives an estimate of this type in a particular case.
\begin{lem}\label{normaL1}
Let $f$ be an even function on $\RR$ such that its Fourier transform $\hat{f}$ is supported in $[-r,r]$. Then $k_{f(\sqrt{\D})}$ satisfies the following estimate:
\begin{align*}
&\int_S|k_{f(\sqrt{\D})}|\dir\\
\leq\,& \begin{cases} 
C\,r^{n/2}\big(\int_0^{\infty}\int_0^{\infty}|f(\sqrt{s^{'\,2}+s^{''\,2}})|^2\,(s^{'\,2}+s^{'\,(n'-1)})\,(s^{''\,2}+s^{''\,(n'-1)})\di s'\di s''\big)^{1/2}~~ \forall r<1&{}\\
C\,r^{3}\big(\int_0^{\infty}\int_0^{\infty}|f(\sqrt{s^{'\,2}+s^{''\,2}})|^2\, (s^{'\,2}+s^{'\,(n'-1)})\,(s^{''\,2}+s^{''\,(n'-1)})   \di s'\di s''\big)^{1/2}~ ~ ~ ~\forall r\geq 1.& {}
\end{cases}
\end{align*}
\end{lem}
\begin{proof}
Let $k_{f(\sqrt{\LQ})}$ denote the convolution kernel of the operator $f(\sqrt{\LQ})$. We have that $k_{f(\sqrt{\D})}=\du\, k_{f(\sqrt{\LQ})}$ and $\mathcal Hk_{f(\sqrt{\LQ})}(s',s'')=f(\sqrt{s'^2+s''^2})$ for all $(s',s'')\in \RR^2_+\,.$ The hypothesis that the Fourier transform of $f$ is contained in $[-r,r]$ implies that $k_{f\left(\sqrt{\LQ}\right)}$ is supported in the ball $B_r$. 

By H\"older's inequality we obtain that
\begin{align}\label{holderp}
\int_S|k_{f(\sqrt{\D})}|\dir&=\int_{B_r}|k_{f(\sqrt{\D})}|\,w^{1/2}\,w^{-1/2}\dir\nonumber\\
&\leq \Big( \int_{B_r}w^{-1}\dir \Big)^{1/2}\Big( \int_{B_r}|k_{f(\sqrt{\D})}|^2\,w\dir   \Big)^{1/2}\,.
\end{align}
By Lemma \ref{pesop} (ii) we have that
\begin{align*}
&\Big( \int_{B_r}|k_{f(\sqrt{\D})}|^2\,w\dir   \Big)^{1/2}\\
\leq&\,  C\, (1+r)\,\Big( \int_{B_r}|k_{f(\sqrt{\D})}|^2\dir  \Big)^{1/2}\\
=&\, C \,(1+r)\,\Big( \int_{B_r}|k_{f(\sqrt{\LQ})}|^2\dil   \Big)^{1/2}\\
=&\, C \,(1+r)\,\Big(\int_0^{\infty}\int_0^{\infty}|f(\sqrt{s'^2+s''^2})|^2|{\bf{c}}(s'^2)|^{-2}\,{\bf{c}}(s''^2)|^{-2}\di s'\di s''\Big)^{1/2}\\
\leq&\, C\,(1+r)\, \Big( \int_0^{\infty} \int_0^{\infty}|f(\sqrt{s'^2+s''^2})|^2(s'^2+s'^{n'-1})\,(s''^2+s''^{n''-1})\di s'\di s''\Big)^{1/2}                 \,.
\end{align*}
Here we have used Plancherel formula and estimate (\ref{HC}) for the Plancherel
measure. Thus, by (\ref{holderp}) and 
Lemma \ref{pesop} (i) we deduce that
\begin{align*}
&\int_S|k_{f(\sqrt{\D})}|\dir\\
\leq\,& \begin{cases} 
C\,r^{n/2}\big(\int_0^{\infty}\int_0^{\infty}|f(\sqrt{s^{'\,2}+s^{''\,2}})|^2\,(s^{'\,2}+s^{'\,(n'-1)})\,(s^{''\,2}+s^{''\,(n'-1)})\di s'\di s''\big)^{1/2}~~ \forall r<1&{}\\
C\,r^{3}\big(\int_0^{\infty}\int_0^{\infty}|f(\sqrt{s^{'\,2}+s^{''\,2}})|^2\, (s^{'\,2}+s^{'\,(n'-1)})\,(s^{''\,2}+s^{''\,(n'-1)})   \di s'\di s''\big)^{1/2}~ ~ ~ ~\forall r\geq 1\,& {}
\end{cases}
\end{align*}
as required.
\end{proof}
The previous result gives a good estimate for the $L^1$-norm of the kernel of a multiplier whose Fourier transform has compact support. Next lemma shows that every function $f$ supported in $[1/2,2]$ may be written as sum of functions whose Fourier transform has compact support.
\begin{lem}\label{decompp}
Let $q',q'',Q',Q''$ be real numbers such that $0<q'\leq Q'$ and $0<q''\leq Q''$ and $f$ be a function in $H^s(\RR)$ supported in $[1/2,2]$. Let $f_{\ell}$ and $f_{\ell,\,t}$ be as in Lemma \ref{decomp} (with $q=q'+q''$, $Q=Q'+Q''$).
Then
\begin{itemize}
\item[(i)]$\int_0^{\infty}\int_0^{\infty}|f_{\ell}(\sqrt{s'^2+s''^2})|^2\,(s'^{2q'}+s'^{2Q'})\,(s''^{2q''}+s''^{2Q''})\,{\rm{d}} s'\di s''\leq C\, 2^{-2sl}\,\|f\|_{H^s(\RR)}\,.$
\item[(ii)] 
\begin{align*}
&\int_0^{\infty}\int_0^{\infty}|f_{\ell,t}(\sqrt{s'^2+s''^2)})|^2\,(s'^{2q'}+s'^{2Q'})\,(s''^{2q''}+s''^{2Q''})\,{\rm{d}} s'\di s''\\
\leq& \begin{cases}
C\,t^{-(2q'+2q''+2)}\, 2^{-2s\ell}\,\|f\|_{H^s(\RR)}& \forall t\in [1,\infty)\\
C\,t^{-(2Q'+2Q''+2)} \,2^{-2s\ell}\,\|f\|_{H^s(\RR)}& \forall t\in (0,1).
\end{cases}
\end{align*}
\end{itemize}
\end{lem}
\begin{proof}
By \cite[Lemma 1.3]{HEB}, $\|f_{\ell}\|_{L^2(\RR)}\leq C\,2^{-2s\,\ell}\,\|f\|_{H^s(\RR)}$ and
\begin{equation}\label{stimapuntuale}
|f_{\ell}(x)|\leq C\,\|f\|_{L^2(\RR)}\,\frac{2^{\ell}}{[2^{\ell}(|x|-2)]^{Q'+Q''+s+1}}\,,
\end{equation}
for all $x$ such that $|x|>4$.

We now prove (i) by estimating the integrals on the set $\{(s',s''):~ |s'|<4,\,|s''|<4\}$ and on its complementary separately.
\begin{align}\label{onB}
\int_0^{4}\int_0^{4}|f_{\ell}(\sqrt{s'^2+s''^2})|^2\,(s'^{2q'}+s'^{2Q'})\,(s''^{2q''}+s''^{2Q''})\,{\rm{d}} s'\di s''&\leq C\,\int_0^4|f_{\ell}(r)|^2\,r\di r\nonumber\\
&\leq C\, \|f_{\ell}\|_{L^2(\RR)}\nonumber\\
&\leq C\,2^{-2s\,\ell}\,\|f\|_{H^s(\RR)}\,.
\end{align}
The integral on the complementary of this set is estimated by using (\ref{stimapuntuale}):
\begin{align}\label{offB}
&\int_4^{\infty}\int_4^{\infty}|f_{\ell}(\sqrt{s'^2+s''^2})|^2\,(s'^{2q'}+s'^{2Q'})\,(s''^{2q''}+s''^{2Q''})\,{\rm{d}} s'\di s''\\
\leq&\, \int_4^{\infty}\int_4^{\infty}|f_{\ell}(\sqrt{s'^2+s''^2})|^2\,\big(1+|(s',s'')|^2\big)^{Q'+Q''})\,{\rm{d}} s'\di s'\nonumber\\
\leq &\, C\,\int_4^{\infty}|f_{\ell}(r)|^2\,(1+r^2)^{Q'+Q''}\,r\di r\nonumber\\
\leq &\,C\, \|f_{\ell}\|^2_{L^2(\RR)}\,\int_4^{\infty}\frac{2^{\ell}}{[2^{\ell}(r-2)]^{2Q'+2Q''+2s+2}}\,(1+r^2)^{Q'+Q''}\,r\di r\nonumber\\
\leq&\, C\,2^{-2s\,\ell}\,\|f\|_{H^s(\RR)}\int_4^{\infty}\frac{1}{(r-2)^{2Q'+2Q''+2s+2}}\,(1+r^2)^{Q'+Q''}\,r\di r\nonumber\\
\leq& \, C\,2^{-2s\,\ell}\,\|f\|_{H^s(\RR)}\,.
\end{align}
By combining (\ref{onB}) and (\ref{offB}) we obtain (i).

Since $f_{\ell, \,t}(\cdot)=f_{\ell}(t\,\cdot)$, (ii) follows easily by (i) and a simple change of variables. 
\end{proof}
We now prove estimate (\ref{stimaBp}).

\begin{proof}
First we observe that
\begin{align*}
&\int _S |K_{m(t\D)}(x,y)|\,\big(1+t^{-1/2}d(x,y)\big)^{\varepsilon}\dir(x)\\
=&\int_S|k_{m(t\D)}(x)|\,\big(1+t^{-1/2}d(x,e)\big)^{\varepsilon}\dir(x)\qquad\forall y\in S\,.
%&=\int_S |k_{m(t\D)}(y^{-1}x)|\,\delta(y)\,\big(1+t^{-1/2}d(y^{-1}x,e)\big)^{\varepsilon}\,\delta ^{-1}(x)\dil(x)\\
%=&\int_S|k_{m(t\D)}(x)|(1+t^{-1/2}d(x,e))^{\varepsilon}\delta^{-1}(x)d_lx\nonumber\\
%&=\int_S|k_{m(t\D)}(x)|\,\big(1+t^{-1/2}d(x,e)\big)^{\varepsilon}\dir(x)\qquad\forall y\in S\,.
\end{align*}
Now it suffices to define $f(s)=m(s^2)$ for all $s\in \RR^+$. The function $f$ is supported in $[1/2,2]$ and the operator $m(t\D)$ agrees with the operator $f(t^{1/2}\sqrt{\D})$ by spectral theory. By applying Lemma \ref{decompp} with $q'=q''=1$, $Q'={(n'-1)}/{2}$ and $Q''={(n''-1)}/2$ we find functions $f_{\ell,t^{1/2}}$ such that $f(t^{1/2}\cdot)=\sum_{\ell} f_{\ell,t^{1/2}}(\cdot)$ and ${\rm{supp}}(\hat{f}_{\ell,t^{1/2}})\subset [-2^{\ell}t^{1/2}, 2^{\ell}t^{1/2}]$. So we can apply Lemma \ref{normaL1} to each function $f_{\ell,t^{1/2}}$ and sum these esti\-ma\-tes up. We treat the cases $t\geq 1$ and $t<1$ separately.

{\it{Case}}~$t< 1$. In this case the quantity $2^{\ell}t^{1/2}$ is $\geq 1 $ for $\ell\geq (1/2)\,\log (1/t)$ and $<1$ otherwise. By applying again Lemma \ref{normaL1} we have that
\begin{align*}
&\int_S|k_{f_{\ell,t^{1/2}}(\sqrt{\D})}(x)|\,\big(1+t^{-1/2}d(x,e)\big)^{\varepsilon} \dir(x)\\
=\,& \int_{B(e,2^{\ell}t^{1/2})}|k_{f_{\ell,t^{1/2}}(\sqrt{\D})}(x)|\,\big(1+t^{-1/2}d(x,e)\big)^{\varepsilon} \dir(x)\\
\leq\,& C\,(1+t^{-1/2}\,2^{\ell}\,t^{1/2})^{\varepsilon}\,(2^{\ell}t^{1/2})^{\max \{3,n/2\}}\,\cdot\\
&\cdot\Big(\int_0^{\infty}\int_0^{\infty}|f_{\ell,t^{1/2}}(\sqrt{s'^2+s''^2})|^2\,(s'^2+s'^{n'-1})\,(s''^2+s''^{n'-1})\di s'\di s''\Big)^{1/2}\,,
\end{align*}
which, by Lemma \ref{decomp} (iii'), is bounded above by
\begin{align*}
& C\,2^{\ell\varepsilon}\,(2^{\ell}t^{1/2})^{\max \{3,n/2\}}\,t^{-n/2}\,2^{-s_{\infty}\ell}\,\|f\|_{H^{s_{\infty}}(\RR)}\\
\leq\,  &C\,2^{\ell(\max\{3,n/2\}+\varepsilon-s_{\infty})}\,\|f\|_{H^{s_{\infty}}(\RR)}\,.
\end{align*}
Now we sum over $\ell\geq 0$ and obtain that
\begin{align*}
\int_S  |k_{m(t\D)}(x)|\,\big(1+t^{-1/2}d(x,e)\big)^{\varepsilon}\dir(x)&=\int_S |k_{f(t^{1/2}\sqrt{\D})}(x)|\,\big(1+t^{-1/2}d(x,e)\big)^{\varepsilon}\dir(x)\\&\leq C \,\|f\|_{H^{s_{\infty}}(\RR)}\\
&=C\,\|m(\cdot^2 )\|_{H^{s_{\infty}}(\RR)}\\
&\leq C\,\|m\|_{H^{s_{\infty}}(\RR)} \qquad \forall t\in (0,1)\,,
\end{align*}
since $s_{\infty}>\max \{3,n/2\}+\varepsilon$.

{\it{Case}} ~$t\geq 1$. In the same way as above by applying Lemma \ref{normaL1} we obtain that
\begin{align*}
\int_S  |k_{m(t\D)}(x)|\,\big(1+t^{-1/2}d(x,e)\big)^{\varepsilon}\dir(x)&\leq C\,\|m\|_{H^{s_0}(\RR)}\,,
\end{align*}
as required.
\end{proof}

\subsection{Step 2}
By proceeding exactly as in Subsection \ref{stima2section}, we see that to prove (\ref{stimaBp}) it suffices to prove the following $L^1$-estimate of the gradient of the heat kernel $h_t$. 
\begin{prop}\label{normap}
There exists a constant $C$ such that 
\begin{align*}
\int_S |\nabla h_t| \dir&\leq C \,t^{-1/2}\qquad \forall t\in\RR^+ \,.
\end{align*}
\end{prop}
\begin{proof}
It is easy to check that $h_t=h_t'\otimes h_t''$, where $h_t'$ and $h_t''$ denote the heat kernel on $S'$ and $S''$ respectively. Thus,
\begin{align*}
X_i'h_t&=(X_i'h_t')\otimes h_t''\,,\\
X_i''h_t&=h_t'\otimes (X_i''h_t'')\,.
\end{align*}
Now we integrate on $S$ and apply Proposition \ref{norma} to $h_t'$ and $h_t''$:
\begin{align*}
\int_S|X_i'h_t|\dir&\leq\int_S'|X_i'h_t'|\dir '\,\int_S''|h_t''|\dir ''\\
&\leq C\,t^{-1/2}\,;\\
\int_S|X_i''h_t|\dir&\leq\int_S'|h_t'|\dir '\,\int_S''|X_i''h_t''|\dir ''\\
&\leq C\,t^{-1/2}\,,
\end{align*}
as required.
\end{proof}

%faccio la bibliografia

\end{document}